%% file: main.tex
\pdfoutput=1
\documentclass[a4paper,reqno]{amsart}
\usepackage[margin=3cm]{geometry}
\usepackage[T1]{fontenc}
\input{core}
\input{core-ftrm}
\bibliography{references}

\newcommand{\clps}[1]{\textup{\guillemotleft} #1 \textup{\guillemotright}}
\newcommand{\M}{{\b M}}

\newcommand{\Cart}{\b{Cart}}
\newcommand{\F}{{\bb F}}
\newcommand{\flatw}{^{{+}}}
\newcommand{\VDbl}{{\b{VDbl}}}
\newcommand{\ssa}{structure--semantics adjunction}
\newcommand{\LMndX}{{\L\dc{Mnd}(\X)}}

\newtheorem{manualtheoreminner}{Theorem}
\newenvironment{manualtheorem}[1]{%
  \renewcommand\themanualtheoreminner{#1}%
  \manualtheoreminner
}{\endmanualtheoreminner}

\title{The nerve theorem for relative monads}
\author{Nathanael Arkor}
\address{Department of Software Science, Tallinn University of Technology, Estonia}
\author{Dylan McDermott}
\address{Department of Computer Science, Reykjavik University, Iceland}
\datetoday

\keywords{Relative monad, relative adjunction, monadicity, virtual double category, virtual equipment, formal category theory, enriched category theory}
\subjclass[2020]{18D70,18D65,18C15,18C20,18A40,18D60,18C10,18D20,18N10}

\hypersetup{pdfauthor={Nathanael Arkor, Dylan McDermott}}

\begin{document}

\begin{abstract}
    A fundamental result in the theory of monads is the characterisation of the category of algebras for a monad in terms of a pullback of the category of presheaves on the category of free algebras: intuitively, this expresses that every algebra is a colimit of free algebras. We establish an analogous result for enriched relative monads with dense roots, and explain how it generalises the nerve theorems for monads with arities and nervous monads. As an application, we derive sufficient conditions for the existence of algebraic colimits of relative monads. More generally, we establish such a characterisation of the category of algebras in the context of an exact virtual equipment. In doing so, we are led to study the relationship between a $j$-relative monad $T$ and its associated loose-monad $E(j, T)$, and consequently show that the opalgebra object and the algebra object for $T$ may be constructed from certain double categorical limits and colimits associated to $E(j, T)$.
\end{abstract}

\dedicatory{In memory of F.\ William Lawvere, whose deep understanding of the nature of algebra has shaped our own.}

\maketitle

\setcounter{tocdepth}{1}
\tableofcontents

\section{Introduction}

A monad may be viewed as an abstraction of the specification of an algebraic structure by operators and equations~\cite{lawvere1963functorial,linton1966some,linton1966triples}. Consequently, an algebra for a monad is an abstraction of the classical notion of an algebra qua a set equipped with algebraic structure. Classically, every algebra may be presented as a quotient of a free algebra. The same is true for algebras for monads: given a monad $T = (t, \mu, \eta)$ and a $T$-algebra $(a, \alpha)$, the following diagram forms a coequaliser in the category of $T$-algebras~\cite{beck1966untitled}.
\begin{equation}
\label{coequaliser}
\begin{tikzcd}
	tta & ta & a
	\arrow["\alpha", from=1-2, to=1-3]
	\arrow["t\alpha", shift left, from=1-1, to=1-2]
	\arrow["{\mu_a}"', shift right, from=1-1, to=1-2]
\end{tikzcd}
\end{equation}
Consequently, every $T$-algebra is a colimit of free $T$-algebras. This observation may be refined by considering not just an individual $T$-algebra, but the entire category of $T$-algebras. Denote by $u_T \colon \Alg(T) \to A$ the forgetful functor from the category of algebras; by $k_T \colon A \to \Kl(T)$ the inclusion functor into the category of free algebras (\ie{} the Kleisli category); and by ${\yo_A \colon A \to [A\op, \Set]}$ the Yoneda embedding. Then the following diagram forms a pullback of categories~\cite[Observation~1.1]{linton1969outline}.
\begin{equation}
\label{yo-pullback}
\begin{tikzcd}
	{\Alg(T)} & {[\Kl(T)\op, \Set]} \\
	A & {[A\op, \Set]}
	\arrow["{u_T}"', from=1-1, to=2-1]
	\arrow["{[k_T\op, \Set]}", from=1-2, to=2-2]
	\arrow[""{name=0, anchor=center, inner sep=0}, "{\yo_A}"', hook, from=2-1, to=2-2]
	\arrow[hook, from=1-1, to=1-2]
	\arrow["\lrcorner"{anchor=center, pos=0.125}, draw=none, from=1-1, to=0]
\end{tikzcd}
\end{equation}
Assuming that $A$ is small, the presheaf category $[\Kl(T)\op, \Set]$ is the free cocompletion of the category of free $T$-algebras under all small colimits. The pullback above exhibits the category of algebras as a full subcategory of the presheaf category, which is furthermore closed under representables, since every free algebra is, in particular, an algebra.
Consequently, the category of algebras may be seen as a cocompletion of the category of free algebras under a class of small colimits~\cite[Theorem~5.19]{kelly1982basic}.

The primary purpose of this paper is to generalise this result to \emph{relative monads}, which generalise monads by permitting the underlying functors to be arbitrary functors, rather than endofunctors~\cite{altenkirch2010monads,altenkirch2015monads}. Specifically, given a dense functor $\jAE$, we prove that, for a $j$-relative monad $T$, the following diagram forms a pullback of categories, where we denote by $n_j \colon E \to [A\op, \Set]$ the nerve of $j$, which sends an object $e \in \ob E$ to the restricted Yoneda embedding~$a \mapsto E(j a, e)$.
\begin{equation}
\label{n_j-pullback}
\begin{tikzcd}
    {\Alg(T)} & {[\Kl(T)\op, \Set]} \\
    E & {[A\op, \Set]}
    \arrow["{u_T}"', from=1-1, to=2-1]
    \arrow["{[k_T\op, \Set]}", from=1-2, to=2-2]
    \arrow[""{name=0, anchor=center, inner sep=0}, "{n_j}"', hook, from=2-1, to=2-2]
    \arrow[hook, from=1-1, to=1-2]
    \arrow["\lrcorner"{anchor=center, pos=0.125}, draw=none, from=1-1, to=0]
\end{tikzcd}
\end{equation}
Furthermore, the unlabelled functor above is isomorphic to the nerve of the comparison functor $i_T \colon \Kl(T) \to \Alg(T)$ and consequently exhibits the $T$-algebras as certain presheaves on the category of free $T$-algebras: namely, those presheaves in the essential image of the nerve $n_{i_T} \colon \Alg(T) \to [\Kl(T)\op, \Set]$. We therefore call this the \emph{nerve theorem} for relative monads. Such a theorem is of interest for several reasons. From a purely conceptual point of view, it establishes that (under the mild assumption of density of the root $j$) algebras for relative monads are colimits of free algebras. This is deeper an observation than might appear at first, since algebras for relative monads do not admit a coequaliser presentation\footnotemark{} like that for monads~\eqref{coequaliser}, as it is not possible to iterate the underlying functor $t \colon A \to E$.
\footnotetext{A similar obstruction arises when one attempts to generalise the classical monadicity theorem~\cite{beck1966untitled}, which is formulated in terms of coequalisers, to relative monads. However, there, too, it is possible to entirely avoid consideration of coequalisers~\cite{arkor2024relative} (\cf{}~\cite{pare1971absolute}).}%
From a practical point of view, it facilitates simple proofs that the category of algebras for a (relative) monad inherits structure from its base: for instance, it follows easily from \eqref{n_j-pullback} that the category of algebras for an accessible monad on a locally presentable category is itself locally presentable. But perhaps of most interest is the role of the theorem in categorical logic.

Indeed, despite relative monads and their algebras having been introduced only in the last 15 years, pullbacks resembling \eqref{n_j-pullback} are much older. A similar pullback first appeared in \cite[\S2]{linton1969outline}, in \citeauthor{linton1969outline}'s study of \citeauthor{lawvere1963functorial}'s \ssa{}, which relates algebraic theories to their categories of algebras~\cite[Theorem~III.1.2]{lawvere1963functorial}. Motivated by \textcite{linton1969outline}, such pullbacks appeared also in the work of  \citeauthor{diers1974jadjonction}~\cites[\S4.0]{diers1974jadjonction}[\S1]{diers1976foncteur} and of \textcite[Theorem~2.2.7]{lee1977relative}, both of whom took \eqref{n_j-pullback} as the \emph{definition} of categories of algebras for structures that turn out to be equivalent to relative monads~\cite{arkor2024formal}. More recently, the pullback \eqref{n_j-pullback} has been rediscovered in work on generalisations of the notion of algebraic theory, and nerve theorems for associated classes of monads~\cite{leinster2004nerves,weber2007familial,nishizawa2009lawvere,lack2009gabriel,mellies2010segal,lack2011notions,berger2012monads,bourke2019monads,lucyshyn2023enriched}. However, the connection to relative monads has not been observed\footnotemark{}.
\footnotetext{Though see \cref{BFG}.}%
In fact, the nerve theorem for relative monads strictly subsumes the nerve theorems for monads that have appeared previously in the literature. Thus, the importance of the nerve theorem for relative monads has essentially already been appreciated throughout the categorical literature, albeit implicitly.

\subsection{Relative monads and distributors}

While it is entirely possible simply to give a direct proof of the nerve theorem (as we shall do in \cref{unenriched-pullback-theorem}), this sheds little light on the conceptual reason the theorem holds. Thus, the secondary purpose of this paper is to elucidate the phenomenon.

The key to understanding the nature of \eqref{n_j-pullback} lies in the study of the connection between the theory of relative monads and the theory of distributors (\aka{} profunctors or (bi)modules). It was observed in \cite{arkor2024formal} that to every $(\jAE)$-relative monad $T$ there is an associated \emph{loose-monad} (\aka{} promonad, profunctor monad, or arrow) $E(j, T)$ on $A$ (\cf{}~\cite{diers1975jmonades}). In fact, the connection between relative monads and loose-monads is very strong: $j$-relative monads can be characterised as loose-monads whose carrier has the form $E(j, t)$, for a functor $t \colon A \to E$, satisfying a natural compatibility condition (\cref{relative-monad-as-section}). In other words, the $j$-relative monad $T$ and the loose-monad $E(j, T)$ may be seen as different presentations of the same structure. It is then natural to ask to what extent we can understand relative monads via their associated loose-monads: for instance, whether we can capture the $T$-algebras (\aka{} the left-modules, which are classified by the \EM{} category) and $T$-opalgebras (\aka{} the right-modules, which are classified by the Kleisli category) via natural structures associated to $E(j, T)$.

This turns out to be the case: $T$-algebras may be characterised as certain loose-monad modules into $E(j, T)$, assuming $j$ is dense; while $T$-opalgebras may be characterised as certain loose-monad morphisms from $E(j, T)$. Since the Kleisli category and the \EM{} category for $T$ exhibit a universal opalgebra and a universal algebra for $T$ respectively, we may thus characterise the Kleisli category as a certain universal loose-monad morphism associated to $E(j, T)$ -- its \emph{collapse}~\cite{schultz2015regular,wood1985proarrows} -- and the \EM{} category as a certain universal loose-monad module associated to $E(j, T)$ -- its \emph{semanticiser} (\cref{semanticiser}). The collapse of a loose-monad is a kind of (virtual) double categorical colimit, while the semanticiser is a kind of (virtual) double categorical limit. The final step in understanding the nerve theorem is to observe that semanticisers may be constructed from pullbacks and presheaf categories, which, in particular, recovers the nerve theorem \eqref{n_j-pullback}.

In the special case of non-relative monads, the relationship between monads and loose-monads, and the connection to the theory of algebras and opalgebras, was first observed by \textcite{wood1985proarrows} in the setting of proarrow equipments, inspired by \citeauthor{thiebaud1971relative}'s study of the \ssa{}~\cite{thiebaud1971relative}. However, it is fair to say that this relationship is not well known, and the perspective we present is new. Therefore, we hope that this paper serves as an exposition of these ideas even in the non-relative case.

Before we begin, we should say a word about the setting in which we work. While we have spoken so far only about unenriched relative monads, the results we have mentioned hold more generally. Following \cite{arkor2024formal,arkor2024relative}, we shall work throughout in the context of a \emph{virtual equipment}~\cite{cruttwell2010unified}, which is a two-dimensional framework for formal category theory. This permits us to establish our theorems for various flavours of category theory at once: in particular, we obtain a nerve theorem for enriched relative monads (\cref{enriched-pullback-theorem}) by instantiating the main theorem (\cref{pullback-theorem}) in the virtual equipment $\VCat$ of categories enriched in a monoidal category $\V$.

\subsection{Outline of the paper}

We begin in \cref{pullback-theorem-and-applications} by sketching a direct proof of the nerve theorem for unenriched relative monads, presenting several applications of the theorem, and discussing the connection to nerve theorems for monads. In \cref{prerequisites}, we briefly recall the formal categorical concepts from \cite{arkor2024formal} that will be fundamental to our study of the nerve theorem. In \cref{associated-loose-monad}, we study the canonical loose-monad associated to a relative monad, and thereby characterise relative monads as certain sections of loose-monads. We take a brief interlude in \cref{exact-vdcs} to introduce exactness for \ve{}s, which is necessary to relate the algebra objects and opalgebra objects for a relative monad $T$ to its associated loose-monad $E(j, T)$, which we do in \cref{algebras-and-opalgebras}. In \cref{the-pullback-theorem}, we show that, assuming the existence of presheaf objects, we obtain a formal nerve theorem; and in \cref{loose-monads-in-VCat} demonstrate how the formal theory may be applied to $\VCat$ to obtain a nerve theorem for enriched relative monads. We conclude in \cref{loose-monads-and-yo-relative-monads} by observing that loose-monads may be seen as monads relative to Yoneda embeddings (\cf{}~\cite[\S5]{altenkirch2015monads}). Finally, in \cref{strictification} we include a strictification result for \ve{}s (\cref{strictification-for-restriction}) and pseudo equipments (\cref{strictification-of-pseudo-equipments}), which provides conceptual justification for a simplifying assumption taken in the paper~(\cref{strict-ve}).

\begin{remark}
    In \cite[\S5.5]{arkor2022monadic}, the first-named author gave an alternative proof of the nerve theorem for relative monads, based on embeddings of relative monads rather than on distributors. We shall relate the two approaches in future work.
\end{remark}

\begin{remark}
    \label{BFG}
    After presenting an early version of this work to the Masaryk University Algebra Seminar, the authors were informed that John Bourke, Marcelo Fiore, and Richard Garner have, in unpublished joint work, independently established a nerve theorem for enriched relative monads.
\end{remark}

\subsection{Notation}

We have chosen to refrain from some common abuses of notation. We denote the pullback of a cospan $f \colon A \to C \from B \cocolon g$ by $f \times_C g$ rather than by $A \times_C B$. We denote the underlying function of an (enriched) functor $f \colon A \to B$ by $\ob f \colon \ob A \to \ob B$ rather than by $f \colon \ob A \to \ob B$. For morphisms $f \colon a \to b$ and $g \colon b \to c$, we denote by $(f \d g) \colon a \to c$ or $g f \colon a \to c$ their composite.

\subsection{Acknowledgements}

The authors thank the anonymous reviewer for their comments, and in particular for suggesting a simpler proof of \cref{monoid-section-is-monoid}; and thank John Bourke for insights about limits of locally presentable categories.
The first-named author was supported by a departmental postdoctoral grant from the Department of Software Science at Tallinn University of Technology. The second-named author was supported by Icelandic Research Fund grant \textnumero\,228684-052.

\section{The nerve theorem and its applications}
\label{pullback-theorem-and-applications}

We begin by presenting the nerve theorem for unenriched relative monads (\cref{unenriched-pullback-theorem}), sketching a direct proof of the theorem, and giving several applications. The remainder of the paper is devoted to giving an abstract justification for the theorem (\cref{associated-loose-monad,exact-vdcs,algebras-and-opalgebras,the-pullback-theorem}) and thereby deducing the nerve theorem for enriched relative monads (\cref{loose-monads-in-VCat}). For convenience, we first recall the definitions central to this section.

\subsection{Relative adjunctions and relative monads}

\begin{definition}[{\cite[Definition~2.2]{ulmer1968properties}}]
    Let $\jAE$ be a functor. A \emph{$j$-relative adjunction} (or simply \emph{$j$-adjunction}) in $\Cat$ comprises
    \begin{enumerate}
        \item a functor $\ell \colon A \to C$;
        \item a functor $r \colon C \to E$;
        \item an isomorphism $C(\ob \ell x, y) \iso E(\ob j x, \ob r y)$ natural in $x \in \ob A$ and $y \in \ob C$.
    \end{enumerate}
    We denote this situation by $\ell \jadj r$.
    \[\begin{tikzcd}
    	& C \\
    	A && E
    	\arrow[""{name=0, anchor=center, inner sep=0}, "\ell"{pos=0.4}, from=2-1, to=1-2]
    	\arrow[""{name=1, anchor=center, inner sep=0}, "r"{pos=0.6}, from=1-2, to=2-3]
    	\arrow["j"', from=2-1, to=2-3]
    	\arrow["\dashv"{anchor=center}, shift right=2, draw=none, from=0, to=1]
    \end{tikzcd}\]
\end{definition}

Just as every adjunction induces a monad, every $j$-adjunction induces a $j$-monad.

\begin{definition}[{\cite[Definition~2.1]{altenkirch2015monads}}]
    Let $\jAE$ be a functor. A \emph{$j$-relative monad} (or simply \emph{$j$-monad}) in $\Cat$ comprises
    \begin{enumerate}
        \item for each object $x \in \ob A$, an object $\ob t x \in \ob E$;
        \item for each morphism $f \colon \ob j x \to \ob t y$ in $E$, a morphism $f^\dag \colon \ob t x \to \ob t y$ in $E$;
        \item for each object $x \in \ob A$, a morphism $\eta_x \colon \ob j x \to \ob t x$ in $E$,
    \end{enumerate}
    satisfying the following laws.
    \begin{enumerate}[resume]
        \item $\eta_x \d f^\dag = f$ for each morphism $f \colon \ob j x \to \ob t y$ in $E$.
        \item ${\eta_x}^\dag = 1_{\ob t x}$ for each object $x \in \ob A$.
        \item $(f \d g^\dag)^\dag = f^\dag \d g^\dag$ for each pair of morphisms $f \colon \ob j x \to \ob t y$ and $g \colon \ob j y \to \ob t z$ in $E$.
        \qedhere
    \end{enumerate}
\end{definition}

It follows that $\ob t$ extends uniquely to a functor $t \colon A \to E$ for which $\dag$ and $\eta$ are natural~\cite[Theorem~8.12]{arkor2024formal}.

\begin{definition}[{\cite[\S2.3]{altenkirch2015monads}}]
    Let $\jAE$ be a functor and let $T = (t, \dag, \eta)$ be a $j$-monad. The \emph{Kleisli category} $\Kl(T)$ has the same objects as $A$, and has morphisms given by $\Kl(T)(x, y) \defeq E(\ob j x, \ob t y)$. The identity morphism for an object $x \in \ob A$ is given by $\eta_x \colon \ob j x \to \ob t x$ in $E$; and the composite of morphisms $f \colon x \to y$ and $g \colon y \to z$ in $\Kl(T)$ is given by $(f \d g^\dag) \colon \ob j x \to \ob t z$ in $E$.
\end{definition}

The Kleisli category for a $j$-monad forms a $j$-adjunction, whose left adjoint sends a morphism $f \colon x \to y$ in $A$ to $(f \d \eta_y) \colon \ob j x \to \ob t y$ in $E$, and whose right adjoint is given by the action of $\dag$.
\[\begin{tikzcd}
	& {\Kl(T)} \\
	A && E
	\arrow[""{name=0, anchor=center, inner sep=0}, "{v_T}", from=1-2, to=2-3]
	\arrow[""{name=1, anchor=center, inner sep=0}, "{k_T}", from=2-1, to=1-2]
	\arrow["j"', from=2-1, to=2-3]
	\arrow["\dashv"{anchor=center}, shift right=2, draw=none, from=1, to=0]
\end{tikzcd}\]

\begin{definition}[{\cite[Definition~2.11]{altenkirch2015monads}}]
    Let $\jAE$ be a functor and let $T = (t, \dag, \eta)$ be a $j$-monad. A \emph{$T$-algebra} comprises
    \begin{enumerate}
        \item an object $e \in \ob E$;
        \item for each morphism $f \colon \ob j x \to e$ in $E$, a morphism $f^\aop \colon \ob t x \to e$ in $E$,
    \end{enumerate}
    satisfying the following laws.
    \begin{enumerate}[resume]
        \item $\eta_x \d f^\dag = f$ for each morphism $f \colon \ob j x \to e$ in $E$.
        \item $(f \d g^\aop)^\aop = f^\dag \d g^\aop$ for each pair of morphisms $f \colon \ob j x \to \ob t y$ and $g \colon \ob j y \to e$ in $E$.
    \end{enumerate}
    A \emph{$T$-algebra morphism} from $(e, \aop)$ to $(e', \aop')$ is a morphism $\epsilon \colon e \to e'$ such that $f^\aop \d \epsilon = (f \d \epsilon)^{\aop'}$ for each morphism $f \colon \ob j x \to e$ in $E$. $T$-algebras and their morphisms form a category $\Alg(T)$.
\end{definition}

It follows that, for each $T$-algebra $(e, \aop)$, the extension operator $\aop$ is natural. The category of algebras for a $j$-monad forms a $j$-adjunction, whose left adjoint is given by applying $t$, and whose right adjoint forgets the $T$-algebra structure.
\[\begin{tikzcd}
	& {\Alg(T)} \\
	A && E
	\arrow[""{name=0, anchor=center, inner sep=0}, "{u_T}", from=1-2, to=2-3]
	\arrow[""{name=1, anchor=center, inner sep=0}, "{f_T}", from=2-1, to=1-2]
	\arrow["j"', from=2-1, to=2-3]
	\arrow["\dashv"{anchor=center}, shift right=2, draw=none, from=1, to=0]
\end{tikzcd}\]

There is a canonical \ff{} comparison functor $i_T \colon \Kl(T) \to \Alg(T)$, which sends each object $x \in \ob A$ to $\ob t x$, and each morphism $f \colon \ob j x \to \ob t y$ in $E$ to $f^\dag$, thus exhibiting $\Kl(T)$ as the category of free $T$-algebras.

\subsection{The nerve theorem}

\begin{definition}
    For a functor $\jAE$ between locally small categories, we denote by $n_j \colon E \to [A\op, \Set]$ the \emph{nerve} of $j$, which sends an object $e \in \ob E$ to the restricted Yoneda embedding $a \mapsto E(ja, e)$. By definition, $j$ is \emph{dense} if its nerve $n_j$ is \ff{}.
\end{definition}

\begin{theorem}
    \label{unenriched-pullback-theorem}
    Let $\jAE$ be a dense functor between locally small categories. For a $j$-monad $T$, the category of $T$-algebras exhibits the following pullback of categories.
    \[\begin{tikzcd}
    	{\Alg(T)} & {[\Kl(T)\op, \Set]} \\
    	E & {[A\op, \Set]}
    	\arrow["{u_T}"', from=1-1, to=2-1]
    	\arrow["{[k_T\op, \Set]}", from=1-2, to=2-2]
    	\arrow[""{name=0, anchor=center, inner sep=0}, "{n_j}"', hook, from=2-1, to=2-2]
    	\arrow[hook, from=1-1, to=1-2]
    	\arrow["\lrcorner"{anchor=center, pos=0.125}, draw=none, from=1-1, to=0]
    \end{tikzcd}\]
    Furthermore, the unlabelled functor above is isomorphic to the nerve of the comparison functor $i_T \colon \Kl(T) \to \Alg(T)$, exhibiting $i_T$ as a dense functor.
\end{theorem}

\begin{sketch}
    Given a $T$-algebra $(e \in |E|, \aop \colon E(j{-}, e) \tto E(t{-}, e))$, we define a presheaf on $\Kl(T)$ by the following assignments,
    \begin{align*}
        a & \mapsto E(\ob j a, e) &
        (f \colon \ob j a \to \ob t a') & \mapsto \big((g \colon \ob j a' \to e) \mapsto ((f \d g^\aop) \colon \ob j a \to e)\big)
    \end{align*}
    with functoriality following from the two $T$-algebra laws. A $T$-algebra morphism induces a natural transformation between the corresponding presheaves by postcomposition. This defines a functor $\Alg(T) \to [\Kl(T)\op, \Set]$ rendering the square above commutative, and thus induces a functor $\Alg(T) \to n_j \times_{[A\op, \Set]} [k_T\op, \Set]$ into the pullback.

    In the other direction, given an object $e \in \ob E$ and a presheaf $p \colon \Kl(T)\op \to \Set$ in the pullback, observe that each Kleisli morphism $f \colon \ob j a \to \ob t a'$ induces a function $p(f) \colon p(a') = E(\ob j a', e) \to E(\ob j a, e) = p(a)$. This assignment defines a natural transformation $E(\ob j a, \ob t a') \times E(\ob j a', e) \to E(\ob j a, e)$, which, by transposing, is equivalently a natural transformation $E(\ob j a', e) \to [A\op, \Set](E(j{-}, \ob t a'), E(j{-}, e))$. When $j$ is dense, so that $n_j$ is \ff{}, the codomain is isomorphic to $E(\ob t a', e)$, which produces a function $E(\ob j a', e) \to E(\ob t a', e)$ for each $a' \in \ob A$. This family of functions forms a $T$-algebra structure on $e$, the $T$-algebra laws following from functoriality of $p$. Furthermore, given a morphism $\epsilon \colon e \to e'$ and a natural transformation $p \tto p'$ in the pullback, naturality of the latter implies that $\epsilon$ is a morphism between the induced $T$-algebras. The assignment thus extends to a functor $n_j \times_{[A\op, \Set]} [k_T\op, \Set] \to \Alg(T)$.

    We leave to the reader the routine checks that the functors $\Alg(T) \rightleftarrows n_j \times_{[A\op, \Set]} [k_T\op, \Set]$ thus defined are inverse to one another.

    Finally, since $f_T \jadj u_T$, and $f_T = (k_T \d i_T)$, we have an isomorphism \[E(\ob j a, e) \iso \Alg(T)(\ob{i_T k_T} a, (e, \aop)) = \Alg(T)(\ob{i_T} a, (e, \aop))\] natural in $(e, \aop) \in \ob{\Alg(T)}$, using that $k_T$ is \ioo{}. Thus, to establish that the unlabelled functor is isomorphic to the nerve of $i_T$, it is enough to verify that this isomorphism is natural in $a \in \ob{\Kl(T)}$. However, naturality corresponds to the commutativity of the following square, for each morphism $f \colon \ob j a \to \ob ta'$, which is trivial.
    \[\begin{tikzcd}[column sep=large]
    	{E(\ob ja, e)} & {E(\ob ja', e)} \\
    	{\Alg(T)(\ob{i_T}a, (e, \aop))} & {\Alg(T)(\ob{i_T}a', (e, \aop))}
    	\arrow["{\ph^\aop}"', from=1-1, to=2-1]
    	\arrow["{f \d \ph^\aop}"', from=1-2, to=1-1]
    	\arrow["{\ph^\aop}", from=1-2, to=2-2]
    	\arrow["{(f \d \ph)^\aop}", from=2-2, to=2-1]
    \end{tikzcd}\]
    Thus the unlabelled functor is isomorphic to the nerve $n_{i_T} \colon \Alg(T) \to [\Kl(T)\op, \Set]$. That the nerve is \ff{} (and hence that $i_T$ is dense) follows from stability of \ff{} functors under pullback.
\end{sketch}

Dually, the category of coalgebras for a relative comonad may be expressed as a pullback over a category of copresheaves. (Since the theory of relative comonads is formally dual to the theory of relative monads, we shall leave subsequent dualisations for the reader to spell out.)

\begin{manualtheorem}{\ref*{unenriched-pullback-theorem}$\co$}
    \label{unenriched-pullback-theorem-for-comonads}
    Let $i \colon Z \to U$ be a codense functor. For an $i$-comonad $D$, the category of $D$-coalgebras exhibits the following pullback of categories.
    \[\begin{tikzcd}
        {\Coalg(D)} & {[\Kl(D), \Set]\op} \\
        U & {[Z, \Set]\op}
        \arrow[hook, from=1-1, to=1-2]
        \arrow["{u_D}"', from=1-1, to=2-1]
        \arrow["{[k_D, \Set]\op}", from=1-2, to=2-2]
        \arrow[""{name=0, anchor=center, inner sep=0}, "{m_i}"', hook, from=2-1, to=2-2]
        \arrow["\lrcorner"{anchor=center, pos=0.125}, draw=none, from=1-1, to=0]
    \end{tikzcd}\]
    Above, we denote by $m_i \colon U \to [Z, \Set]\op$ the co-nerve of $i$, which sends an object $u \in \ob U$ to the restricted Yoneda embedding $z \mapsto U(u, iz)$. Furthermore, the unlabelled functor above is isomorphic to the co-nerve of the comparison functor $i_D \colon \Kl(D) \to \Coalg(D)$, exhibiting $i_D$ as a codense functor.
    \qed
\end{manualtheorem}

\begin{remark}
    \newcommand{\triangleup}{\flip\triangledown}
    When $j$ is not dense, there is still a canonical comparison functor
    \[\Alg(T) \to n_j \times_{[A\op, \Set]} [{k_T}\op, \Set]\]
    from the category of algebras into the pullback, but it is not, in general, an equivalence. For instance, denote by $V$ the freestanding span $\{ \triangledown \from \lozenge \to \triangleup \}$ and let $j \colon 1 \to V$ denote the constant functor on $\lozenge$. $j$ is not dense, since $V(\lozenge, \triangledown) \iso 1 \iso V(\lozenge, \triangleup)$ but $V(\triangledown, \triangleup) = \emptyset$. Consider $T$ the constant $j$-monad on $\triangleup$. $k_T$ is the identity on $1$, so that  ${k_T}^*$ is the identity on $[1\op, \Set] \equiv \Set$, and thus the apex of the pullback is $V$. However, $\Alg(T) \iso 1 \not\equiv V$, since there is a unique $T$-algebra (whose underlying object is $\triangleup$).

    On the other hand, there are cases in which $j$ is not dense, but the square \eqref{n_j-pullback} in \cref{unenriched-pullback-theorem} is still a pullback. For instance, suppose that $\jAE$ is \ff{}. Then the Kleisli category for $j$, viewed as a trivial $j$-monad, is isomorphic to $A$~\cite[Proposition~6.55]{arkor2024formal}, while the category of algebras is $E$~\cite[Proposition~6.42]{arkor2024formal}. Therefore, the square \eqref{n_j-pullback} trivially forms a pullback in $\Cat$, as shown below. However, $j$ need not be dense.
    \[\begin{tikzcd}
    	E & {[\Kl(j)\op, \Set]} \\
    	E & {[A\op, \Set]}
    	\arrow["{1_E}"', from=1-1, to=2-1]
    	\arrow["\iso", from=1-2, to=2-2]
    	\arrow[""{name=0, anchor=center, inner sep=0}, "{n_j}"', from=2-1, to=2-2]
    	\arrow[from=1-1, to=1-2]
    	\arrow["\lrcorner"{anchor=center, pos=0.125}, draw=none, from=1-1, to=0]
    \end{tikzcd}\]
    (For an alternative proof that the nerve theorem holds when $j = t$ and $j$ is \ff{}, observe that, in the proof of \cref{unenriched-pullback-theorem}, we then have that
    \[[A\op, \Set](E(j{-}, \ob j a'), E(j{-}, e)) \iso [A\op, \Set](A({-}, a'), E(j{-}, e)) \iso E(\ob j a', e)\]
    using \ffness{} of $j$, followed by the Yoneda lemma.)
\end{remark}

\begin{remark}
    Pullbacks of \ioo{} functors along nerves, as below, were first considered by \textcite[\S5]{linton1969outline}, who called objects of the pullback \emph{$k$-algebras (rel.\ $j$)}. Indeed, elements of the proof of \cref{unenriched-pullback-theorem} are present \loccit. However, \citeauthor{linton1969outline} had no notion of relative monad, nor of algebras thereof, to formulate the nerve theorem at this level of generality.
    \[\begin{tikzcd}
    	\cdot & {[K\op, \Set]} \\
    	E & {[A\op, \Set]}
    	\arrow[from=1-1, to=2-1]
    	\arrow["{[k\op, \Set]}", from=1-2, to=2-2]
    	\arrow[""{name=0, anchor=center, inner sep=0}, "{n_j}"', from=2-1, to=2-2]
    	\arrow[from=1-1, to=1-2]
    	\arrow["\lrcorner"{anchor=center, pos=0.125}, draw=none, from=1-1, to=0]
    \end{tikzcd}\]
    Later, \citeauthor{diers1975jmonades} and \citeauthor{lee1977relative} independently established the equivalence of the following structures~\cites[Th\'eor\`eme~2.1 \& Th\'eor\`eme~2.2]{diers1975jmonades}[Theorem~2.2.7 \& Corollary~2.2.8]{lee1977relative}.
    \begin{enumerate}
        \item Pullbacks of the form in \cref{unenriched-pullback-theorem}.
        \item Terminal resolutions of $j$-monads.
        \item Right $j$-adjoints creating $j$-absolute colimits.
    \end{enumerate}
    However, neither author had an independent definition of the category of algebras for a relative monad: \citeauthor{diers1975jmonades} took (1) to be the definition thereof, whilst \citeauthor{lee1977relative} took (2) to be the definition. Therefore, \cref{unenriched-pullback-theorem}, which adds a fourth equivalent characterisation to the list above, is entirely new.
    \begin{enumerate}[resume]
        \item Categories of algebras for $j$-monads.
    \end{enumerate}
    (In our formal setting, the equivalence between (2 \& 4) follows from the universal property of algebra objects \cite[Corollary~6.41]{arkor2024formal}, while the equivalence between (3 \& 4) follows from the relative monadicity theorem~\cite[Corollary~4.8]{arkor2024relative}.)
\end{remark}

The proof of \cref{unenriched-pullback-theorem} readily generalises, with appropriate modifications, to enriched relative monads. However, rather than give such a concrete proof, we shall derive the enriched nerve theorem (\cref{enriched-pullback-theorem}) from a more general result using formal techniques.

\subsection{Applications}

We give a number of applications of the nerve theorem. In our examples, we shall freely make use of the enriched nerve theorem, which will be proven in \cref{enriched-pullback-theorem}. For simplicity, we shall here take $\V$ to be a closed monoidal locally presentable category, while noting that the nerve theorem itself holds under much weaker assumptions on the base of enrichment $\V$.

First, we observe that the nerve theorem for relative monads does indeed generalise the classical characterisation theorem for the category of algebras for a monad.

\begin{example}
    Let $A$ be a $\V$-category admitting a presheaf $\V$-category $\P A$, and let $T$ be a $\V$-enriched monad on $A$. Taking $j = 1_A$, so that the nerve of $j$ is the Yoneda embedding, the following diagram forms a pullback. This recovers the original characterisation of \textcite[Observation~1.1]{linton1969outline}, as well as the enriched variant~\cites{linton1972algebren}[Theorem~14]{street1972formal}{linton1974relative}.
    \[\begin{tikzcd}
    	{\Alg(T)} & {\P(\Kl(T))} \\
    	A & {\P A}
    	\arrow["{u_T}"', from=1-1, to=2-1]
    	\arrow["{{k_T}^*}", from=1-2, to=2-2]
    	\arrow[""{name=0, anchor=center, inner sep=0}, "{\yo_A}"', hook, from=2-1, to=2-2]
    	\arrow[hook, from=1-1, to=1-2]
    	\arrow["\lrcorner"{anchor=center, pos=0.125}, draw=none, from=1-1, to=0]
    \end{tikzcd}\]
\end{example}

One of the motivating applications of the nerve theorem for relative monads is to the theory of algebraic theories, in the sense of \textcite[Chapter~2]{lawvere1963functorial} (\cf{}~\cite[Example~5.6]{arkor2024relative}).

\begin{example}
    \label{algebraic-theories}
    Denote by $\F$ the category of finite ordinals, which is the free category with strictly associative and unital finite coproducts on a single object $1$. The inclusion $j \colon \F \equiv \FinSet \ffto \Set$ of finite ordinals into sets exhibits a cocompletion under sifted colimits, and is hence dense and \ff{}.

    Now consider a $j$-monad $T$. The inclusion of $\bb F$ into the Kleisli category of $T$ forms an algebraic theory $k_T \colon \bb F \to \Kl(T)$, since $k_T$ is \ioo{} and preserves finite coproducts. The category of $T$-algebras forms a pullback as follows.
    \[\begin{tikzcd}
    	{\Alg(T)} & {[\Kl(T)\op, \Set]} \\
    	\Set & {[\F\op, \Set]}
    	\arrow["{u_T}"', from=1-1, to=2-1]
    	\arrow["{[k_T\op, \Set]}", from=1-2, to=2-2]
    	\arrow[""{name=0, anchor=center, inner sep=0}, "{n_j}"', hook, from=2-1, to=2-2]
    	\arrow[hook, from=1-1, to=1-2]
    	\arrow["\lrcorner"{anchor=center, pos=0.125}, draw=none, from=1-1, to=0]
    \end{tikzcd}\]
    Noting that $\Kl(T)$ has finite coproducts, the following also forms a pullback.
    \[\begin{tikzcd}
    	{\Cart[\Kl(T)\op, \Set]} & {[\Kl(T)\op, \Set]} \\
    	{\Cart[\F\op, \Set]} & {[\F\op, \Set]}
    	\arrow["{\Cart[k_T\op, \Set]}"', from=1-1, to=2-1]
    	\arrow["{[k_T\op, \Set]}", from=1-2, to=2-2]
    	\arrow[""{name=0, anchor=center, inner sep=0}, hook, from=2-1, to=2-2]
    	\arrow[hook, from=1-1, to=1-2]
    	\arrow["\lrcorner"{anchor=center, pos=0.125}, draw=none, from=1-1, to=0]
    \end{tikzcd}\]
    Moreover, since $k_T$ is \ioo{}, $[k_T\op, \Set]$ is an amnestic isofibration, and so both pullbacks are furthermore bipullbacks~\cite[Corollary~1]{joyal1993pullbacks}. By the universal property of $\F$, there is an equivalence of categories $\Cart[\F\op, \Set] \equiv \Set$ concrete over $[\F\op, \Set]$. Consequently, the category of $T$-algebras is concretely equivalent to the category of finite product-preserving functors $\Kl(T)\op \to \Set$, and thus to the category of algebras for the algebraic theory $k_T$.

    Conversely, the relative monadicity theorem establishes that the category of models for an algebraic theory is strictly $j$-monadic~\cite[Example~5.6]{arkor2024relative}. Together, these observations exhibit an equivalence between the category of algebraic theories and the category of $j$-monads, which commutes with taking categories of algebras (\cf{}~\cite[Chapter~3]{arkor2022monadic}).
\end{example}

The analysis of \cref{algebraic-theories} extends to more general notions of algebraic theory; this perspective will be explicated in future work (\cf{}~\cites[Chapter~7]{arkor2022monadic}[]{arkor2022relative}).

Another useful application of the nerve theorem is in establishing properties of the categories of algebras for relative monads: for instance, local presentability.

\begin{example}
    \label{LP}
    Let $\jAE$ be a dense $\V$-functor with small domain and locally presentable codomain. The nerve $n_j \colon E \to \P A$ has a left adjoint, since $E$ is small-cocomplete. Therefore, for every $j$-monad $T$, each of the $\V$-functors appearing in the cospan of the pullback of $\V$-categories below is a right adjoint between locally presentable $\V$-categories.
    \[\begin{tikzcd}
    	{\Alg(T)} & {\P(\Kl(T))} \\
    	E & {\P A}
    	\arrow["{u_T}"', from=1-1, to=2-1]
    	\arrow["{{k_T}^*}", from=1-2, to=2-2]
    	\arrow[""{name=0, anchor=center, inner sep=0}, "{n_j}"', hook, from=2-1, to=2-2]
    	\arrow[hook, from=1-1, to=1-2]
    	\arrow["\lrcorner"{anchor=center, pos=0.125}, draw=none, from=1-1, to=0]
    \end{tikzcd}\]
    Consequently, by \cite[Theorem~6.11]{bird1984limits}, the $\V$-category of $T$-algebras is locally presentable, and each of the functors above is a right adjoint. Furthermore, it follows from \cite[Proposition~4.12]{arkor2024relative} that $u_T$ is strictly monadic.

    In particular, if $A$ is $\kappa$-cocomplete, for some regular cardinal $\kappa$, and $j$ preserves $\kappa$-small colimits, then $n_j$ is $\kappa$-accessible (\ie{} preserves $\kappa$-filtered colimits), in which case $\Alg(T)$ is locally $\kappa$-presentable and each of the functors above is $\kappa$-accessible~\cite[Theorem~6.10]{bird1984limits}.
\end{example}

An immediate consequence of \cref{LP} is a simple proof that categories of algebras for accessible monads on locally presentable categories are themselves locally presentable~\cite{gabriel1971lokal,ulmer1971locally}.

\begin{example}
    Let $\kappa$ be a regular cardinal and let $T$ be a $\kappa$-accessible monad on a locally $\kappa$-presentable $\V$-category $E$. Denote by $j \colon A \to E$ the full subcategory of $\kappa$-presentable objects in $E$. By \cite[Example~4.8]{arkor2024formal}, the forgetful $\V$-functor $u_T \colon \Alg(T) \to E$ is strictly $j$-monadic. Thus, by \cref{LP}, $\Alg(T)$ is locally $\kappa$-presentable and $u_T$ is $\kappa$-accessible.
\end{example}

\begin{example}
    \label{soundness}
    \Cref{algebraic-theories} and \cref{LP} (for fixed $\kappa$) are instances of a more general phenomenon that holds for any class of weights that is \emph{sound} in the sense of \cite{adamek2002classification,lack2011notions}. Let $\Psi$ be a class of small weights and let $\Psi\flatw$ be the class of $\Psi$-flat weights~\cite{kelly2005notes}. We say that $\Psi$ is \emph{sound} if, for every $\Psi$-cocomplete $\V$-category $A$, the $\V$-category of $\Psi$-exact $\V$-presheaves $\P_\Psi(A)$ is equivalent to the free cocompletion $\Psi\flatw(A)$ of $A$ under $\Psi\flatw$-weighted colimits.

    Let $A$ be a small $\Psi$-cocomplete $\V$-category, and let $E$ be a \emph{locally $\Psi$-presentable $\V$-category}, \ie{} the free $\Psi\flatw$-cocompletion of a small $\Psi$-cocomplete $\V$-category. Then, for any dense $\V$-functor $\jAE$, the $\V$-category of algebras for any $j$-monad is also locally $\Psi$-presentable. We briefly sketch the proof idea; the full details will appear elsewhere.

    Let $T$ be a $j$-monad. The nerve theorem implies that the following diagram is a pullback of $\V$-categories; since ${k_T}^*$ is an amnestic isofibration, it is also a bipullback~\cite[Corollary~1]{joyal1993pullbacks}.
    \[\begin{tikzcd}
    	{\Alg(T)} & {\P(\Kl(T))} \\
    	E & {\P A}
    	\arrow["{u_T}"', from=1-1, to=2-1]
    	\arrow["{{k_T}^*}", from=1-2, to=2-2]
    	\arrow[""{name=0, anchor=center, inner sep=0}, "{n_j}"', hook, from=2-1, to=2-2]
    	\arrow[hook, from=1-1, to=1-2]
    	\arrow["\lrcorner"{anchor=center, pos=0.125}, draw=none, from=1-1, to=0]
    \end{tikzcd}\]
    By analogous arguments to \cref{LP}, $n_j$ and ${k_T}^*$ preserve $\Psi$-flat colimits and admit left adjoints. Thus, it suffices to show that the 2-category of locally $\Psi$-presentable $\V$-categories and $\Psi\flatw$-cocontinuous right-adjoint functors admits bipullbacks, which will be preserved by the forgetful 2-functor into $\VCat$. By a duality theorem for $\Psi$-complete $\V$-categories analogous to that of \cite[Theorem~3.1]{centazzo2002duality}, this 2-category is dually biequivalent to the 2-category of small $\Psi$-complete $\V$-categories admitting absolute limits, which is the 2-category of algebras for an accessible 2-monad on $\VCat$, and hence admits bicolimits~\cite[Theorem~5.8]{blackwell1989two}. Since bilimits are preserved by biequivalence, the 2-category of locally $\Psi$-presentable $\V$-categories consequently admits bipullbacks, and thus the $\V$-category of $T$-algebras in the pullback above is locally $\Psi$-presentable.

    In particular, when $\V = \Set$ and $\Psi$ is the class of finite discrete weights, this recovers \cref{algebraic-theories}; and, when $\Psi$ is the class of $\kappa$-small weights, recovers \cref{LP}.
\end{example}

Next, we give an application of the nerve theorem to colimits of relative monads. Recall from \cite[Corollary~6.40]{arkor2024formal} that the assignment taking each $j$-monad $T$ to its forgetful functor $u_T \colon \Alg(T) \to E$ forms a \ff{} functor $u_{({-})} \colon \RMnd(j)\op \to \Cat/E$. We extend the terminology of \cite[\S26]{kelly1980unified} from monads to relative monads.

\begin{definition}
    Let $\jAE$ and $T_{\ph} \colon I \to \RMnd(j)$ be functors.
    A colimit of $T$ is \emph{algebraic} if it is preserved by $u_{({-})}\op \colon \RMnd(j) \to (\Cat/E)\op$.
\end{definition}

In other words, given some functor $T \colon I \to \RMnd(j)$ admitting a colimit $\colim_{i \in I} T_i$, the colimit is algebraic if it is sent by $u_{({-})}$ to a limit in $\Cat/E$. The following then generalises \cite[Proposition~26.4]{kelly1980unified} from monads to relative monads with dense roots.

\begin{proposition}
    \label{algebraic-colimit-iff-left-j-adjoint}
    Let $\jAE$ be a dense functor with small domain. A diagram ${T \colon I \to \RMnd(j)}$ of $j$-monads admits an algebraic colimit if and only if the limit $\lim_{i \in \ob I} u_{T_i}$ in $\Cat/E$ admits a left $j$-adjoint.
\end{proposition}

\begin{proof}
  Denote by $u \colon D \to E$ the limit $\lim_{i \in \ob I} u_{T_i}$.
  The functor $u_{({-})} \colon \RMnd(j)\op \to \Cat/E$ reflects limits because it is \ff{}, so it is enough to show that, if $u$ admits a left $j$-adjoint, then it exhibits the category of algebras for the induced $j$-monad.

  Denote by $k \colon A \to B$ the colimit $\colim_{i \in \ob I} k_{T_i}$ in $A/\Cat$. The functor \[[\ph\op, \Set] \colon A/\Cat \to (\Cat/[A\op, \Set])\op\] admits a right adjoint, and hence preserves colimits: this follows from the ``preliminary \ssa{}'' of \cite[Theorem~3.1]{linton1969outline} relative to the Yoneda embedding $\yo_A \colon A \to [A\op, \Set]$, observing that $A\op/\Cat \iso A/\Cat$.
  We thus have that $[(\colim_{i \in \ob I} k_{T_i})\op, \Set]$ exhibits the limit $\lim_{i \in \ob I} [{k_{T_i}}\op, \Set]$ in $\Cat/[A\op, \Set]$. Now, for each $i \in \ob I$, the nerve theorem implies that the following square is a pullback of categories.
\[\begin{tikzcd}
	{\Alg(T_i)} & {[\Kl(T_i)\op, \Set]} \\
	E & {[A\op, \Set]}
	\arrow["{u_{T_i}}"', from=1-1, to=2-1]
	\arrow["{[k_{T_i}\op, \Set]}", from=1-2, to=2-2]
	\arrow[""{name=0, anchor=center, inner sep=0}, "{n_j}"', hook, from=2-1, to=2-2]
	\arrow[hook, from=1-1, to=1-2]
	\arrow["\lrcorner"{anchor=center, pos=0.125}, draw=none, from=1-1, to=0]
\end{tikzcd}\]
  Since the pullback functor $n_j \times_{[A\op, \Set]} \ph \colon \Cat/[A\op, \Set] \to \Cat/E$ is right-adjoint (to postcomposition by $n_j$), it preserves limits, so that the following square also exhibits a pullback of categories.
  \begin{equation}
  \label{limit-pullback}
\begin{tikzcd}
	D & {[B\op, \Set]} \\
	E & {[A\op, \Set]}
	\arrow["{u = \lim_{i \in \ob I} u_{T_i}}"', from=1-1, to=2-1]
	\arrow["{\lim_{i \in \ob I} [{k_{T_i}}\op, \Set] = [k\op, \Set]}", from=1-2, to=2-2]
	\arrow[""{name=0, anchor=center, inner sep=0}, "{n_j}"', hook, from=2-1, to=2-2]
	\arrow[hook, from=1-1, to=1-2]
	\arrow["\lrcorner"{anchor=center, pos=0.125}, draw=none, from=1-1, to=0]
\end{tikzcd}
\end{equation}
  $k$ is \ioo{}, so that $[k\op, \Set]$ strictly creates colimits. Thus $u$ strictly creates those colimits preserved by $n_j$, \ie{} the $j$-absolute colimits, by \cite[Proposition~21.7.2(c)]{schubert1972categories}. $u$ is therefore $j$-monadic if and only if it admits a left $j$-adjoint by \cite[Proposition~4.12]{arkor2024relative}.
\end{proof}

\begin{example}
    Let $\jAE$ be a dense functor with small domain and locally presentable codomain. Then each of the functors appearing in the cospan of the pullback \eqref{limit-pullback} is a right adjoint between locally presentable categories, so that $u$ is also right adjoint by \cite[Theorem~6.11]{bird1984limits}. Thus $\RMnd(j)$ admits all small algebraic colimits.
\end{example}

As a final example, we explain how the nerve theorem for relative monads relates to the nerve theorems for monads arising in the theory of monads with arities~\cite{weber2007familial} and nervous monads~\cite{bourke2019monads}.

\begin{example}[Nerve theorems for monads]
    \label{nerve-theorems}
    Let $\jAE$ be a dense $\V$-functor with small domain and let $T$ be a $\V$-enriched monad on $E$. Precomposition by $j$ defines a $\V$-enriched $j$-monad $(j \d T)$~\cite[Proposition~5.36]{arkor2024formal}. The nerve theorem for relative monads implies that the following diagram is a pullback of $\V$-categories.
    \[\begin{tikzcd}
    	{\Alg(j \d T)} & {\P(\Kl(j \d T))} \\
    	E & {\P A}
    	\arrow["{u_{j \d T}}"', from=1-1, to=2-1]
    	\arrow["{{k_{j \d T}}^*}", from=1-2, to=2-2]
    	\arrow[""{name=0, anchor=center, inner sep=0}, "{n_j}"', hook, from=2-1, to=2-2]
    	\arrow[hook, from=1-1, to=1-2]
    	\arrow["\lrcorner"{anchor=center, pos=0.125}, draw=none, from=1-1, to=0]
    \end{tikzcd}\]
    The Kleisli $\V$-category for $(j \d T)$ is given by the full image of $(j \d k_T)$, \ie{} the full subcategory of $\Kl(T)$ spanned by the objects of $A$.
    \[\begin{tikzcd}
    	{\Kl(j \d T)} & {\Kl(T)} \\
    	A & E
    	\arrow["{k_T}"', two heads, from=2-2, to=1-2]
    	\arrow["j"', from=2-1, to=2-2]
    	\arrow["{k_{j \d T}}", two heads, from=2-1, to=1-1]
    	\arrow[hook, from=1-1, to=1-2]
    \end{tikzcd}\]
    Now, assuming that $u_{j \d T}$ admits a left adjoint, \cite[Proposition~4.12]{arkor2024relative} implies that $u_{j \d T}$ is strictly monadic, and so the $\V$-category of $T$-algebras is isomorphic over $E$ to the $\V$-category of $(j \d T)$-algebras.
    \[\begin{tikzcd}
    	{\Alg(T)} && {\Alg(j \d T)} \\
    	& E
    	\arrow["\iso", from=1-1, to=1-3]
    	\arrow["{u_T}"', from=1-1, to=2-2]
    	\arrow["{u_{j \d T}}", from=1-3, to=2-2]
    \end{tikzcd}\]
    Monads satisfying this condition are called \emph{$j$-ary} in \cite[Definition~5.4.1]{arkor2022monadic}. In particular, monads with arities $j$~\cite[Definition~4.1]{weber2007familial} and $j$-nervous monads~\cite[Definition~17]{bourke2019monads} both satisfy the given assumptions, and are hence $j$-ary. We thereby recover the nerve theorems of \cites[Theorem~4.10]{weber2007familial}[\S4]{mellies2010segal}[Theorem~1.10]{berger2012monads}, and the characterisations of categories of algebras for nervous monads in \cites[Theorem~18]{bourke2019monads}[Proposition~4.10]{lucyshyn2023enriched}. In future work, we shall study the notion of $j$-ary monad in detail to clarify the precise relationship between these concepts (\cf{}~\cite[\S5.4]{arkor2022monadic}).
\end{example}

\section{Formal categorical prerequisites}
\label{prerequisites}

We shall work throughout the remainder of paper in the context of a \ve{} $\X$, adopting the terminology and notation of \cite{arkor2024formal}.
In this section, we shall briefly recall the fundamental concepts of formal category theory in a \ve{}.
A more detailed introduction may be found in \S2~--~\S3 of \cite{arkor2024formal}.

A \ve{} is, in particular, a \vdc{}, which is a two-dimensional structure having \emph{objects}; \emph{tight-cells} ($\to$); \emph{loose-cells} ($\lto$); and \emph{2-cells} ($\tto$) of the following form (\cf~\cite[Definition~2.1]{arkor2024formal}).
\[\begin{tikzcd}
    {A_0} & {A_1} & \cdots & {A_{n - 1}} & {A_n} \\
    {B_0} &&&& {B_n}
    \arrow["q", "\shortmid"{marking}, from=2-5, to=2-1]
    \arrow["{p_n}"', "\shortmid"{marking}, from=1-5, to=1-4]
    \arrow["{p_1}"', "\shortmid"{marking}, from=1-2, to=1-1]
    \arrow["{p_{n - 1}}"', "\shortmid"{marking}, from=1-4, to=1-3]
    \arrow["{p_2}"', "\shortmid"{marking}, from=1-3, to=1-2]
    \arrow[""{name=0, anchor=center, inner sep=0}, "{f_n}", from=1-5, to=2-5]
    \arrow[""{name=1, anchor=center, inner sep=0}, "{f_0}"', from=1-1, to=2-1]
    \arrow["\phi"{description}, draw=none, from=0, to=1]
\end{tikzcd}\]
For instance, the \vdc{} $\VCat$ has as objects the (possibly large) $\V$-categories; as tight-cells the $\V$-functors; as loose-cells the $\V$-distributors; and as 2-cells the $\V$-natural transformations, which are families of morphisms
\[\phi_{x_0, \ldots, x_n} \colon p_1(x_0, x_1) \otimes \cdots \otimes p_n(x_{n - 1}, x_n) \to q(\ob{f_0} x_0, \ob{f_n} x_n)\]
in $\V$ for each $x_0 \in \ob{A_0}, \dots, x_n \in \ob{A_n}$, satisfying several naturality laws~\cite[Definition~8.1]{arkor2024formal}. For tight-cells $f \colon A \to B$ and $g \colon B \to C$, we denote by $(f \d g) \colon A \to C$ or $g f \colon A \to C$ their composite. For each pair of objects $A$ and $B$ in $\X$, the category of loose-cells $A \lto B$ and globular 2-cells is denoted by $\X\lh{A, B}$, while the set of tight-cells $A \to B$ is denoted by $\X[A, B]$.

A \ve{} is a \vdc{} for which every object $A$ admits a loose-identity $A(1, 1) \colon A \lto A$ (\cite[Definition~2.4]{arkor2024formal}); and for which, for every loose-cell $p \colon X \lto Y$ and tight-cells $g \colon W \to X$ and $f \colon Z \to Y$, there is a restriction loose-cell $p(f, g) \colon W \lto Z$ (\cite[Definition~2.7]{arkor2024formal}). We denote the restriction $A(1, 1)(f, g)$ along a loose-identity $A(1, 1)$ by $A(f, g)$: in $\VCat$, this restriction is given by the hom-objects of $A$. Finally, for each tight-cell $f \colon A \to B$, we denote by $\pc f \colon A(1, 1) \tto B(f, f)$ and by $\cp f \colon B(1, f), B(f, 1) \tto A(1, 1)$ the unit and counit respectively of the loose-adjunction $B(1, f) \adj B(f, 1)$~\cite[Notation~2.8]{arkor2024formal}. In a \ve{}, we may consider 2-cells between tight-cells, and consequently each set of tight-cells $\X[A, B]$ extends to a category. A virtual equipment is representable if it admits composite of loose-cells; a representable virtual equipment is equivalently a pseudo double category with companions and conjoints.

\subsection{Relative monads, opalgebra objects and algebra objects}

We recall the concepts introduced in \cite{arkor2024formal} that are most relevant to our study of the nerve theorem. The central concept is that of a relative monad, which is a generalisation of a monad that permits the carrier to be an arbitrary tight-cell, rather than an endo-tight-cell. Accordingly, a relative monad is parameterised by a tight-cell $\jAE$ that plays the same role that the identity tight-cell plays for a non-relative monad.

\begin{definition}[{\cite[Definition~4.1]{arkor2024formal}}]
    \label{relative-monad}
    A \emph{relative monad} comprises
    \begin{enumerate}
        \item a tight-cell $\jAE$, the \emph{root};
        \item a tight-cell $t \colon A \to E$, the \emph{carrier};
        \item a 2-cell $\dag \colon E(j, t) \tto E(t, t)$, the \emph{extension operator};
        \item a 2-cell $\eta \colon j \tto t$, the \emph{unit},
    \end{enumerate}
    satisfying associativity and left- and right-unit laws.
    A \emph{$j$-relative monad} (or simply \emph{$j$-monad}) is a relative monad with root $j$.
\end{definition}

Relative monads admit notions of algebra and algebra morphism analogous to those for non-relative monads. However, it is important to observe that, in the context of a \ve{}, the appropriate notion of morphism of algebras is graded by a chain of loose-cells. This permits the consideration of morphisms between algebras with different domains, in contrast to morphisms of algebras in a 2-category~\cite{street1972formal}.

\begin{definition}[{\cite[Definitions~6.1 \& 6.29]{arkor2024formal}}]
	\label{algebra}
    Let $T$ be a relative monad. An \emph{algebra for $T$} (or simply \emph{$T$-algebra}) comprises
    \begin{enumerate}
		\item an object $D$, the \emph{domain};
        \item a tight-cell $e \colon D \to E$, the \emph{carrier};
        \item a 2-cell $\aop \colon E(j, e) \tto E(t, e)$, the \emph{extension operator},
    \end{enumerate}
    satisfying compatibility laws with the extension operator $\dag$ and unit $\eta$ for $T$.

    Let $(e \colon D \to E, \aop)$ and $(e' \colon D' \to E, \aop')$ be $T$-algebras. A \emph{$(p_1, \ldots, p_n)$-graded $T$-algebra morphism} from $(e, \aop)$ to $(e', \aop')$ is a 2-cell \[\epsilon \colon E(1, e), p_1, \ldots, p_n \tto E(1, e')\] satisfying a compatibility law with the two extension operators $\aop$ and $\aop'$.
\end{definition}

\begin{remark}[{\cite[Remark~6.30]{arkor2024formal}}]
    \label{alternative-graded-morphism}
    A $(p_1, \ldots, p_n)$-graded $T$-algebra morphism from $(e, \aop)$ to $(e', \aop')$ is equivalently a 2-cell \[\epsilon \colon p_1, \ldots, p_n \tto E(e, e')\]
    satisfying a compatibility law with the two extension operators $\aop$ and $\aop'$. We shall use the term \emph{algebra morphism} to refer to both forms of 2-cell interchangeably.
\end{remark}

An algebra object for a relative monad $T$ is a universal $T$-algebra. In $\VCat$, this corresponds to the $\V$-category of $T$-algebras~\cite[\S8.3]{arkor2024formal}.

\begin{definition}[{\cite[Definition~6.33]{arkor2024formal}}]
	\label{algebra-object}
	Let $T$ be a relative monad. A $T$-algebra $(u_T \colon \Alg(T) \to E, \aop_T)$ is an
	\emph{algebra object} for $T$ when
	\begin{enumerate}
		\item \label{algebra-object-tight-cell-UP} for every $T$-algebra $(e \colon D \to E, \aop)$, there is a
		unique tight-cell $\unit_{(e, \aop)} \colon D \to \Alg(T)$ such that $\unit_{(e, \aop)} \d u_T = e$ and $\aop_T(1, \unit_{(e, \aop)}) = \aop$;
		\item \label{algebra-object-2-cell-UP} for every graded $T$-algebra morphism $\epsilon \colon E(1, e), p_1, \dots, p_n \tto E(1, e')$, there is a unique 2-cell $\unit_\epsilon \colon \Alg(T)(1, \unit_{(e, \aop)}), p_1, \dots, p_n \tto \Alg(T)(1, \unit_{(e', \aop')})$ such that:
		\[
		\epsilon~=~
		\begin{tikzcd}
			E & D & \cdots & {D'} \\
			E &&& {D'}
			\arrow[""{name=0, anchor=center, inner sep=0}, Rightarrow, no head, from=1-4, to=2-4]
			\arrow["{p_n}"', "\shortmid"{marking}, from=1-4, to=1-3]
			\arrow["{p_1}"', "\shortmid"{marking}, from=1-3, to=1-2]
			\arrow["{E(1, e)}"', "\shortmid"{marking}, from=1-2, to=1-1]
			\arrow["{E(1, e')}", "\shortmid"{marking}, from=2-4, to=2-1]
			\arrow[""{name=1, anchor=center, inner sep=0}, Rightarrow, no head, from=1-1, to=2-1]
			\arrow["{E(1, u_T), \unit_\epsilon}"{description}, draw=none, from=0, to=1]
		\end{tikzcd}
		\]
        \qedhere
	\end{enumerate}
\end{definition}

\begin{definition}[{\cite[Definitions~6.4 \& 6.43]{arkor2024formal}}]
	\label{opalgebra}
    Let $T$ be a relative monad. An \emph{opalgebra for $T$} (or simply \emph{$T$-opalgebra}) comprises
    \begin{enumerate}
		\item an object $B$, the \emph{codomain};
        \item a tight-cell $a \colon A \to B$, the \emph{carrier} or \emph{underlying tight-cell};
        \item a 2-cell $\oop \colon E(j, t) \tto B(a, a)$, the \emph{extension operator},
    \end{enumerate}
    satisfying compatibility laws with the extension operator $\dag$ and unit $\eta$ for $T$.

    Let $(a \colon A \to B, \oop)$ and $(a' \colon A \to B', \oop')$ be $T$-opalgebras. A \emph{$(p_1, \dots, p_n)$-graded $T$-opalgebra morphism} from $(a, \oop)$ to $(a', \oop')$ is a 2-cell
    \[\alpha \colon p_1, \ldots, p_n, B(1, a) \tto B'(1, a')\]
    satisfying a compatibility law with the two extension operators $\oop$ and $\oop'$.
\end{definition}

An opalgebra object for a relative monad $T$ is a universal $T$-opalgebra. In $\VCat$, this corresponds to the Kleisli $\V$-category~\cite[\S8.4]{arkor2024formal}.

\begin{definition}[{\cite[Definition~6.45]{arkor2024formal}}]
	\label{opalgebra-object}
	Let $T$ be a relative monad. A $T$-opalgebra $(k_T \colon A \to \Opalg(T), \oop_T)$ is called an
	\emph{opalgebra object} for $T$ when
	\begin{enumerate}
		\item for every $T$-opalgebra $(a \colon A \to B, \oop)$, there is a unique tight-cell $[]_{(a, \oop)} \colon \Opalg(T) \to B$ such that $k_T \d []_{(a, \oop)} = a$ and $\oop_T \d []_{(a, \oop)} = \oop$;
		\item for every graded $T$-opalgebra morphism $\alpha \colon p_1, \dots, p_n, B(1, a) \tto B'(1, a')$, there is a unique 2-cell $[]_\alpha \colon p_1, \dots, p_n, B(1, []_{(a, \oop)}) \tto B'(1, []_{(a', \oop')})$ such that:
		\[
		\alpha~=~
		\begin{tikzcd}
			{B'} & \cdots & B & A \\
			{B'} &&& A
			\arrow["{B(1, a)}"', "\shortmid"{marking}, from=1-4, to=1-3]
			\arrow[""{name=0, anchor=center, inner sep=0}, Rightarrow, no head, from=1-4, to=2-4]
			\arrow["{B'(1, a')}", "\shortmid"{marking}, from=2-4, to=2-1]
			\arrow[""{name=1, anchor=center, inner sep=0}, Rightarrow, no head, from=1-1, to=2-1]
			\arrow["{p_n}"', "\shortmid"{marking}, from=1-3, to=1-2]
			\arrow["{p_1}"', "\shortmid"{marking}, from=1-2, to=1-1]
			\arrow["{[]_\alpha(1, k_T)}"{description}, draw=none, from=0, to=1]
		\end{tikzcd}
		\]
  \end{enumerate}
\end{definition}

For every \ve{} $\X$, there is a dual \ve{} $\X\co$. Results about relative comonads and their coalgebras will consequently follow from the results in this paper by duality~\cite[\S7]{arkor2024formal}; we shall not state these dual formulations explicitly.

\subsection{Loose-monads}

A loose-monad on an object $A$ in a \vdc{} is a monoid in the multicategory of loose-cells $A \lto A$. Loose-monads in $\VCat$ have also been called \emph{promonads}, \emph{profunctor monads}, or \emph{arrows} (\cf~\cref{arrows}). Loose-monads in a given \vdc{} form a \vdc{}.

\begin{definition}[{\cite[Definition~5.3.1]{leinster2004higher}}]
	\label{LMnd}
    Let $\X$ be a \vdc{}. The \vdc{} $\LMndX$ is defined as follows.
    \begin{enumerate}
        \item An object is a \emph{loose-monad} in $\X$, comprising an object $A \in \X$, a loose-cell $T \colon A \lto A$, and 2-cells
        \[
        \begin{tikzcd}
            A & A & A \\
            A && A
            \arrow[""{name=0, anchor=center, inner sep=0}, Rightarrow, no head, from=1-1, to=2-1]
            \arrow["T"', "\shortmid"{marking}, from=1-2, to=1-1]
            \arrow["T"', "\shortmid"{marking}, from=1-3, to=1-2]
            \arrow[""{name=1, anchor=center, inner sep=0}, Rightarrow, no head, from=1-3, to=2-3]
            \arrow["T", "\shortmid"{marking}, from=2-3, to=2-1]
            \arrow["\mu"{description}, draw=none, from=1, to=0]
        \end{tikzcd}
        \hspace{8em}
        \begin{tikzcd}
            A & A \\
            A & A
            \arrow[""{name=0, anchor=center, inner sep=0}, Rightarrow, no head, from=1-1, to=2-1]
            \arrow[Rightarrow, no head, from=1-2, to=1-1]
            \arrow[""{name=1, anchor=center, inner sep=0}, Rightarrow, no head, from=1-2, to=2-2]
            \arrow["T", "\shortmid"{marking}, from=2-2, to=2-1]
            \arrow["\eta"{description}, draw=none, from=1, to=0]
        \end{tikzcd}
        \]
        subject to the following laws.
        \begin{align*}
            (\mu, 1_A) \d \mu & = (1_A, \mu) \d \mu &
            (\eta, 1_A) \d \mu & = 1_A &
            (1_A, \eta) \d \mu & = 1_A
        \end{align*}
        \item A tight-cell from $(A, S, \mu_A, \eta_A)$ to $(B, T, \mu_B, \eta_B)$ is a \emph{loose-monad morphism}, comprising a tight-cell $f \colon A \to B$ and a 2-cell
        \[\begin{tikzcd}
            A & A \\
            B & B
            \arrow[""{name=0, anchor=center, inner sep=0}, "f"', from=1-1, to=2-1]
            \arrow["S"', "\shortmid"{marking}, from=1-2, to=1-1]
            \arrow[""{name=1, anchor=center, inner sep=0}, "f", from=1-2, to=2-2]
            \arrow["T", "\shortmid"{marking}, from=2-2, to=2-1]
            \arrow["\phi"{description}, draw=none, from=1, to=0]
        \end{tikzcd}\]
        subject to the following laws.
        \begin{align*}
            \mu_A \d f & = (f, f) \d \mu_B &
            \eta_A \d f & = 1_f \d \eta_B
        \end{align*}
        \item A loose-cell is a \emph{loose-monad module}, comprising a loose-cell $p \colon A \lto B$ and 2-cells
        \[
        \begin{tikzcd}
            B & B & A \\
            B && A
            \arrow[""{name=0, anchor=center, inner sep=0}, Rightarrow, no head, from=1-1, to=2-1]
            \arrow["T"', "\shortmid"{marking}, from=1-2, to=1-1]
            \arrow["p"', "\shortmid"{marking}, from=1-3, to=1-2]
            \arrow[""{name=1, anchor=center, inner sep=0}, Rightarrow, no head, from=1-3, to=2-3]
            \arrow["p", "\shortmid"{marking}, from=2-3, to=2-1]
            \arrow["\lambda"{description}, draw=none, from=1, to=0]
        \end{tikzcd}
        \hspace{4em}
        \begin{tikzcd}
            B & A & A \\
            B && A
            \arrow[""{name=0, anchor=center, inner sep=0}, Rightarrow, no head, from=1-1, to=2-1]
            \arrow["p"', "\shortmid"{marking}, from=1-2, to=1-1]
            \arrow["S"', "\shortmid"{marking}, from=1-3, to=1-2]
            \arrow[""{name=1, anchor=center, inner sep=0}, Rightarrow, no head, from=1-3, to=2-3]
            \arrow["p", "\shortmid"{marking}, from=2-3, to=2-1]
            \arrow["\rho"{description}, draw=none, from=1, to=0]
        \end{tikzcd}
        \]
        subject to the following laws.
        \begin{align*}
            (\mu_B, 1_p) \d \lambda & = (1_B, \lambda) \d \lambda &
            (\eta_B, 1_p) \d \lambda & = 1_p \\
            (1_p, \mu_A) \d \rho & = (\rho, 1_A) \d \rho &
            (1_p, \eta_A) \d \rho & = 1_p \\
            (\lambda, 1_A) \d \rho & = (1_B, \rho) \d \lambda
        \end{align*}
        \item A 2-cell is \emph{loose-monad transformation}, comprising a 2-cell in $\X$
        \[\begin{tikzcd}
            {A_0} & \cdots & {A_n} \\
            {B_0} && {B_n}
            \arrow[""{name=0, anchor=center, inner sep=0}, "{f_0}"', from=1-1, to=2-1]
            \arrow["{p_1}"', "\shortmid"{marking}, from=1-2, to=1-1]
            \arrow["{p_n}"', "\shortmid"{marking}, from=1-3, to=1-2]
            \arrow[""{name=1, anchor=center, inner sep=0}, "{f_n}", from=1-3, to=2-3]
            \arrow["q", "\shortmid"{marking}, from=2-3, to=2-1]
            \arrow["\psi"{description}, draw=none, from=1, to=0]
        \end{tikzcd}\]
        subject to the following laws.\footnotemark{}
        \begin{align*}
            (\phi_0, \psi) \d \lambda_p & = (\psi, \phi_n) \d \rho_q \tag{$n = 0$} \\
            (\lambda_{p_1}, 1_{p_2}, \ldots, 1_{p_n}) \d \psi & = (\phi_0, \psi) \d \lambda_q \tag{$n \ge 1$} \\
            (1_{p_1}, \ldots, 1_{p_{n - 1}}, \rho_{p_n}) \d \psi & = (\psi, \phi_n) \d \rho_q \tag{$n \ge 1$} \\
            (1_{p_1}, \ldots, \lambda_{p_{i + 1}}, \ldots, 1_{p_n}) \d \psi & = (1_{p_1}, \ldots, \rho_{p_i}, \ldots, 1_{p_n}) \d \psi \tag{$1 \leq i < n$}
        \end{align*}
        \qedhere
    \end{enumerate}
    \footnotetext{Note that \cite[Definition~5.3.1]{leinster2004higher} is incomplete, as it omits the coherence condition for nullary 2-cells.}
\end{definition}

\begin{remark}
    In \cite[Definition~2.8]{cruttwell2010unified}, in which \vdcs{} are named after their loose-cells rather than their objects, $\LMndX$ is called $\dc{Mod}(\X)$.
\end{remark}

For a fixed object $A \in \X$, we denote by $\LMnd(A)$ the category of loose-monads on $A$~\cite[Definition~2.16]{arkor2024formal}.

\section{The loose-monad associated to a relative monad}
\label{associated-loose-monad}

We now proceed with our pursuit of a more abstract understanding of the nerve theorem. This begins with the understanding of relative monads in terms of loose-monads. It is shown in \cite[Theorem~4.22]{arkor2024formal} that, given a $j$-monad $T$ with carrier $t \colon A \to E$, the loose-cell $E(j, t) \colon A \lto A$ has the structure of a loose-monad. In this section we sharpen this result by showing that if $E(j, t)$ has the structure of a loose-monad, and satisfies a mild compatibility condition, the tight-cell $t \colon A \to E$ has the structure of a $j$-monad. This establishes a strong relationship between a relative monad $T$ and its associated loose-monad $E(j, T)$, which we will build upon in the subsequent sections.

To begin, we shall need an abstract result about monoids in multicategories (\cref{monoid-section-is-monoid}). We will then specialise to the multicategory of endo-loose-cells on a fixed object in a \ve{} to obtain a statement about relative monads (\cref{relative-monad-as-section}).

\begin{definition}
    A morphism $g \colon B \to C$ in a multicategory is \emph{monic} if, for each pair of multimorphisms $f, f' \colon A_1, \ldots, A_n \to B$, we have $(f \d g = f' \d g) \implies (f = f')$.
\end{definition}

\begin{example}
    For each section--retraction pair $s \colon B \rightleftarrows A \cocolon r$ (\ie{} satisfying $(s \d r) = 1_B$), the section $s$ is monic.
\end{example}

\begin{lemma}
    \label{section-monoid-morphism}
    Let $\M$ be a multicategory and let $(A, \mu_A, \eta_A)$ be a monoid in $\M$. Let $B$ be an object together with a morphism $s \colon B \to A$. If $s$ is monic, then there is at most one monoid structure on $B$ for which $s$ is a monoid morphism.
\end{lemma}

\begin{proof}
    Suppose that $B$ is equipped with two monoid structures $(\mu_B, \eta_B)$ and $(\mu'_B, \eta'_B)$ for which $s$ is a monoid morphism, so that the following diagrams commute. Consequently, since $s$ is monic, $\mu_B = \mu'_B$ and $\eta_B = \eta'_B$.
    \[
        \begin{tikzcd}
            & {B,B} \\
            B & {A,A} & B \\
            & A
            \arrow["{\mu_B}"', from=1-2, to=2-1]
            \arrow["{s,s}"{description}, from=1-2, to=2-2]
            \arrow["{\mu'_B}", from=1-2, to=2-3]
            \arrow["s"', from=2-1, to=3-2]
            \arrow["{\mu_A}"{description}, from=2-2, to=3-2]
            \arrow["s", from=2-3, to=3-2]
        \end{tikzcd}
    \hspace{6em}
    \begin{tikzcd}
        B & {} & B \\
        & A
        \arrow["s"', from=1-1, to=2-2]
        \arrow["{\eta_B}"', from=1-2, to=1-1]
        \arrow["{\eta'_B}", from=1-2, to=1-3]
        \arrow["{\eta_A}"{description}, from=1-2, to=2-2]
        \arrow["s", from=1-3, to=2-2]
    \end{tikzcd}
    \]
\end{proof}

In particular, given a section--retraction pair $s \colon B \rightleftarrows A \cocolon r$, if $B$ admits a monoid structure for which $s$ is a monoid morphism, the monoid structure is necessarily given by $B, B \xto{s, s} A, A \xto{\mu_A} A \xto r B$ and ${} \xto{\eta_A} A \xto r B$. It is natural to ask when this canonical structure does indeed form a monoid. With this in mind, we introduce the following definition.

\begin{definition}
    \label{monoid-section}
    Let $(A, \mu_A, \eta_A)$ be a monoid in a multicategory. An \emph{$(A, \mu_A, \eta_A)$-section} comprises a section--retraction pair $s \colon B \rightleftarrows A \cocolon r$ rendering the following diagram commutative.
    \[\begin{tikzcd}
    	& {B,B} \\
    	{A,A} && {A,A} \\
    	A & B & A \\
    	& {}
    	\arrow["{s,s}"', from=1-2, to=2-1]
    	\arrow["{\mu_A}"', from=2-1, to=3-1]
    	\arrow["{s,s}", from=1-2, to=2-3]
    	\arrow["{\mu_A}", from=2-3, to=3-3]
    	\arrow["r"{description}, from=3-1, to=3-2]
    	\arrow["s"{description}, from=3-2, to=3-3]
    	\arrow["{\eta_A}", from=4-2, to=3-1]
    	\arrow["{\eta_A}"', from=4-2, to=3-3]
    \end{tikzcd}\]
\end{definition}

\begin{example}
    Let $(A, \mu_A, \eta_A)$ be a monoid in a multicategory and let $e \colon A \to A$ be an idempotent monoid homomorphism. If $e$ admits a splitting $(r \d s) \colon A \to B \to A$, then $(s, r)$ forms an $(A, \mu_A, \eta_A)$-section.
\end{example}

\begin{example}
    If a multicategory admits a unit $\eta_J \colon {} \to J$, then $J$ carries a canonical monoid structure (\cf~\cite[Proposition~4.12]{arkor2024formal}).
    Up to isomorphism, the only section of this monoid is $J$ itself.
    Indeed, for each section $(s, r)$ of the monoid $J$, the equation $(\eta_J \d r \d s) = \eta_J$ implies that $(r \d s)$ is the identity on $J$, so that $s$ is an isomorphism of monoids.
\end{example}

Intuitively, the two commutativity laws in \cref{monoid-section} are exactly those stating that $s$ is a monoid morphism to $A$ from the tentative monoid structure on $B$. The following proposition shows that, under these assumptions, the tentative monoid structure is indeed a monoid structure.

\begin{proposition}
    \label{monoid-section-is-monoid}
    Let $\M$ be a multicategory and let $(A, \mu_A, \eta_A)$ be a monoid in $\M$. An $(A, \mu_A, \eta_A)$-section $(B, s, r)$ endows $B$ with a unique monoid structure such that $s$ is a monoid morphism.
\end{proposition}

Note that $r$ need not be a monoid morphism (though does preserve the unit, by definition).

\begin{proof}
    By \cref{section-monoid-morphism}, the unit and multiplication are uniquely determined. Define
    \begin{align*}
        \eta_B & \defeq {} \xto{\eta_A} A \xto{r} B &
        \mu_B & \defeq B,B \xto{s, s} A,A \xto{\mu_A} A \xto{r} B
    \end{align*}
    We will show that $(B, \mu_B, \eta_B)$ is a monoid. The unit laws follow from commutativity of the following diagrams, since $s$ is monic.
    \[
    \begin{tikzcd}
        B && {B,B} \\
        & {A, A} & B \\
        A && A
        \arrow["{\eta_B, B}", from=1-1, to=1-3]
        \arrow["s"', from=1-1, to=3-1]
        \arrow["{s, s}"{description}, from=1-3, to=2-2]
        \arrow["{\mu_B}", from=1-3, to=2-3]
        \arrow["{\mu_A}"{description}, from=2-2, to=3-3]
        \arrow["s", from=2-3, to=3-3]
        \arrow["{\eta_A, A}"{description}, from=3-1, to=2-2]
        \arrow[Rightarrow, no head, from=3-1, to=3-3]
    \end{tikzcd}
    \hspace{6em}
    \begin{tikzcd}
        B && {B,B} \\
        & {A, A} & B \\
        A && A
        \arrow["{B, \eta_B}", from=1-1, to=1-3]
        \arrow["s"', from=1-1, to=3-1]
        \arrow["{s, s}"{description}, from=1-3, to=2-2]
        \arrow["{\mu_B}", from=1-3, to=2-3]
        \arrow["{\mu_A}"{description}, from=2-2, to=3-3]
        \arrow["s", from=2-3, to=3-3]
        \arrow["{A, \eta_A}"{description}, from=3-1, to=2-2]
        \arrow[Rightarrow, no head, from=3-1, to=3-3]
    \end{tikzcd}
    \]
    The associativity law follows from commutativity of the following diagram, since $s$ is monic.
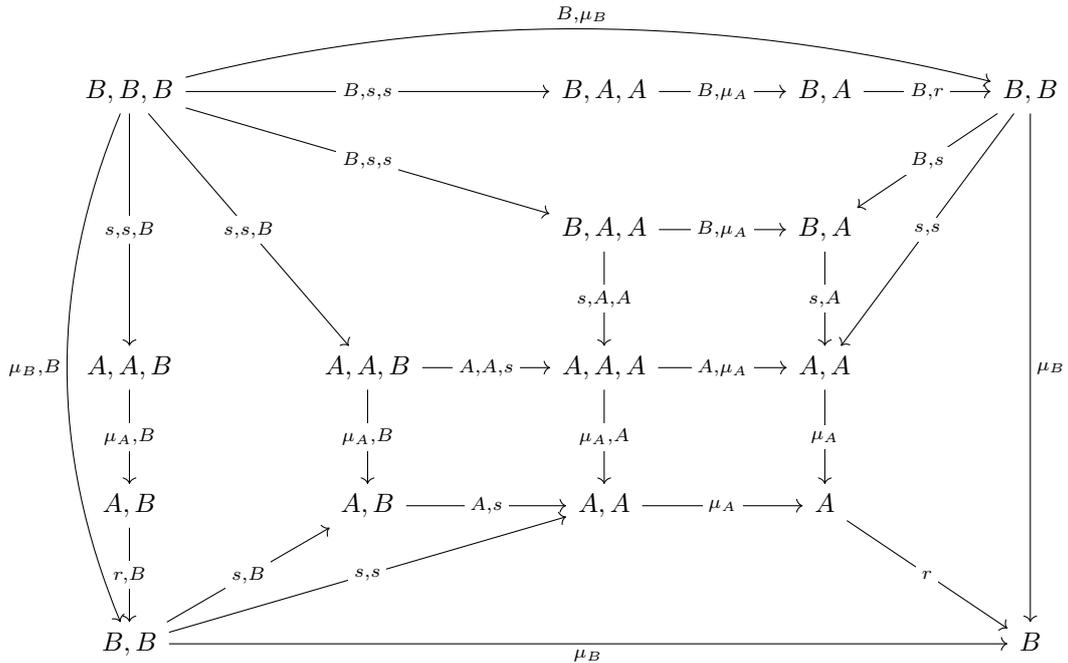
    \[\begin{tikzcd}[sep=large]
        {B,B,B} &&& {B,B} \\
        & {A,A,A} & {A,A} & B \\
        & {A,A} \\
        {B,B} & B && A
        \arrow["{B,\mu_B}", from=1-1, to=1-4]
        \arrow["{s,s,s}"{description}, from=1-1, to=2-2]
        \arrow["{\mu_B,B}"', from=1-1, to=4-1]
        \arrow["{s, s}"{description}, from=1-4, to=2-3]
        \arrow["{\mu_B}", from=1-4, to=2-4]
        \arrow["{A,\mu_A}"{description}, from=2-2, to=2-3]
        \arrow["{\mu_A,A}"{description}, from=2-2, to=3-2]
        \arrow["{\mu_A}"{description}, from=2-3, to=4-4]
        \arrow["s", from=2-4, to=4-4]
        \arrow["{\mu_A}"{description}, from=3-2, to=4-4]
        \arrow["{s,s}"{description}, from=4-1, to=3-2]
        \arrow["{\mu_B}"', from=4-1, to=4-2]
        \arrow["s"', from=4-2, to=4-4]
    \end{tikzcd}\]
    $s$ is then a monoid morphism, the unit law following from the unit section law; and the multiplication law following from the multiplication section law.
\end{proof}

\begin{remark}
     Ralph Sarkis and Jean-Baptiste Vienney have observed that there is a complementary notion of \emph{monoid retract}, which implies that $r$, rather than $s$, is a monoid morphism.
\end{remark}

Recall that, for any object $A$ in a \ve{} $\X$, the endo-loose-cells $A \lto A$ form a multicategory $\X\lh{A, A}$. Every tight-cell $t \colon A \to E$ induces a monoid $E(t, t)$ in this multicategory. The insight relating relative monads to loose-monads is the observation that the laws for a $j$-relative monad with carrier $t$ are precisely the laws for a section of the monoid $E(t, t)$. Thus, every relative monad is a section of an endomorphism monoid. This may be viewed as a Cayley representation for relative monads, though it is sharper than the typical formulations of representation theorems (\cf{}~\cite{rivas2017notions}), as the laws for the submonoid are deduced rather than assumed.

\begin{proposition}
    \label{relative-monad-as-section}
    Let $t \colon A \to E$ be a tight-cell in a \ve{} with chosen restrictions. The following are in bijection.
    \begin{enumerate}
        \item Relative monads with carrier $t$.
        \item $E(t, t)$-sections with strictly $t$-corepresentable carrier.
    \end{enumerate}
\end{proposition}

\begin{proof}
    Let $(B, s, r)$ be an $E(t, t)$-section. Strict $t$-corepresentability means that $B$ is of the form $E(j, t)$ for some tight-cell $\jAE$ (it is this that requires that we have chosen restrictions, since otherwise it is only possible to consider corepresentability up to isomorphism). We have an induced 2-cell ${} \xtto{\pc t} E(t, t) \xtto r E(j, t)$, which is equivalent to a 2-cell $j \tto t$, that we shall call $\eta$. Defining $\dag \defeq s$, we therefore have a section--retraction pair
    \[E(j, t) \xtto\dag E(t, t) \xtto{E(\eta, t)} E(j, t)\]
    which gives us the root, unit, and first unit law for a relative monad.

    For the second unit law, observe that the following diagram commutes (as is particularly clear from the corresponding string diagrams).
    \[\begin{tikzcd}
    	{A(1, 1)} & {E(t, t)} \\
    	{E(j, j)} & {E(j, t)}
    	\arrow["{\pc t}", from=1-1, to=1-2]
    	\arrow["{\pc j}"', from=1-1, to=2-1]
    	\arrow["{E(\eta, t)}", from=1-2, to=2-2]
    	\arrow["{E(j, \eta)}"', from=2-1, to=2-2]
    \end{tikzcd}\]
    Therefore the unit law for the section is the second unit law for a relative monad. For the associativity law, observe that the following diagram commutes.
    \[\begin{tikzcd}
        {B,B} & {A,A} & A & A \\
        {B,A} & {A,A} & A & B
        \arrow["{s,s}", from=1-1, to=1-2]
        \arrow["{B,s}"', from=1-1, to=2-1]
        \arrow["{\mu_A}", from=1-2, to=1-3]
        \arrow[Rightarrow, no head, from=1-2, to=2-2]
        \arrow["{r \d s}", from=1-3, to=1-4]
        \arrow[Rightarrow, no head, from=1-3, to=2-3]
        \arrow["{s,A}"', from=2-1, to=2-2]
        \arrow["{\mu_A}"', from=2-2, to=2-3]
        \arrow["r"', from=2-3, to=2-4]
        \arrow["s"', from=2-4, to=1-4]
    \end{tikzcd}\]
    In our case, the bottom composite is
    \[E(j, t), E(t, t) \xtto{\dag, E(t, t)} E(t, t), E(t, t) \xtto{\cp t(t, t)} E(t, t) \xtto{E(\eta, t)} E(j, t)\]
    which is equal to $\pc t(j, t)$ by the first unit law of the relative monad. Thus, the multiplication law for the section is the associativity law for the relative monad.

    For the converse, every $j$-relative monad with carrier $t$ induces an $E(t, t)$-section $(E(j, t), \dag, E(\eta, t))$.
\end{proof}

Note that we do not recover a relationship between $j$-monad morphisms and loose-monad morphisms this way, since it is $j$ that varies above, rather than $t$. However, we shall not be concerned with the morphisms for what follows.

\begin{remark}
    Dually, for a tight-cell $d \colon Z \to U$, relative comonads are the $U(d, d)$-sections with strictly $d$-representable carriers.
\end{remark}

We observe in passing that, as a consequence of \cref{relative-monad-as-section}, the associativity law in \cite[Definition~4.1]{arkor2024formal} can be replaced by the following equation.
\begin{equation}
\begin{tikzcd}[column sep=large]
    A & A & A \\
    A & A & A \\
    A && A
    \arrow["{E(j, t)}"', "\shortmid"{marking}, from=1-3, to=1-2]
    \arrow["{E(j, t)}"', "\shortmid"{marking}, from=1-2, to=1-1]
    \arrow["{E(t, t)}"{description}, from=2-3, to=2-2]
    \arrow["{E(t, t)}"{description}, from=2-2, to=2-1]
    \arrow[""{name=0, anchor=center, inner sep=0}, Rightarrow, no head, from=1-3, to=2-3]
    \arrow[""{name=1, anchor=center, inner sep=0}, Rightarrow, no head, from=1-2, to=2-2]
    \arrow[""{name=2, anchor=center, inner sep=0}, Rightarrow, no head, from=1-1, to=2-1]
    \arrow["{E(t, t)}", "\shortmid"{marking}, from=3-3, to=3-1]
    \arrow[""{name=3, anchor=center, inner sep=0}, Rightarrow, no head, from=2-3, to=3-3]
    \arrow[""{name=4, anchor=center, inner sep=0}, Rightarrow, no head, from=2-1, to=3-1]
    \arrow["\dag"{description}, draw=none, from=0, to=1]
    \arrow["\dag"{description}, draw=none, from=1, to=2]
    \arrow["{\cp t(t, t)}"{description}, draw=none, from=3, to=4]
\end{tikzcd}
\quad = \quad
\begin{tikzcd}[column sep=large]
	A & A & A \\
	A & A & A \\
	A && A \\
	A && A \\
	A && A
	\arrow["{E(j, t)}"', "\shortmid"{marking}, from=1-3, to=1-2]
	\arrow["{E(j, t)}"', "\shortmid"{marking}, from=1-2, to=1-1]
	\arrow["{E(t, t)}"{description}, from=2-3, to=2-2]
	\arrow["{E(t, t)}"{description}, from=2-2, to=2-1]
	\arrow[""{name=0, anchor=center, inner sep=0}, Rightarrow, no head, from=1-3, to=2-3]
	\arrow[""{name=1, anchor=center, inner sep=0}, Rightarrow, no head, from=1-2, to=2-2]
	\arrow[""{name=2, anchor=center, inner sep=0}, Rightarrow, no head, from=1-1, to=2-1]
	\arrow["{E(t, t)}"{description}, from=3-3, to=3-1]
	\arrow[""{name=3, anchor=center, inner sep=0}, Rightarrow, no head, from=2-3, to=3-3]
	\arrow[""{name=4, anchor=center, inner sep=0}, Rightarrow, no head, from=2-1, to=3-1]
	\arrow["{E(j, t)}"{description}, from=4-3, to=4-1]
	\arrow["{E(t, t)}", "\shortmid"{marking}, from=5-3, to=5-1]
	\arrow[""{name=5, anchor=center, inner sep=0}, Rightarrow, no head, from=3-3, to=4-3]
	\arrow[""{name=6, anchor=center, inner sep=0}, Rightarrow, no head, from=4-3, to=5-3]
	\arrow[""{name=7, anchor=center, inner sep=0}, Rightarrow, no head, from=3-1, to=4-1]
	\arrow[""{name=8, anchor=center, inner sep=0}, Rightarrow, no head, from=4-1, to=5-1]
	\arrow["\dag"{description}, draw=none, from=0, to=1]
	\arrow["\dag"{description}, draw=none, from=1, to=2]
	\arrow["{\cp t(t, t)}"{description}, draw=none, from=3, to=4]
	\arrow["{E(\eta, t)}"{description}, draw=none, from=5, to=7]
	\arrow["\dag"{description}, draw=none, from=6, to=8]
\end{tikzcd}
\end{equation}
For relative monads in $\Cat$, this says that the associativity law is equivalently specified by asking that, given morphisms $f \colon \ob jx \to \ob ty$ and $g \colon \ob jy \to \ob tz$, if we first form the composite $(f^\dag \d g^\dag) \colon \ob tx \to \ob ty \to \ob tz$ by extending both morphisms, then, subsequently, precomposing the unit and then extending does nothing. In other words, the following equation holds.
\[(\eta_x \d f^\dag \d g^\dag)^\dag = f^\dag \d g^\dag\]

\begin{corollary}
    \label{from-monad-to-associated-loose-monad}
    Let $T$ be a $j$-monad. Then $E(j, T)$ is a loose-monad, and $\dag \colon E(j, t) \tto E(t, t)$ is a loose-monad morphism.
\end{corollary}

This statement follows from \cite[Theorem~4.22 \& Example~6.6 \& Lemma~6.7]{arkor2024formal}. However, the formulation of this section facilitates a direct proof.

\begin{proof}
    Direct from \cref{monoid-section-is-monoid,monoid-section-is-monoid}.
\end{proof}

\section{Exact \ve{}s}
\label{exact-vdcs}

In our previous work on the formal theory of relative monads~\cite{arkor2024formal,arkor2024relative}, we worked within the setting of a \ve{} $\X$, without additional global assumptions. In contrast, to establish the nerve theorem formally, we require $\X$ to admit more structure. Conceptually, the relevant structure is an exactness property, corresponding to the existence of certain colimits that interact well with restrictions (which may be thought of as a limit-like notion). While not every \ve{} satisfies this property, it is commonly satisfied in settings of interest for formal category theory. For instance, it is satisfied in every \ve{} constructed from a \vdc{} with restrictions via the loose-monads construction $\L\dc{Mnd}$, which includes \ve{}s of enriched categories, internal categories, generalised multicategories, and so on~\cite{cruttwell2010unified}.

Before introducing the definition of exactness for a \ve{}, we first make a small simplification to our setting, motivated by the observation that, while restrictions in a general \vdc{} are identified up to isomorphism by their universal property, in many examples of interest there is a canonical choice of restrictions that are particularly well behaved. For instance, in the \ve{} $\VCat$ of $\V$-enriched categories, restriction is given by pre- and postcomposition of $\V$-functors, and is consequently strictly functorial (rather than being only pseudofunctorial, as is automatic for restrictions by their universal property). This is convenient, as it means that we may often reason about loose-cells in $\VCat$ up to equality, rather than up to isomorphism. The following definition captures this property.

\begin{definition}
    \label{strict-ve}
    A \ve{} is \emph{strict} if equipped with a strictly functorial choice of restrictions, in the sense that following 2-cells exhibit the chosen cartesian 2-cells (so that $p(1, 1) = p$ and $p(f, g)(f', g') = p(ff', gg')$), for all objects $A, A', B, B''$, tight-cells $f, f', g, g'$, and loose-cell $p$.
    \[
    \begin{tikzcd}[column sep=large]
    	A & B \\
    	A & B
    	\arrow["p", "\shortmid"{marking}, from=2-2, to=2-1]
    	\arrow[""{name=0, anchor=center, inner sep=0}, Rightarrow, no head, from=1-1, to=2-1]
    	\arrow[""{name=1, anchor=center, inner sep=0}, Rightarrow, no head, from=1-2, to=2-2]
    	\arrow["{p(1, 1)}"', "\shortmid"{marking}, from=1-2, to=1-1]
    	\arrow["{=}"{description}, draw=none, from=0, to=1]
    \end{tikzcd}
    \hspace{4em}
    \begin{tikzcd}
    	{A''} && {B''} \\
    	{A'} && {B'} \\
    	A && B
    	\arrow["p", "\shortmid"{marking}, from=3-3, to=3-1]
    	\arrow[""{name=0, anchor=center, inner sep=0}, "f"', from=2-1, to=3-1]
    	\arrow[""{name=1, anchor=center, inner sep=0}, "g", from=2-3, to=3-3]
    	\arrow["{p(f, g)}"{description}, from=2-3, to=2-1]
    	\arrow[""{name=2, anchor=center, inner sep=0}, "{g'}", from=1-3, to=2-3]
    	\arrow[""{name=3, anchor=center, inner sep=0}, "{f'}"', from=1-1, to=2-1]
    	\arrow["{p(ff', gg')}"', "\shortmid"{marking}, from=1-3, to=1-1]
    	\arrow["\cart"{description}, draw=none, from=0, to=1]
    	\arrow["\cart"{description}, draw=none, from=2, to=3]
    \end{tikzcd}
    \]
\end{definition}

We shall henceforth assume our ambient \ve{} $\X$ is strict. While this is not a necessary assumption (that is, our theorems continue to hold, in a weaker sense, without the assumption of strictness), it significantly simplifies reasoning that involves restrictions. For instance, without this assumption, we would have to consider in places universal properties that identify objects only up to equivalence, rather than isomorphism (\cf{}~\cref{non-strict-presheaf-objects}). Given that examples of interest are strict, this seems a reasonable trade-off. Lest the reader worry that we are sacrificing generality for simplicity, we note that strictness is not a restrictive assumption, as every \ve{} is equivalent to a strict one~(\cref{strictification-for-restriction}).

We may now introduce the appropriate notion of exactness for a \ve{}. Note that we shall not assume our ambient \ve{} $\X$ is exact throughout the rest of the paper: exactness will be explicitly assumed when it is needed.

\begin{definition}
    \label{exact-ve}
    A strict \ve{} $\X$ is \emph{exact} when the following conditions hold.
    \begin{enumerate}
        \item For every object $A$ in $\X$ and loose-monad $T$ on $A$, there is an object $\clps T$ and tight-cell $\copi_T \colon A \to \clps T$, the \emph{collapse} of $T$,
        satisfying $T = \clps T(\copi_T, \copi_T)$ and
        such that, for every object $B$ in $\X$ and loose-monad morphism $(f, \phi) \colon T \to B(1, 1)$, there is a unique tight-cell $[]_f \colon \clps T \to B$ factoring $(f, \phi)$.
        \[
        \phi \quad = \quad
        \begin{tikzcd}
        	A & A \\
        	{\clps T} & {\clps T} \\
        	B & B
        	\arrow[""{name=0, anchor=center, inner sep=0}, "{\copi_T}"{description}, from=1-1, to=2-1]
        	\arrow[""{name=1, anchor=center, inner sep=0}, "{\copi_T}"{description}, from=1-2, to=2-2]
        	\arrow["\shortmid"{marking}, Rightarrow, no head, from=2-1, to=2-2]
        	\arrow["T"', "\shortmid"{marking}, from=1-2, to=1-1]
        	\arrow[""{name=2, anchor=center, inner sep=0}, "{[]_f}"{description}, from=2-1, to=3-1]
        	\arrow[""{name=3, anchor=center, inner sep=0}, "{[]_f}"{description}, from=2-2, to=3-2]
        	\arrow["\shortmid"{marking}, Rightarrow, no head, from=3-1, to=3-2]
        	\arrow["f"', curve={height=24pt}, from=1-1, to=3-1]
        	\arrow["f", curve={height=-24pt}, from=1-2, to=3-2]
        	\arrow["\cart"{description}, draw=none, from=0, to=1]
        	\arrow["{=}"{description}, draw=none, from=2, to=3]
        \end{tikzcd}
        \]
        \item For every object $A$ in $\X$, the identity on $1_A$ exhibits the collapse of the loose-identity $A(1, 1)$.
        \item \label{module-collapse} For every loose-monad module $p \colon T' \lto T$ in $\X$, there is a loose-cell $\clps p \colon \clps{T'} \lto \clps T$, the \emph{collapse} of $p$,
        satisfying $p = \clps p(\copi_T, \copi_{T'})$ and
        such that, for every loose-monad transformation,
        \[\begin{tikzcd}
            {T_0} & \cdots & {T_n} \\
            {B(1, 1)} && {B'(1, 1)}
            \arrow[""{name=0, anchor=center, inner sep=0}, "{(f, \phi)}"', from=1-1, to=2-1]
            \arrow["{p_1}"', "\shortmid"{marking}, from=1-2, to=1-1]
            \arrow["{p_n}"', "\shortmid"{marking}, from=1-3, to=1-2]
            \arrow[""{name=1, anchor=center, inner sep=0}, "{(f', \phi')}", from=1-3, to=2-3]
            \arrow["q", "\shortmid"{marking}, from=2-3, to=2-1]
            \arrow["\psi"{description}, draw=none, from=0, to=1]
        \end{tikzcd}\]
        there is a unique 2-cell $[]_\psi$ factoring $\psi$.
        \[
        \psi \quad = \quad
        \begin{tikzcd}[column sep=large]
        	{A_0} & \cdots & {A_n} \\
        	{\clps{T_0}} & \cdots & {\clps{T_n}} \\
        	B && {B'}
        	\arrow["{\copi_{T_0}}"', from=1-1, to=2-1]
        	\arrow[""{name=0, anchor=center, inner sep=0}, "{p_1}"', "\shortmid"{marking}, from=1-2, to=1-1]
        	\arrow[""{name=0p, anchor=center, inner sep=0}, phantom, from=1-2, to=1-1, start anchor=center, end anchor=center]
        	\arrow["\cdots"{description}, draw=none, from=1-2, to=2-2]
        	\arrow[""{name=1, anchor=center, inner sep=0}, "{p_n}"', "\shortmid"{marking}, from=1-3, to=1-2]
        	\arrow[""{name=1p, anchor=center, inner sep=0}, phantom, from=1-3, to=1-2, start anchor=center, end anchor=center]
        	\arrow["{\copi_{T_n}}", from=1-3, to=2-3]
        	\arrow[""{name=2, anchor=center, inner sep=0}, "{[]_f}"', from=2-1, to=3-1]
        	\arrow[""{name=3, anchor=center, inner sep=0}, "{\clps{p_1}}"{description}, from=2-2, to=2-1]
        	\arrow[""{name=3p, anchor=center, inner sep=0}, phantom, from=2-2, to=2-1, start anchor=center, end anchor=center]
        	\arrow[""{name=4, anchor=center, inner sep=0}, "{\clps{p_n}}"{description}, from=2-3, to=2-2]
        	\arrow[""{name=4p, anchor=center, inner sep=0}, phantom, from=2-3, to=2-2, start anchor=center, end anchor=center]
        	\arrow[""{name=5, anchor=center, inner sep=0}, "{[]_{f'}}", from=2-3, to=3-3]
        	\arrow["q", "\shortmid"{marking}, from=3-3, to=3-1]
        	\arrow["\cart"{description}, draw=none, from=0p, to=3p]
        	\arrow["\cart"{description}, draw=none, from=1p, to=4p]
        	\arrow["{[]_\psi}"{description}, draw=none, from=5, to=2]
        \end{tikzcd}
        \]
        \item For every loose-monad $T$ in $\X$, the collapse of $T$ qua a loose-monad exhibits the collapse of $T$ qua an identity loose-monad module $T \lto T$.
	\item \label{loose-cell-between-collapses-is-collapse} Each loose-cell $p \colon \clps{T'} \lto \clps T$ between collapses is equal to the collapse $\clps{p(\copi_T, \copi_{T'})}$ of the loose-monad module $p(\copi_T, \copi_{T'}) \colon T' \lto T$. \qedhere
    \end{enumerate}
\end{definition}

\begin{remark}
    \label{schultz}
    \Cref{exact-ve} is inspired by that of \textcite[Definition~5.1]{schultz2015regular} (which, in turn, is inspired by \citeauthor{wood1985proarrows}'s Axiom 5~\cite[\S2]{wood1985proarrows} and its subsequent study in \cite[\S15 \& \S16]{garner2016enriched}). \citeauthor{schultz2015regular} motivates the terminology by drawing a number of connections between exact \ve{}s and exact categories in the sense of \textcite{barr1971exact}. However, the definition of \citeauthor{schultz2015regular} is too weak, as it specifies a universal property only for unary 2-cells. (The unary property is sufficient for representable \ve{}s, which is \citeauthor{schultz2015regular}'s primary interest.) Note that our universal property \cref{module-collapse} for multiary 2-cells can only be defined in the presence of collapses for \emph{all} loose-monad modules. In contrast, \textcite[Definition~3.9]{schultz2015regular} gives a definition of collapse for an individual loose-monad module that does not require the existence of collapses for other loose-monad modules, but is consequently too weak to prove theorems of interest (\eg{}~\cref{opalgebra-object-from-collapse}).
\end{remark}

\begin{remark}
    Abstractly, \cref{exact-ve} corresponds to asking for the inclusion $\X \to \LMndX$ of $\X$ into the \ve{} of loose-monads and modules in $\X$ -- which sends each object $A$ to the loose-monad $A(1, 1)$ -- to admit a left-adjoint retraction $\clps{{-}} \colon \LMndX \to \X$ that strictly preserves loose-identities and restrictions, and whose unit components are cartesian. This condition is equivalent to asking for $\X$ to be a well-behaved algebra for the $\L\dc{Mnd}$ construction (\cf{}~\cites[Remark~5.15]{cruttwell2010unified}[Proposition~47]{wood1985proarrows}). However, we shall defer this perspective to future work.
\end{remark}

\section{Algebras and opalgebras from exactness}
\label{algebras-and-opalgebras}

In \cref{associated-loose-monad}, we showed that $j$-relative monads are essentially characterised by loose-monads of the form $E(j, T)$. An essential aspect of the theory of relative monads is the theory of algebras and opalgebras~\cite[\S6]{arkor2024formal}. In this section, we shall show that the relationship between a $j$-monad $T$ and its associated loose-monad $E(j, T)$ extends to algebras and opalgebras: in other words, that we may characterise both the $T$-algebras and the $T$-opalgebras in terms of $E(j, T)$. In particular, this characterisation extends to the universal algebras and universal opalgebras, which in $\VCat$ are precisely the notions of \EM{} $\V$-category, and Kleisli $\V$-category. This has several useful consequences: for instance, we shall show that it is possible to construct algebra objects and opalgebra objects in a \ve{} with suitable limits and colimits (\cref{opalgebra-object-from-collapse,semanticiser-theorem}). In particular, the construction of an algebra object in this manner generalises the nerve theorem to settings in which presheaf categories may not exist.

\subsection{Opalgebra objects via collapse}
\label{opalgebras-and-exact-ves}

We begin by considering opalgebras. In addition to being the simpler of the two to characterise, the description of algebra objects in terms of $E(j, T)$ involves the corresponding description of opalgebras.

\begin{lemma}
  \label{opalgebras-via-EjT}
  Let $\jAE$ be a tight-cell and let $T$ be a $j$-monad. For each tight-cell $a \colon A \to B$, there are bijections between
  \begin{itemize}
    \item $T$-opalgebra structures on $a$, and loose-monad morphisms $E(j, T) \tto B(a, a)$;
    \item $(p_1, \ldots, p_n)$-graded $T$-opalgebra morphisms as on the left below, and loose-monad transformations as on the right below.
      \[
\begin{tikzcd}[column sep=1.6em]
	{B'} & \cdots & B & A \\
	{B'} &&& A
	\arrow[Rightarrow, no head, from=1-4, to=2-4]
	\arrow[Rightarrow, no head, from=1-1, to=2-1]
	\arrow["{B'(1, a')}", "\shortmid"{marking}, from=2-4, to=2-1]
	\arrow["{B(1, a)}"', "\shortmid"{marking}, from=1-4, to=1-3]
	\arrow["{p_n}"', "\shortmid"{marking}, from=1-3, to=1-2]
	\arrow["{p_1}"', "\shortmid"{marking}, from=1-2, to=1-1]
\end{tikzcd}
\qquad
\begin{tikzcd}
	{B'(1, 1)} & \cdots & {B(1, 1)} & {E(j, T)} & {B'(1, 1)} \\
	{B'(1, 1)} &&&& {B'(1, 1)}
	\arrow[Rightarrow, no head, from=1-1, to=2-1]
	\arrow["{B(1, a)}"', "\shortmid"{marking}, from=1-4, to=1-3]
	\arrow["{p_n}"', "\shortmid"{marking}, from=1-3, to=1-2]
	\arrow["{p_1}"', "\shortmid"{marking}, from=1-2, to=1-1]
	\arrow["{B'(a', 1)}"', "\shortmid"{marking}, from=1-5, to=1-4]
	\arrow[Rightarrow, no head, from=1-5, to=2-5]
	\arrow["\shortmid"{marking}, Rightarrow, no head, from=2-1, to=2-5]
\end{tikzcd}
\]
  \end{itemize}
\end{lemma}

\begin{proof}
  The correspondence between opalgebras and loose-monad morphisms is \cite[Lemma~6.7]{arkor2024formal}. In the correspondence between opalgebra morphisms and loose-monad transformations, the only nontrivial law is the compatibility between the action of $B(1, a)$ and the action of $B'(a', 1)$, which corresponds to the single compatibility condition for an opalgebra morphism.
\end{proof}

Consequently, it is natural to expect that opalgebra objects for $T$ are characterised by universal loose-monad morphisms from $E(j, T)$. This is where exactness enters the picture (\cref{exact-ve}). The universal properties of opalgebra objects (\cite[Definition~6.45]{arkor2024formal}) and of collapses are closely related. First, the distinction between their one-dimensional universal properties is exactly the respective distinction between loose-monad morphisms of the form on the left below, and loose-monad morphisms of the form on the right below. It is straightforward to confirm that the two forms are in bijection, and so opalgebra objects and collapses satisfy the same one-dimensional universal property.
\[
\begin{tikzcd}
    A & A \\
    A & A
    \arrow["T"', "\shortmid"{marking}, from=1-2, to=1-1]
    \arrow["{B(f, f)}", "\shortmid"{marking}, from=2-2, to=2-1]
    \arrow[""{name=0, anchor=center, inner sep=0}, Rightarrow, no head, from=1-2, to=2-2]
    \arrow[""{name=1, anchor=center, inner sep=0}, Rightarrow, no head, from=1-1, to=2-1]
    \arrow["\phi"{description}, draw=none, from=1, to=0]
\end{tikzcd}
\hspace{8em}
\begin{tikzcd}
    A & A \\
    B & B
    \arrow["T"', "\shortmid"{marking}, from=1-2, to=1-1]
    \arrow["\shortmid"{marking}, Rightarrow, no head, from=2-2, to=2-1]
    \arrow[""{name=0, anchor=center, inner sep=0}, "f", from=1-2, to=2-2]
    \arrow[""{name=1, anchor=center, inner sep=0}, "f"', from=1-1, to=2-1]
    \arrow["\phi"{description}, draw=none, from=1, to=0]
\end{tikzcd}
\]

Second, the distinction between the two-dimensional universal properties is the distinction between a ``one-sided'' universal property in the case of opalgebra objects, which involves a chain of loose-cells stemming from a single opalgebra object, and a ``multivariant'' universal property in the case of collapses, which involves a chain of loose-cells between multiple collapses. Thus, collapses satisfy a stronger universal property than opalgebra objects (indeed, this is the reason we must continue to work with collapses in \cref{algebras-from-collapse}, rather than opalgebra objects, as we shall have need of this stronger universal property). The preceding discussion provides the intuition for the following theorem.

\begin{theorem}
    \label{opalgebra-object-from-collapse}
    Let $\jAE$ be a tight-cell and let $T$ be a $j$-monad. If $\X$ is exact, then the coprojection $\copi_{E(j, T)} \colon A \to \clps{E(j, T)}$ exhibits an opalgebra object for $T$.
\end{theorem}

\begin{proof}
    For convenience, we shall denote by $\copi$ the tight-cell $\copi_{E(j, T)}$. We shall show that the canonical opcartesian 2-cell $E(j, t) \tto \clps{E(j, T)}(\copi, \copi)$ exhibits the opalgebra object.

    By definition, we have that $E(j, T) \tto \clps{E(j, T)}(\copi, \copi)$ is a loose-monad morphism, hence is an opalgebra by \cref{opalgebras-via-EjT}. Let $(a, \oop)$ be a $T$-opalgebra, \ie{} a loose-monad morphism ${E(j, T) \tto B(a, a)}$. The universal property of the collapse gives a unique factorisation of $a$ through $\copi$, and $\oop$ through the cartesian 2-cell defining the collapse, as below on the left. This is equivalent to a unique factorisation of $\oop$ through the induced opcartesian 2-cell (by definition of opcartesianness), as below on the right, and hence exhibits the one-dimensional universal property of an opalgebra object.
    \[
    \begin{tikzcd}[column sep=huge]
    	A & A \\
    	{\clps{E(j, T)}} & {\clps{E(j, T)}} \\
    	B & B
    	\arrow["{E(j, T)}"', "\shortmid"{marking}, from=1-2, to=1-1]
    	\arrow[""{name=0, anchor=center, inner sep=0}, "\copi"{description}, from=1-1, to=2-1]
    	\arrow[""{name=1, anchor=center, inner sep=0}, "\copi"{description}, from=1-2, to=2-2]
    	\arrow[""{name=2, anchor=center, inner sep=0}, "{[]_a}"{description}, from=2-1, to=3-1]
    	\arrow[""{name=3, anchor=center, inner sep=0}, "{[]_a}"{description}, from=2-2, to=3-2]
    	\arrow["\shortmid"{marking}, Rightarrow, no head, from=3-2, to=3-1]
    	\arrow["\shortmid"{marking}, Rightarrow, no head, from=2-2, to=2-1]
    	\arrow["a"', curve={height=40pt}, from=1-1, to=3-1]
    	\arrow["a", curve={height=-40pt}, from=1-2, to=3-2]
    	\arrow["{[]_\oop}"{description}, draw=none, from=2, to=3]
    	\arrow["\cart"{description}, draw=none, from=0, to=1]
    \end{tikzcd}
    \hspace{6em}
    \begin{tikzcd}[column sep=8em]
    	A & A \\
    	A & A \\
    	A & A
    	\arrow["{E(j, T)}"', "\shortmid"{marking}, from=1-2, to=1-1]
    	\arrow[""{name=0, anchor=center, inner sep=0}, Rightarrow, no head, from=1-1, to=2-1]
    	\arrow[""{name=1, anchor=center, inner sep=0}, Rightarrow, no head, from=1-2, to=2-2]
    	\arrow[""{name=2, anchor=center, inner sep=0}, Rightarrow, no head, from=2-1, to=3-1]
    	\arrow[""{name=3, anchor=center, inner sep=0}, Rightarrow, no head, from=2-2, to=3-2]
    	\arrow["{B(a, a)}", "\shortmid"{marking}, from=3-2, to=3-1]
    	\arrow["{\clps{E(j, T)}(\copi, \copi)}"{description}, from=2-2, to=2-1]
    	\arrow[draw=none, from=2, to=3]
    	\arrow["\opcart"{description}, draw=none, from=0, to=1]
    \end{tikzcd}
    \]

    Now let $\alpha \colon p_1, \ldots, p_n, B(1, a) \tto B'(1, a')$ be a $(p_1, \ldots, p_n)$-graded $T$-opalgebra morphism, which is equivalently a loose-monad transformation of the following form by \cref{opalgebras-via-EjT}.
    \[\begin{tikzcd}
    	{B'(1, 1)} & \cdots & {B(1, 1)} & {E(j, T)} & {B'(1, 1)} \\
    	{B'(1, 1)} &&&& {B'(1, 1)}
    	\arrow[Rightarrow, no head, from=1-1, to=2-1]
    	\arrow["{B(1, a)}"', "\shortmid"{marking}, from=1-4, to=1-3]
    	\arrow["{p_n}"', "\shortmid"{marking}, from=1-3, to=1-2]
    	\arrow["{p_1}"', "\shortmid"{marking}, from=1-2, to=1-1]
    	\arrow["{B'(a', 1)}"', "\shortmid"{marking}, from=1-5, to=1-4]
    	\arrow[Rightarrow, no head, from=1-5, to=2-5]
    	\arrow["\shortmid"{marking}, Rightarrow, no head, from=2-1, to=2-5]
    \end{tikzcd}\]
    Exactness thus gives a unique factorisation of $\alpha$ through the cartesian 2-cells defining the collapses of $B(1, a)$ and $B'(a', 1)$, as below.
    \[\begin{tikzcd}[column sep=6em]
    	{B'} & \cdots & B & A & {B'} \\
    	{B'} & \cdots & B & {\clps{E(j, T)}} & {B'} \\
    	{B'} &&&& {B'}
    	\arrow["{B(1, a)}"', "\shortmid"{marking}, from=1-4, to=1-3]
    	\arrow["{p_n}"', "\shortmid"{marking}, from=1-3, to=1-2]
    	\arrow["{p_1}"', "\shortmid"{marking}, from=1-2, to=1-1]
    	\arrow["{B'(a', 1)}"', "\shortmid"{marking}, from=1-5, to=1-4]
    	\arrow["\shortmid"{marking}, Rightarrow, no head, from=3-1, to=3-5]
    	\arrow[Rightarrow, no head, from=2-1, to=3-1]
    	\arrow["{p_n}"', "\shortmid"{marking}, from=2-3, to=2-2]
    	\arrow["{p_1}"', "\shortmid"{marking}, from=2-2, to=2-1]
    	\arrow[""{name=0, anchor=center, inner sep=0}, Rightarrow, no head, from=1-1, to=2-1]
    	\arrow[""{name=1, anchor=center, inner sep=0}, Rightarrow, no head, from=1-3, to=2-3]
    	\arrow["{\clps{B(1, a)}}"{description}, from=2-4, to=2-3]
    	\arrow[Rightarrow, no head, from=2-5, to=3-5]
    	\arrow["{\clps{B'(a', 1)}}"{description}, from=2-5, to=2-4]
    	\arrow[""{name=2, anchor=center, inner sep=0}, Rightarrow, no head, from=1-5, to=2-5]
    	\arrow[""{name=3, anchor=center, inner sep=0}, "\copi"{description}, from=1-4, to=2-4]
    	\arrow["{=}"{description}, draw=none, from=0, to=1]
    	\arrow["\cart"{description}, draw=none, from=1, to=3]
    	\arrow["\cart"{description}, draw=none, from=3, to=2]
    \end{tikzcd}\]
    By \cref{loose-cell-between-collapses-is-collapse}, we have $\clps{B(1, a)} \iso B(1, []_a)$ and $\clps{B'(a', 1)} \iso B'([]_{a'}, 1)$, and so this is equivalent to a factorisation of $\alpha$ through $\copi$ as below, which thus exhibits the collapse as satisfying the two-dimensional universal property of an opalgebra object for $T$.
    \[\begin{tikzcd}
    	{B'} & \cdots & B & {\clps{E(j, T)}} \\
    	{B'} &&& {\clps{E(j, T)}}
    	\arrow["{B'(1, []_{\oop'})}", "\shortmid"{marking}, from=2-4, to=2-1]
    	\arrow[Rightarrow, no head, from=1-1, to=2-1]
    	\arrow["{p_n}"', "\shortmid"{marking}, from=1-3, to=1-2]
    	\arrow["{p_1}"', "\shortmid"{marking}, from=1-2, to=1-1]
    	\arrow["{B(1, []_\oop)}"', "\shortmid"{marking}, from=1-4, to=1-3]
    	\arrow[Rightarrow, no head, from=1-4, to=2-4]
    \end{tikzcd}\]
\end{proof}

As a consequence, every exact \ve{} admits opalgebra objects for relative monads. However, we shall see that, since the universal property of a collapse is stronger than that of an opalgebra object, opalgebra objects in exact equipments are particularly well behaved.

\subsection{Algebra objects via collapse}
\label{algebras-from-collapse}

We now consider algebras. In contrast to the consideration of opalgebras in \cref{opalgebras-and-exact-ves}, it will be necessary here to assume that $j$ is dense. We first observe that $T$-algebras and their graded morphisms may be characterised in terms of the loose-monad $E(j, T)$. In contrast to opalgebras (\cref{opalgebras-via-EjT}), algebras may not be characterised in terms of loose-monad morphisms: instead, they may be characterised in terms of loose-monad modules. (Note that, by taking companions, we could also have characterised $T$-opalgebras in terms of loose-monad modules $E(j, T) \lto B(1, 1)$.)

\begin{lemma}\label{algebras-via-EjT}
  Let $\jAE$ be a dense tight-cell and let $T$ be a $j$-monad. For each tight-cell $e \colon D \to E$, there are bijections between
  \begin{itemize}
    \item $T$-algebra structures on $e$, and loose-monad modules $D(1, 1) \lto E(j, T)$ with carrier $E(j, e) \colon D \lto A$;
    \item $(p_1, \ldots, p_n)$-graded $T$-algebra morphisms as on the left below, and loose-monad transformations as on the right below.
      \[
\begin{tikzcd}
	E & D & \cdots & {D'} \\
	E &&& {D'}
	\arrow[Rightarrow, no head, from=1-1, to=2-1]
	\arrow["{p_n}"', "\shortmid"{marking}, from=1-4, to=1-3]
	\arrow["{p_1}"', "\shortmid"{marking}, from=1-3, to=1-2]
	\arrow["{E(1, e)}"', "\shortmid"{marking}, from=1-2, to=1-1]
	\arrow["{E(1, e')}", "\shortmid"{marking}, from=2-4, to=2-1]
	\arrow[Rightarrow, no head, from=1-4, to=2-4]
\end{tikzcd}
\qquad
\begin{tikzcd}
	{E(j, T)} & {D(1, 1)} & \cdots & {D'(1, 1)} \\
	{E(j, T)} &&& {D'(1, 1)}
	\arrow[Rightarrow, no head, from=1-1, to=2-1]
	\arrow["{p_n}"', "\shortmid"{marking}, from=1-4, to=1-3]
	\arrow["{p_1}"', "\shortmid"{marking}, from=1-3, to=1-2]
	\arrow["{E(j, e)}"', "\shortmid"{marking}, from=1-2, to=1-1]
	\arrow["{E(j, e')}", "\shortmid"{marking}, from=2-4, to=2-1]
	\arrow[Rightarrow, no head, from=1-4, to=2-4]
\end{tikzcd}
\]
  \end{itemize}
\end{lemma}
\begin{proof}
  A $T$-algebra structure on $e$ is a 2-cell $E(j, e) \tto E(t, e)$ compatible with the unit and extension operators of $T$.
  A 2-cell $E(j, e) \tto E(t, e)$ is equivalently a 2-cell $E(1, t), E(j, e) \tto E(1, e)$, and hence, by density of $j$, equivalently a 2-cell $E(j, t), E(j, e) \tto E(j, e)$.
  The unit and extension laws for a $T$-algebra then correspond to the unit and multiplication laws of a loose-monad module.

  Similarly, 2-cells $E(1, e), p_1, \dots, p_n \tto E(1, e')$ are, by density of $j$, in bijection with 2-cells $E(j, e), p_1, \dots, p_n \tto E(j, e')$.
  The former is a $T$-algebra morphism exactly when the latter is a loose-monad transformation.
\end{proof}

Perhaps surprisingly, the asymmetry between algebras and opalgebras means that universal algebras are characterised quite differently (with respect to the associated loose-monad) to universal opalgebras. In fact, universal algebras are characterised in terms of universal opalgebras, together with the following virtual double categorical limit notion.

\begin{definition}
    \label{semanticiser}
    Let $n \colon E \lto A$ be a loose-cell and let $k \colon A \to K$ be a tight-cell.
    A \emph{semanticiser} of $k$ relative to $n$ comprises an object $n \times_A k$, a tight-cell $\pi_1 \colon n \times_A k \to E$, and a loose-cell $\pi_2 \colon n \times_A k \lto K$ satisfying $\pi_2(k, 1) = n(1, \pi_1)$, which is universal in the following sense.
    \[\begin{tikzcd}
    	{n \times_A k} & K \\
    	E & A
    	\arrow["n"', "\shortmid"{marking}, from=2-1, to=2-2]
    	\arrow["k"', from=2-2, to=1-2]
    	\arrow["{\pi_1}"', from=1-1, to=2-1]
    	\arrow["{\pi_2}", "\shortmid"{marking}, from=1-1, to=1-2]
    \end{tikzcd}\]
    \begin{enumerate}
        \item \label{semanticiser-1} For each tight-cell $e \colon \cdot \to E$ and loose-cell $p \colon \cdot \lto K$ such that $p(k, 1) = n(1, e)$, there is a unique tight-cell $\tp{e, p} \colon \cdot \to n \times_A k$ such that $\tp{e, p} \d \pi_1 = e$ and $\pi_2(1, \tp{e, p}) = p$.
        \[\begin{tikzcd}
        	\cdot \\
        	& {n \times_A k} & K \\
        	& E & A
        	\arrow["n"', "\shortmid"{marking}, from=3-2, to=3-3]
        	\arrow["k"', from=3-3, to=2-3]
        	\arrow["{\pi_1}"{description}, from=2-2, to=3-2]
        	\arrow["{\pi_2}"{description}, from=2-2, to=2-3]
        	\arrow["{\tp{e, p}}"{description}, dashed, from=1-1, to=2-2]
        	\arrow["p", "\shortmid"{marking}, curve={height=-12pt}, from=1-1, to=2-3]
        	\arrow["e"', curve={height=12pt}, from=1-1, to=3-2]
        \end{tikzcd}\]
        \item \label{semanticiser-2} For each pair of 2-cells
          \[
            \chi_1 \colon E(1, e), s_1, \dots, s_n \tto E(1, e')
            \qquad
            \chi_2 \colon p, s_1, \dots, s_n \tto p'
            \qquad
            (n \geq 0)
          \]
          that induce the same 2-cell $n(1, e), s_1, \dots, s_n \tto n(1, e')$ by composition with $n$ and $K(k, 1)$ respectively, there is a unique 2-cell \[\tp{\chi_1, \chi_2} \colon ({n \times_A k})(1, \tp{e, p}), s_1, \dots, s_n \tto ({n \times_A k})(1, \tp{e', p'})\] that induces $\chi_1$ and $\chi_2$ by composition with $E(1, \pi_1)$ and $\pi_2$ respectively. \qedhere
    \end{enumerate}
\end{definition}

\begin{remark}
    The name \emph{semanticiser} is intended to be suggestive of the operation of taking the semantics of an algebraic theory in \citeauthor{lawvere1963functorial}'s \ssa{} and in \citeauthor{linton1969outline}'s subsequent generalisation. We will expand upon this connection in future work.
\end{remark}

We observe that \cref{semanticiser} simplifies in examples of interest. First, we must introduce the notion of density for a loose-cell.

\begin{definition}
    \label{dense-loose-cell}
    A loose-cell $p \colon A \lto B$ is \emph{dense} when the identity 2-cell $1_p \colon p \tto p$ exhibits $A(1, 1)$ as the right lift $p \rf p$.
\end{definition}

Observe that a tight-cell $\jAE$ is dense in the sense of \cite[Definition~3.19]{arkor2024formal} if and only if its conjoint $E(j, 1) \colon E \lto A$ is dense in the sense of \cref{dense-loose-cell}.

\begin{lemma}
    \label{semanticiser-density}
    Consider a square of the following shape, such that $\pi_2(k, 1) = n(1, e)$ and in which $n$ is dense.
    \[\begin{tikzcd}
    	{n \times_A k} & K \\
    	E & A
    	\arrow["n"', "\shortmid"{marking}, from=2-1, to=2-2]
    	\arrow["k"', from=2-2, to=1-2]
    	\arrow["{\pi_1}"', from=1-1, to=2-1]
    	\arrow["{\pi_2}", "\shortmid"{marking}, from=1-1, to=1-2]
    \end{tikzcd}\]
    Then \cref{semanticiser-2} is equivalent to the condition that, for every pair of tight-cells $e$ and $e'$ with codomain $E$, every pair of loose-cells $p$ and $p'$ with codomain $K$, and every 2-cell
    \[
        \chi_2 \colon p, s_1, \dots, s_n \tto p'
        \qquad
        (n \geq 0)
    \]
    such that $p(k, 1) = n(1, e)$ and $p'(k, 1) = n(1, e')$, there is a unique 2-cell \[({n \times_A k})(1, \tp{e, p}), s_1, \dots, s_n \tto ({n \times_A k})(1, \tp{e', p'})\] that induces $\chi_2$ by composition with $\pi_2$.
    Consequently, the above square is a semanticiser if and only if it satisfies \cref{semanticiser-1} and $\pi_2$ is dense.
\end{lemma}

\begin{proof}
  Each 2-cell $\chi_2 \colon p, s_1, \dots, s_n \tto p'$ induces a 2-cell $n(1, e), s_1, \dots, s_n \tto n(1, e')$ by postcomposing $K(k, 1)$, but these are in bijection, by density of $n$, with 2-cells $E(1, e), s_1, \dots, s_n \tto E(1, e')$.
  Hence, in \cref{semanticiser-2}, the 2-cell $\chi_2$ uniquely determines $\chi_1$, and the equivalence follows.

  Assuming the square satisfies \cref{semanticiser-1}, it suffices to consider the case in which ${e = e' = \pi_1}$ and $p = p' = \pi_2$; the general case follows by expressing $p$ and $p'$ as restrictions of $\pi_2$.
  The required condition then expresses a bijection between 2-cells $\pi_2, s_1, \dots, s_n \tto \pi_2$ and 2-cells ${s_1, \dots, s_n \tto ({n \times_A k})(1, 1)}$; this bijection is precisely density of $\pi_2$.
\end{proof}

We are now in a position to prove our main result, characterising the algebra object for a relative monad in terms of its associated loose-monad and its opalgebra object $k_T \colon A \to \Opalg(T)$.

\begin{theorem}
    \label{semanticiser-theorem}
    Let $\jAE$ be a dense tight-cell and let $T$ be a $j$-monad.
    If $\X$ is exact, then a tight-cell $u \colon D \to E$ exhibits the algebra object for $T$ if and only if there exists a loose-cell $\pi \colon D \lto \Opalg(T)$ exhibiting the following square as a semanticiser.
    \[\begin{tikzcd}
    	D & {\Opalg(T)} \\
    	E & A
    	\arrow["{E(j, 1)}"', "\shortmid"{marking}, from=2-1, to=2-2]
    	\arrow["{k_T}"', from=2-2, to=1-2]
    	\arrow["u"', from=1-1, to=2-1]
    	\arrow["\pi", "\shortmid"{marking}, from=1-1, to=1-2]
    \end{tikzcd}\]
\end{theorem}

\begin{proof}
  The tight-cell $k_T \colon A \to \Opalg(T)$ exhibits the collapse of $E(j, T)$ by \cref{opalgebra-object-from-collapse}, and, by \cref{algebras-via-EjT}, we can view $T$-algebras and their morphisms equivalently as loose-monad modules and transformations.

  To give a loose-cell $\pi \colon D \lto \Opalg(T)$ such that $\pi(k_T, 1) = E(j, u)$ is thus equivalently to give a $T$-algebra with carrier $u$.
  We show that such a $T$-algebra is the algebra object exactly when the square above is a semanticiser.

  A $j$-representable loose-monad module $E(j, e) \colon D(1, 1) \lto T$ is equivalently a loose-cell $\clps{E(j, e)} \colon D \lto \Opalg(T)$, by taking the collapse.
  Given this data, the universal properties of the semanticiser and of the algebra object both ask for a unique tight-cells $\unit_e \colon \cdot \to D$ such that $e = \unit_d \d u$.
  The semanticiser additionally requires $\pi(1, \unit_e) = \clps{E(j, e)}$, while the algebra object additionally requires $E(j, e)$ to be the module induced by restricting $E(j, u)$ along $\unit_e$.
  These two conditions are equivalent because the collapse of the latter loose-module is $\pi(1, \unit_e)$.

  Let $E(j, e)$ and $E(j, e') \colon D(1, 1) \lto T$ be $j$-representable loose-monad modules.
  A loose-monad transformation $E(j, e), s_1, \dots, s_n \tto E(j, e')$ is equivalently a 2-cell $\chi_2 \colon \clps{E(j, e)}, s_1, \dots, s_n \tto \clps{E(j, e')}$.
  By \cref{semanticiser-density} it follows that the two-dimensional universal properties of the semanticiser and of the algebra object are equivalent.
\end{proof}

\begin{remark}
    From \cref{semanticiser-theorem}, taking $j = 1$ and assuming $\X$ is representable, we recover \cite[Proposition~7]{wood1985proarrows}, in which semanticisers relative to $A(1, 1)$ are called \emph{universal kones}.
\end{remark}

Recall that the opalgebra and algebra objects for a $j$-monad $T$ form initial and terminal resolutions $k_T \jadj v_T$ and $f_T \jadj u_T$ respectively~\cite[Corollary~6.41 \& Corollary~6.51]{arkor2024formal}. As a consequence, they are related by a unique morphism of resolutions $i_T \colon \Opalg(T) \to \Alg(T)$~\cite[Definition~5.23 \& Remark~6.54]{arkor2024formal}, the comparison tight-cell, which renders the following diagram commutative.
\[\begin{tikzcd}[row sep=small]
	& {\Opalg(T)} \\
	A && E \\
	& {\Alg(T)}
	\arrow["{i_T}"{description}, dashed, from=1-2, to=3-2]
	\arrow["{k_T}", from=2-1, to=1-2]
	\arrow["{f_T}"', from=2-1, to=3-2]
	\arrow["{v_T}", from=1-2, to=2-3]
	\arrow["{u_T}"', from=3-2, to=2-3]
\end{tikzcd}\]
What is remarkable about \cref{semanticiser-theorem} is that, in the setting of exact \ve{}s, it exhibits a \emph{joint} universal property of the opalgebra and algebra objects, mediated by the comparison tight-cell. As a consequence, in this setting, $i_T$ is particularly well behaved. We summarise this observation in the following proposition.

\begin{proposition}
    \label{density-of-comparison}
    Let $j \colon A \to E$ be a dense tight-cell and let $T$ be a $j$-monad admitting an algebra object.
    Assume that $\X$ is exact.
    Then the loose-cell $\Alg(T)(i_T, 1) \colon \Alg(T) \lto \Opalg(T)$ is isomorphic to a loose-cell $\pi$ that exhibits the following square as a semanticiser.
    \[\begin{tikzcd}
    	{\Alg(T)} & {\Opalg(T)} \\
    	E & A
    	\arrow["{E(j, 1)}"', "\shortmid"{marking}, from=2-1, to=2-2]
    	\arrow["{k_T}"', from=2-2, to=1-2]
    	\arrow["{u_T}"', from=1-1, to=2-1]
    	\arrow["\pi", "\shortmid"{marking}, from=1-1, to=1-2]
    \end{tikzcd}\]
    Consequently, the comparison tight-cell $i_T \colon \Opalg(T) \to \Alg(T)$ is dense and \ff{}.
\end{proposition}

\begin{proof}
  Since $k_T \colon A \to \Opalg(T)$ exhibits the collapse of the loose-monad $E(j, T)$, the restriction $\Alg(T)(i_Tk_T, 1) = \Alg(T)(f_T, 1)$ forms a loose-monad module $\Alg(T)(f_T, 1) \colon \Alg(T)(1, 1) \lto E(j, T)$, whose collapse is the loose-cell $\Alg(T)(i_T, 1)$.
  There is an isomorphism $\Alg(T)(f_T, 1) \iso E(j, u_T)$ of loose-monad modules, and hence an isomorphism $\Alg(T)(i_T, 1) \iso \clps{\Alg(T)(j, u_T)}$ between their collapses.
  We can therefore take the collapse $\clps{\Alg(T)(j, u_T)}$ for $\pi$ by \cref{semanticiser-theorem}.

  Density of $i_T$ is immediate from density of $\pi$, which follows from \cref{semanticiser-density}.
  For \ffness{} of $i_T$, it suffices by \cite[Corollary~3.28]{arkor2024formal} to exhibit an isomorphism $\Opalg(T)(1, 1) \iso \Alg(T)(i_T, i_T)$, which we do as follows.
  \begin{align*}
    \Opalg(T)(1, 1)
    =
    \clps{E(j, T)}
    =
    \clps{E(j, u_T)(1, i_Tk_T)}
    =
    \clps{E(j, u_T)}(1, i_T)
    \iso
    \Alg(T)(i_T, i_T)
  \end{align*}
\end{proof}

\section{The formal nerve theorem}
\label{the-pullback-theorem}

The semanticiser theorem of \cref{semanticiser-theorem} may be viewed as a formulation of the nerve theorem that involves neither pullbacks nor presheaves. This makes it particularly general, as it holds even in settings in which neither exist. However, it is also a somewhat less convenient formulation, as semanticisers are an unfamiliar notion. In this section, we give sufficient conditions for semanticisers to exist in a \ve{}, showing that it suffices for presheaf objects and pullbacks to exist. This provides the final step in our formal understanding of the nerve theorem, allowing us, in \cref{pullback-theorem}, to capture the algebra object for a relative monad as a pullback of the expected form.

\subsection{Presheaves in \ve{}s}

We begin by defining the notion of a presheaf object in a \ve{}, which classifies loose-cells into an object $A$ by tight-cells into a corresponding object of presheaves $\P A$.

\begin{definition}
  \label{presheaf-object}
  Let $A$ be an object of a strict \ve{} $\X$. A \emph{presheaf object} for $A$ comprises an object $\P A$ and a dense loose-cell $\pi_A \colon \P A \lto A$, such that, for every loose-cell $p \colon X \lto A$, there is a unique tight-cell $\breve p \colon X \to \P A$ satisfying $p = \pi_A(1, \breve p)$.
  \begin{enumerate}
      \item We denote by
      \[\yo_A \defeq \widebreve{A(1, 1)} \colon A \to \P A\]
      the \emph{presheaf embedding}, which is the unique tight-cell satisfying $A(1, 1) = \pi_A(1, \yo_A)$.
      \item For a tight-cell $j \colon A \to E$, we denote by
      \[n_j \defeq \widebreve{E(j, 1)} \colon E \to \P A\]
      the \emph{nerve of $j$}, which is the unique tight-cell satisfying $E(j, 1) = \pi_A(1, n_j)$. (By definition, $\yo_A$ is the nerve of the identity $1_A$.)
      \item For $k \colon A \to B$ a tight-cell between objects admitting presheaf objects, we denote by
      \[k^* \defeq \widebreve{\pi_B(k, 1)} \colon \P B \to \P A\]
      the \emph{restriction\footnotemark{} along $k$}, which is the unique tight-cell satisfying $\pi_B(k, 1) = \pi_A(1, k^*)$. \qedhere
      \footnotetext{Here, \emph{restriction} is used in a different sense than the restriction $p(f, g)$ of a loose-cell along a pair of tight-cells, but there is no ambiguity between the two usages.}
  \end{enumerate}
\end{definition}

\begin{remark}
    \label{presheaf-object-as-limit}
    Observe that \cref{presheaf-object} is essentially a double categorical (or, more precisely, an equipment-theoretic) notion of limit over a diagram with a single object $A$, where the projection $\pi_A \colon \P A \lto A$ is required to be loose, and the unique mediating morphisms $\breve p \colon X \to \P A$ are required to be tight: the difference in tightness between the projections and mediating morphisms is what permits the limit of a one-object diagram to be nontrivial. However, \cref{presheaf-object} is not a limit in the sense of \cite{grandis1999limits} or \cite{patterson2024products}, nor, as far as we are aware, in any other sense defined in the literature. This suggests it may be fruitful to study more general notions of double limit capturing this one. (Note that a general definition would be expected also to impose the two-dimensional universal property of \cref{presheaf-two-dimensional-UP}, which is automatic under our assumptions.)
\end{remark}

\begin{remark}
    \label{non-strict-presheaf-objects}
    By virtue of \cref{Yoneda}, our definition of presheaf object is a strict analogue of the \emph{Yoneda embeddings} of \cite[Definition~1.1.2.2.1]{kern2021categorified}, and of the dual to the \emph{Yoneda morphisms} of \cite[Definition~4.2]{koudenburg2024formal} (the difference in variance is due to \citeauthor{koudenburg2024formal}'s convention for the direction of distributors, which is opposite to ours). \citeauthor{kern2021categorified} and \citeauthor{koudenburg2024formal} do not require uniqueness of the mediating tight-cell $\breve p \colon X \to \P A$, though density of $\pi_A$ implies that it is automatically essentially unique. Our (strict) presheaf objects are thus analogous to 2-categorical limits (\aka{} 2-limits), whereas \citeauthor{kern2021categorified}'s and \citeauthor{koudenburg2024formal}'s (non-strict) presheaf objects are analogous to bicategorical limits (\aka{} bilimits).

    In a strict representable \ve{}, our definition coincides with the \emph{power objects} of \cite[Definition~11.4]{lambert2022double} assuming density of $\P A \lto A$, which are motivated by the \emph{power allegories} of \cite[\S2.41]{freyd1990categories}. Similarly, by \cref{presheaf-via-adjunction}, presheaf objects are strongly related to the \emph{power objects} of \cite[\S2.5]{marmolejo2009duality} and \emph{strict representability} of tight-cells in the sense of \cites[11]{pare2021double}[22]{pare2022horizontal}, though neither definition require density of $\pi_A$ (equivalently of $\yo_A$). Related also are the \emph{Yoneda situations} of \cite[Definition~6]{mellies2008free}, though these do not classify loose-cells, but merely exhibit an embedding of tight-cells $X \to \P A$ into loose-cells $X \lto A$ (\cf~\cite[\S7(2)]{street1978yoneda}).
\end{remark}

\begin{lemma}
    \label{presheaf-hom-is-right-lift}
    Let $A$ be an object admitting a presheaf object. For every pair of loose-cells $p \colon X \lto A$ and $q \colon Y \lto A$, the canonical 2-cell
    \[\begin{tikzcd}
    	& X \\
    	Y && A
    	\arrow[""{name=0, anchor=center, inner sep=0}, "q"', "\shortmid"{marking}, from=2-1, to=2-3]
    	\arrow["p", "\shortmid"{marking}, from=1-2, to=2-3]
    	\arrow["{\P A(\breve p, \breve q)}", "\shortmid"{marking}, from=2-1, to=1-2]
    	\arrow[shorten <=3pt, shorten >=6pt, Rightarrow, from=1-2, to=0]
    \end{tikzcd}\]
    exhibits $\P A(\breve p, \breve q)$ as the right lift $q \rf p$ of $q$ through $p$.
\end{lemma}

\begin{proof}
    We have $\P A(\breve p, \breve q) \iso (\pi_A \rf \pi_A)(\breve p, \breve q) \iso \pi_A(1, \breve q) \rf \pi_A(1, \breve p) = q \rf p$ using density of $\pi_A$, \cite[Lemma~3.6]{arkor2024formal}, and the universal property of the presheaf object.
\end{proof}

Consequently, while \cref{presheaf-object} appears only to specify a one-dimensional universal property, density of the projection $\pi_A \colon \P A \lto A$ implies that presheaf objects further satisfy a two-dimensional universal property, analogous to the two-dimensional universal property of an algebra object (\cf{}~\cite[Definition~6.33]{arkor2024formal}).

\begin{proposition}
    \label{presheaf-two-dimensional-UP}
    Let $A$ be an object admitting a presheaf object. Every 2-cell $\phi$ of the following form
    \[\begin{tikzcd}
    	A & X & \cdots & {X'} \\
    	A &&& {X'}
    	\arrow["q", "\shortmid"{marking}, from=2-4, to=2-1]
    	\arrow[""{name=0, anchor=center, inner sep=0}, Rightarrow, no head, from=1-1, to=2-1]
    	\arrow["p"', "\shortmid"{marking}, from=1-2, to=1-1]
    	\arrow["{p_1}"', "\shortmid"{marking}, from=1-3, to=1-2]
    	\arrow["{p_n}"', "\shortmid"{marking}, from=1-4, to=1-3]
    	\arrow[""{name=1, anchor=center, inner sep=0}, Rightarrow, no head, from=1-4, to=2-4]
    	\arrow["\phi"{description}, draw=none, from=0, to=1]
    \end{tikzcd}\]
    factors uniquely through a 2-cell $\breve\phi$, as below.
    \[\begin{tikzcd}
    	A & {\P A} & X & \cdots & {X'} \\
    	A & {\P A} &&& {X'}
    	\arrow[""{name=0, anchor=center, inner sep=0}, Rightarrow, no head, from=1-2, to=2-2]
    	\arrow["{\P A(1, \breve p)}"', "\shortmid"{marking}, from=1-3, to=1-2]
    	\arrow["{\P A(1, \breve q)}", "\shortmid"{marking}, from=2-5, to=2-2]
    	\arrow["{p_1}"', "\shortmid"{marking}, from=1-4, to=1-3]
    	\arrow["{p_n}"', "\shortmid"{marking}, from=1-5, to=1-4]
    	\arrow[""{name=1, anchor=center, inner sep=0}, Rightarrow, no head, from=1-5, to=2-5]
    	\arrow["{\pi_A}"', "\shortmid"{marking}, from=1-2, to=1-1]
    	\arrow["{\pi_A}", "\shortmid"{marking}, from=2-2, to=2-1]
    	\arrow[""{name=2, anchor=center, inner sep=0}, Rightarrow, no head, from=1-1, to=2-1]
    	\arrow["\breve\phi"{description}, draw=none, from=0, to=1]
    	\arrow["{=}"{description}, draw=none, from=2, to=0]
    \end{tikzcd}\]
\end{proposition}

\begin{proof}
    Since $q \rf p \iso \P A(\breve p, \breve q)$ by \cref{presheaf-hom-is-right-lift}, the factorisation follows directly from the universal property of the right lift~\cite[Definition~3.1]{arkor2024formal}.
\end{proof}

Consequently, the universal property of a presheaf object may be rephrased in terms of a two-dimensional adjointness property.

\begin{lemma}
    \label{presheaf-via-adjunction}
    A dense loose-cell $\pi_A \colon \P A \lto A$ exhibits a presheaf object for $A$ if and only if, for each object $X \in \X$,
    \[\widebreve{\ph} \colon \X\lh{X, A} \rightleftarrows \X[X, \P A] \cocolon \pi_A(1, {-})\]
    exhibits an isomorphism of categories.
\end{lemma}

\begin{proof}
    Suppose that $\pi_A \colon \P A \lto A$ exhibits a presheaf object. That the assignment forms a bijection on objects is immediate from the universal property; that it forms a bijection on morphisms is immediate from \cref{presheaf-two-dimensional-UP}.

    Conversely, suppose that $\pi_A \colon \P A \lto A$ exhibits an isomorphism of categories. The universal property of \cref{presheaf-object} follows trivially.
\end{proof}

The following lemma shows that presheaves interact nicely with restrictions.

\begin{lemma}
    \label{presheaf-restriction}
    Consider a diagram $A \xto f B \xlfrom p C \xfrom g D$. Suppose that $A$ and $B$ admit presheaf objects. Then $\widebreve{p(f, g)} = (g \d \breve p \d f^*) \colon D \to \P A$.
\end{lemma}

\begin{proof}
    We have
    \[
    C(1, f^* \breve p g) = \P A(1, f^*)(1, \breve p g)
    = \pi_B(f, 1)(1, \breve p g)
    = \pi_B(f, \breve p g)
    = \pi_B(1, \breve p)(f, g)
    = p(f, g)
    \]
    hence $f^* \breve p g = \widebreve{p(f, g)}$ using the universal property of the presheaf object.
\end{proof}

We defer an in-depth exploration of the consequences of admitting a presheaf object. However, it is instructive to observe that it is possible to deduce from \cref{presheaf-object} a formal statement of the Yoneda lemma, as follows.

\begin{lemma}
    \label{Yoneda}
    For every object $A$ admitting a presheaf object, there is an isomorphism ${\P A(\yo_A, 1) \iso \pi_A}$.
\end{lemma}

\begin{proof}
    By \cref{presheaf-hom-is-right-lift}, $\P A(\yo_A, 1) \iso \pi_A \rf A(1, 1) \iso \pi_A$.
\end{proof}

In particular, for every tight-cell $\breve p \colon \cdot \to \P A$, which we may view as a presheaf on $A$, we have $\P A(\yo_A, \breve p) \iso \pi_A(1, \breve p) = p$. It will also be useful, for \cref{loose-monads-and-yo-relative-monads}, to observe that nerves are right adjoints relative to the presheaf embedding.

\begin{lemma}
    \label{nerve-is-relative-adjoint}
    Let $\jAE$ be a tight-cell for which $A$ admits a presheaf object. Then $j \radj{\yo_A} n_j$.
\end{lemma}

\begin{proof}
    Using \cref{Yoneda}, we have $E(j, 1) = \pi_A(1, n_j) \iso \P A(\yo_A, n_j)$.
\end{proof}

\subsection{Pullbacks in \ve{}s}

Next, inevitably, we require a notion of pullback in a \ve{}.
The appropriate notion of pullback in a 2-category is a pullback in the underlying 1-category, together with a universal property on 2-cells. A pullback in a \ve{} is similar, except that we require the universal property on 2-cells to be suitably multiary.

\begin{definition}
    \label{pullback-of-2-cells}
    A \emph{pullback} of a cospan of 2-cells $\phi \colon p \tto r \tfrom q \cocolon \psi$ in a \vdc{} is a commuting square of 2-cells
    \[\begin{tikzcd}
    	{\phi \times_r \psi} & q \\
    	p & r
    	\arrow["{\pi_1}"', Rightarrow, from=1-1, to=2-1]
    	\arrow["{\pi_2}", Rightarrow, from=1-1, to=1-2]
    	\arrow["\psi", Rightarrow, from=1-2, to=2-2]
    	\arrow["\phi"', Rightarrow, from=2-1, to=2-2]
    \end{tikzcd}\]
      such that, for each span of 2-cells
      \[
        p \xtfrom{\chi_1} s_1, \dots, s_n \xtto{\chi_2} q
        \qquad(n \geq 0)
      \]
      such that $\chi_1 \d \phi = \chi_2 \d \psi$, there is a unique
      2-cell $\tp{\chi_1, \chi_2} \colon s_1, \dots, s_n \tto \phi \times_r \psi$ rendering the following diagram commutative.
    \[\begin{tikzcd}
    	{s_1, \dots, s_n} \\
    	& {\phi \times_r \psi} & q \\
    	& p & r
    	\arrow["{\pi_1}"', Rightarrow, from=2-2, to=3-2]
    	\arrow["{\pi_2}", Rightarrow, from=2-2, to=2-3]
    	\arrow["\psi", Rightarrow, from=2-3, to=3-3]
    	\arrow["\phi"', Rightarrow, from=3-2, to=3-3]
    	\arrow["{\chi_1}"', curve={height=12pt}, Rightarrow, from=1-1, to=3-2]
    	\arrow["{\chi_2}", curve={height=-12pt}, Rightarrow, from=1-1, to=2-3]
    	\arrow["{\tp{\chi_1, \chi_2}}"{description}, Rightarrow, dashed, from=1-1, to=2-2]
    \end{tikzcd}\]
\end{definition}

Concisely, a pullback of 2-cells is a pullback in the corresponding hom-multicategory of $\X$. Pullbacks of 2-cells interact nicely with restrictions.

\begin{lemma}
  \label{restriction-of-pullback}
  In a \ve{} $\X$, the restriction of any pullback of 2-cells is also a pullback.
\[\begin{tikzcd}[column sep=large]
	{(\phi \times_r \psi)(x, y)} & {q(x, y)} \\
	{p(x, y)} & {r(x, y)}
	\arrow["{\pi_1(x, y)}"', Rightarrow, from=1-1, to=2-1]
	\arrow["{\pi_2(x, y)}", Rightarrow, from=1-1, to=1-2]
	\arrow["{\psi(x, y)}", Rightarrow, from=1-2, to=2-2]
	\arrow["{\phi(x, y)}"', Rightarrow, from=2-1, to=2-2]
\end{tikzcd}\]
\end{lemma}

\begin{proof}
  Since $\X$ admits loose-identities, 2-cells of the form
  $s_1, \dots, s_n \tto t(x, y)$ are in bijection with 2-cells of the form
  $X(1, x), s_1, \dots, s_n, Y(y, 1) \tto t$, natural in $t$.
  The result follows immediately by taking $t = (\phi \times_r \psi)$ and using the universal property.
\end{proof}

\begin{definition}
    \label{pullback-of-tight-cells}
    A \emph{pullback} of a cospan $f \colon A \to C \from B \cocolon g$ of tight-cells in a \ve{} $\X$ is a pullback in the underlying
    1-category of objects and tight-cells of $\X$, as on the left below, such that the induced square on the right below is a pullback of 2-cells.
    \[\begin{tikzcd}
    	{f \times_C g} & B \\
    	A & C
    	\arrow["{\pi_2}", from=1-1, to=1-2]
    	\arrow["{\pi_1}"', from=1-1, to=2-1]
    	\arrow[""{name=0, anchor=center, inner sep=0}, "f"', from=2-1, to=2-2]
    	\arrow["g", from=1-2, to=2-2]
    	\arrow["\lrcorner"{anchor=center, pos=0.125}, draw=none, from=1-1, to=0]
    \end{tikzcd}
    \hspace{6em}
    \begin{tikzcd}
        {(f \times_C g)(1, 1)} & {B(\pi_2, \pi_2)} \\
        {A(\pi_1, \pi_1)} & {C(f\pi_1, g\pi_2)}
        \arrow["{\pc{\pi_1}}"', Rightarrow, from=1-1, to=2-1]
        \arrow["{\pc{\pi_2}}", Rightarrow, from=1-1, to=1-2]
        \arrow["{\pc{g}(\pi_2, \pi_2)}", Rightarrow, from=1-2, to=2-2]
        \arrow["{\pc{f}(\pi_1, \pi_1)}"', Rightarrow, from=2-1, to=2-2]
    \end{tikzcd}\]
\end{definition}

In particular, a pullback in $\X$ is a 2-pullback in the tight 2-category $\tX$, but in general the pullback has a stronger universal property.

\subsection{The nerve theorem}

In the presence of presheaf objects, we can give a construction of semanticisers in terms of pullbacks.

\begin{lemma}
    \label{semanticiser-via-pullback}
    Let $n \colon E \lto A$ be a loose-cell and let $k \colon A \to K$ be a tight-cell for which $A$ and $K$ admit presheaf objects (\cref{presheaf-object}). Then the diagram on the left exhibits a semanticiser (\cref{semanticiser}) if and only if the diagram on the right exhibits a pullback (\cref{pullback-of-tight-cells}).
    \[
    \begin{tikzcd}
    	{n \times_A k} & K \\
    	E & A
    	\arrow["n"', "\shortmid"{marking}, from=2-1, to=2-2]
    	\arrow["k"', from=2-2, to=1-2]
    	\arrow["{\pi_1}"', from=1-1, to=2-1]
    	\arrow["{\pi_2}", "\shortmid"{marking}, from=1-1, to=1-2]
    \end{tikzcd}
    \hspace{4em}
    \begin{tikzcd}
    	{n \times_A k} & {\P K} \\
    	E & {\P A}
    	\arrow["{\pi_1}"', from=1-1, to=2-1]
    	\arrow["{\breve n}"', from=2-1, to=2-2]
    	\arrow["{k^*}", from=1-2, to=2-2]
    	\arrow["{\widebreve{\pi_2}}", from=1-1, to=1-2]
\end{tikzcd}
    \]
\end{lemma}

\begin{proof}
    This is a simple exercise in the use of \cref{presheaf-restriction}. To show the two diagrams satisfy the same one-dimensional universal property, we need (1) their cones to be equivalent; (2) their mediating morphisms to be equivalent.

    For (1), observe that a pair $(e, p)$ of a tight-cell $e \colon \cdot \to E$ and $p \colon \cdot \lto K$ satisfying $n(1, e) = p(k, 1)$ is equivalent to a pair $(e, \breve p)$ of tight-cells $e \colon \cdot \to E$ and $\breve p \colon \cdot \to \P K$ satisfying $e \d \breve n = \breve p \d k^*$ since we have $\widebreve{n(1, e)} = e \d \breve n$ and $\widebreve{p(k, 1)} = \breve p \d k^*$.

    For (2), observe that a tight-cell $\tp{e, p} \colon \cdot \to n \times_A k$ satisfying $\tp{e, p} \d \pi_1 = e$ and $\pi_2(1, \tp{e, p}) = p$ is equivalent to a tight-cell $\tp{e, \breve p} \colon \cdot \to n \times_A k$ satisfying $\tp{e, p} \d \pi_1 = e$ and $\tp{e, \breve p} \d \widebreve{\pi_2} = \breve p$ since, by \cref{presheaf-restriction}, we have $\widebreve{\pi_2(1, \tp{e, p})} = \tp{e, p} \d \widebreve{\pi_2}$.

    To show the two diagrams satisfy the same two-dimensional universal property involves similar reasoning involving (1) equivalence of cones; and (2) equivalence of mediating 2-cells. For (1), observe that the universal property of the semanticiser involves a pair of 2-cells of the following form for each pair $e, e' \colon \cdot \to E$ of tight-cells and each pair $p, p' \colon \cdot \lto K$ of loose-cells,
    \begin{align*}
        \chi_1 & \colon E(1, e), s_1, \ldots, s_n \tto E(1, e') \\
        \chi_2 & \colon p, s_1, \ldots, s_n \tto p'
    \end{align*}
    whereas the universal property of the pullback involves a pair of 2-cells of the following form.
    \begin{align*}
        \chi_1' \colon s_1, \ldots, s_n \tto E(\pi_1, \pi_1) \\
        \chi_2' \colon s_1, \ldots, s_n \tto \P K(\pi_2, \pi_2)
    \end{align*}
    However, by \cref{restriction-of-pullback}, restricting the pullback of the cospan $E(\pi_1, \pi_1) \tto \P A(\breve n \pi_1, k^* \pi_2) \tfrom \P K(\pi_2, \pi_2)$ along $\tp{e, \breve p}$ and $\tp{e', \widebreve{p'}}$ exhibits a pullback of the cospan $E(e, e') \tto \P A(\breve n e, k^* \widebreve{p'}) \tfrom \P K(\breve p, \widebreve{p'})$, which satisfies a universal property involving a pair of 2-cells of the following form.
    \begin{align*}
        \chi_1'' \colon s_1, \ldots, s_n \tto E(e, e') \\
        \chi_2'' \colon s_1, \ldots, s_n \tto \P K(\breve p, \widebreve{p'})
    \end{align*}
    By transposing $e$ and $\breve p$, and using the universal property of $\P K$, the 2-cells $\chi_1''$ and $\chi_2''$ are equivalent to the 2-cells $\chi_1$ and $\chi_2$. Conversely, taking $\tp{e, \breve p}$ and $\tp{e', \widebreve{p'}}$ to be identities in $\chi_1''$ and $\chi_2''$ recovers $\chi_1'$ and $\chi_2'$. (2) then follows by completely analogous reasoning.
\end{proof}

\begin{theorem}[Nerve theorem]
    \label{pullback-theorem}
    Let $\X$ be an exact \ve{}, let $\jAE$ be a dense tight-cell, and let $T$ be a $j$-monad. Suppose that $A$ and $\Opalg(T)$ admit presheaf objects. Then $u \colon D \to E$ exhibits an algebra object for $T$ if and only if there is a loose-cell $\breve p \colon D \to \P(\Opalg(T))$ exhibiting the following diagram as a pullback.
    \[\begin{tikzcd}
    	D & {\P(\Opalg(T))} \\
    	E & {\P A}
    	\arrow["{\breve p}", hook, from=1-1, to=1-2]
    	\arrow["u"', from=1-1, to=2-1]
    	\arrow["{{k_T}^*}", from=1-2, to=2-2]
    	\arrow["{n_j}"', hook, from=2-1, to=2-2]
    \end{tikzcd}\]
    In this case, $\breve p \iso n_{i_T}$, exhibiting the comparison tight-cell $i_T \colon \Opalg(T) \to \Alg(T)$ as dense and \ff{}.
\end{theorem}

\begin{proof}
    Direct from \cref{semanticiser-theorem,semanticiser-via-pullback,density-of-comparison}.
\end{proof}

\begin{remark}
    From \cref{pullback-theorem}, taking $j = 1$, we essentially recover \cites[Theorem~35]{street1974elementary}[Proposition~22]{street1978yoneda}, which establish analogous theorems in the context of 2-categories endowed with attributes and Yoneda structures respectively (there is a strong relationship between Yoneda structures and presheaf objects in virtual equipments, as established in \cite[Theorem~4.24]{koudenburg2024formal}).
\end{remark}

The following establishes that every suitably complete virtual equipment admits algebra objects for relative monads with small domain. In the case of enriched categories, it is weaker than the existence theorem of \cite[Corollary~8.20]{arkor2024formal} (which requires neither smallness nor density), but the assumptions are often convenient to verify in practice.

\begin{corollary}
    Let $\X$ be an exact strict \ve{} and let $\jAE$ be a dense tight-cell. If $A$ admits a presheaf object, and $\X$ admits pullbacks along $n_j \colon E \to \P A$, then a $j$-monad admits an algebra object if its opalgebra object admits a presheaf object.
\end{corollary}

\begin{proof}
    Direct from \cref{pullback-theorem}.
\end{proof}

\begin{remark}
    \Cref{pullback-theorem} involves several concepts that look like colimit notions (opalgebra objects and collapses) and several concepts that look like limit notions (algebra objects, semanticisers, pullbacks, and presheaf objects). In this light, \cref{semanticiser-theorem,pullback-theorem} correspond to a notion of exactness, describing an interaction between certain limits and colimits, while \cref{semanticiser-via-pullback,semanticiser-iff-presheaf-object} corresponds to the construction of one kind of limit from another, akin to the construction of arbitrary limits from products and equalisers. Limits and colimits for \vdcs{} have not yet been defined. However, even in the representable case, these concepts do not appear to be double limits and colimits in the classical sense (\cf{}~\cref{presheaf-object-as-limit}), since they make fundamental use of restrictions. It remains to be seen whether there are natural notions of limit and colimit that capture these concepts.
\end{remark}

\section{The enriched nerve theorem}
\label{loose-monads-in-VCat}

We have now developed the abstract theory necessary to establish the nerve theorem for enriched relative monads. In particular, we shall show that \cref{pullback-theorem} may be applied in the \ve{} $\VCat$ of categories enriched in a monoidal category $\V$ (\cite[Definition~8.1]{arkor2024formal}), thereby obtaining the desired nerve theorem. We make no assumptions on the monoidal category $\V$ except where otherwise specified.

First, we must verify the assumptions of \cref{pullback-theorem} on the \ve{} $\VCat$. Recall that a $\V$-distributor $p \colon A \lto A$ comprises an object $p(x, y) \in \V$ for each $x, y \in \ob A$, along with pre- and postcomposition actions. A $\V$-enriched loose-monad thus comprises
\begin{enumerate}
    \item a $\V$-distributor $T \colon A \lto A$;
    \item a $\V$-natural transformation $\circ \colon T, T \tto T$;
    \item a $\V$-natural transformation $\I \colon {} \tto T$,
\end{enumerate}
satisfying associativity and unitality laws.
Unwrapping the definition, we see that each loose-monad comprises, for each pair of objects $x, y \in \ob A$, an object $T(x, y) \in \V$; for each heteromorphism $f \in T(x, y)$ and heteromorphism $g \in T(y, z)$, a heteromorphism $(g \circ f) \in T(x, z)$; and for each morphism $f \colon x \to y$ in $A$, a heteromorphism $\I_f \in T(x, y)$, subject to unitality and associativity axioms.
This structure strongly resembles that of a $\V$-category. In fact, every $\V$-enriched loose-monad defines a $\V$-category $\clps p$, together with an \ioo{} $\V$-functor $A \to \clps p$.\footnotemark{}
\footnotetext{The converse is also true: one may show that $\V$-enriched loose-monads on $A$ are in bijection with \ioo{} $\V$-functors from $A$ (\cf{}~\cites[Corollary~10.4]{lucyshyn2016enriched}[Theorem~2.3.18]{spivak2017string}). However, we defer a general proof to future work, as we have no need for it here.}

\begin{definition}
    Let $T$ be a $\V$-enriched loose-monad on a $\V$-category $A$. The \emph{collapse} of $T$ is the category $\clps T$ defined by $\ob{\clps T} \defeq \ob A$ and $\clps T(x, y) \defeq T(x, y)$. Denote by $\copi_T \colon A \to \clps T$ the \ioo{} $\V$-functor defined by the unit of $T$.
\end{definition}

As the terminology suggests, the collapse of a $\V$-enriched loose-monad exhibits a collapse in the sense of \cref{exact-ve}.

\begin{proposition}
    \label{VCat-is-exact}
    Let $\V$ be a monoidal category. The \ve{} $\VCat$ is exact.
\end{proposition}

\begin{proof}
    First observe that, for every $\V$-enriched loose-monad $T$, the identity natural transformation on its underlying $\V$-distributor exhibits a cartesian 2-cell, since $\copi_T$ is \ioo{} and the hom-objects of $\clps T$ are defined by $T$.
    \[\begin{tikzcd}
    	A & A \\
    	{\clps T} & {\clps T}
    	\arrow[""{name=0, anchor=center, inner sep=0}, "{\copi_T}"', from=1-1, to=2-1]
    	\arrow[""{name=1, anchor=center, inner sep=0}, "{\copi_T}", from=1-2, to=2-2]
    	\arrow["T"', "\shortmid"{marking}, from=1-2, to=1-1]
    	\arrow["\shortmid"{marking}, Rightarrow, no head, from=2-1, to=2-2]
    	\arrow["\cart"{description}, draw=none, from=0, to=1]
    \end{tikzcd}\]
    Note that $\clps{A(1, 1)} = A(1, 1)$ by definition. Consider a loose-monad morphism $(f, \phi) \colon T \to B(1, 1)$, \ie{} a $\V$-natural transformation
    \[\phi_{x, y} \colon T(x, y) \tto B(\ob f x, \ob f y)\]
    compatible with the multiplication and unit of $T$. The full image factorisation of $f$ into an \ioo{} $\V$-functor followed by a \ff{} $\V$-functor defines a unique factorisation of $(f, \phi)$.

    Now consider a loose-monad module $p \colon T' \lto T$. We define a $\V$-distributor $\clps p \colon \clps{T'} \lto \clps T$ by $\clps p(x, y) \defeq p(x, y)$. The compatibility laws for $p$ correspond exactly to the compatibility laws for $\clps p$. Since $\copi_{T'}$ and $\copi_T$ are identity-on-objects, the identity natural transformation exhibits a cartesian 2-cell.
    \[\begin{tikzcd}
    	A & {A'} \\
    	{\clps T} & {\clps{T'}}
    	\arrow[""{name=0, anchor=center, inner sep=0}, "{\copi_T}"', from=1-1, to=2-1]
    	\arrow[""{name=1, anchor=center, inner sep=0}, "{\copi_{T'}}", from=1-2, to=2-2]
    	\arrow["{\clps p}", "\shortmid"{marking}, from=2-2, to=2-1]
    	\arrow["p"', "\shortmid"{marking}, from=1-2, to=1-1]
    	\arrow["\cart"{description}, draw=none, from=0, to=1]
    \end{tikzcd}\]
    Note that the collapse of the identity loose-monad module corresponding to a loose-monad $T$ is precisely the collapse of $T$. Consider a loose-monad transformation $\psi$ as in \cref{exact-ve}, which is a $\V$-natural transformation $p_1(x_0, x_1), \ldots, p_n(x_{n - 1}, x_n) \tto q(\ob f x_0, \ob{f'} x_n)$ satisfying compatibility laws. It is clear that $\psi$ defines a $\V$-natural transformation framed by the full images as previously described, factoring $\psi$. Finally, consider a loose-monad module $p \colon \clps T \lto \clps{T'}$ between collapses. Since the coprojections $\copi_T$ and $\copi_{T'}$ are \ioo{}, we trivially have $\clps{p(\copi_T, \copi_{T'})} = p$.
\end{proof}

\begin{remark}
    \label{arrows}
    Loose-monads have also been called \emph{promonads} or \emph{profunctor monads}, in accordance with the terminology \emph{profunctors} for what we call distributors. In theoretical computer science, they appear as the \emph{arrows} of \cite[\S4]{hughes2000generalising} (\cf{}~\cite[Definition~4.1]{heunen2006arrows}).

    Modulo strength, loose-monad morphisms from $T$ to $A(1, 1)$ are the \emph{algebras for an arrow} of \cites[Definition~3]{jacobs2006freyd}[Definition~6.5]{jacobs2009categorical}. The collapse appears as the \emph{Freyd category} associated to the arrow $T$ (\cf{}~\cite[Theorem~5.4]{heunen2006arrows}), and the one-dimensional universal property of a collapse is consequently established for the equipment $\Cat$ in \cites[Lemma~7]{jacobs2006freyd}[Lemma~6.1]{jacobs2009categorical}.
\end{remark}

\begin{corollary}[{\cite[Theorem~8.21]{arkor2024formal}}]
    Let $\V$ be a monoidal category. Every relative monad in $\VCat$ admits an opalgebra object.
\end{corollary}

\begin{proof}
    By \cref{VCat-is-exact}, $\VCat$ is exact, so admits opalgebra objects by \cref{opalgebra-object-from-collapse}.
\end{proof}

Next, we verify that the notion of presheaf object in \cref{presheaf-object} does indeed capture the notion of presheaf $\V$-category (\cite[Definition~8.5]{arkor2024formal}).

\begin{proposition}
    \label{VCat-admits-presheaf-objects}
    Let $\V$ be a monoidal category and let $A$ be a small $\V$-category admitting a presheaf $\V$-category. The presheaf $\V$-category $\P A$ is a presheaf object in $\VCat$, for which the Yoneda embedding $\yo_A \colon A \to \P A$ exhibits a presheaf embedding.
\end{proposition}

\begin{proof}
    We define a $\V$-distributor $\pi_A \colon \P A \lto A$ assigning $(a, p) \mapsto p(a)$, with left-action
    \[\{ \circ^p_{a', a} \colon A(a', a) \otimes p(a) \to p(a') \}_{a' \in \ob A, a \in \ob A, p \in \ob{\P A}}\]
    given by the left-action of the presheaf $p$, and right-action given by the universal $\V$-natural transformation associated to the object of $\V$-natural transformations from $p$ to $p'$.
    \[\{ \varpi^{p, p'}_a \colon p(a) \otimes \P A(p, p') \to p'(a) \}_{a \in \ob A, p \in \ob{\P A}, p' \in \ob{\P A}}\]
    The laws for a $\V$-distributor follow from the laws for a presheaf, together with $\V$-naturality. Density of $\pi_A$ follows from \cite[Lemma~8.7]{arkor2024formal}.

    Let $p \colon X \lto A$ be a $\V$-distributor. We aim to define a $\V$-functor $\breve p \colon X \to \P A$. Observe that, in order that $\pi_A$ exhibit a presheaf object, for each $x \in \ob X$ and $a \in \ob A$, we must have $(\pi_A(1, \breve p))(a, x) = \pi_A(a, \ob{\breve p}(x)) = \ob{\breve p}(x)(a)$, with the left-action of $\ob{\breve p}(x)$ given by the left-action of $p$ at $x$. This uniquely determines the action of $\breve p$ on objects. Indeed, this is a valid assignment, as the laws for the $\V$-presheaf then follow from the laws for the $\V$-distributor. For the action of $\breve p$ on hom-objects, observe that, for each pair of objects $x, x' \in \ob X$, we have $\P A(\ob{\breve p}(x), \ob{\breve p}(x')) = \P A(p({-}, x), p({-}, x')) = (p \rf p)(x, x')$ by \cite[Lemma~8.7]{arkor2024formal}. The right-action of $\pi_A$ thus uniquely determines the choice of the morphism $\breve p_{x, x'} \colon X(x, x') \to \P A(\ob{\breve p}(x), \ob{\breve p}(x'))$ in $\V$ to be that corresponding to the canonical $\V$-natural transformation $X(1, 1) \tto p \rf p$.

    That the Yoneda embedding exhibits a presheaf embedding is precisely the Yoneda lemma~\cite[Example~8.6]{arkor2024formal}.
\end{proof}

\begin{remark}
    While presheaf $\V$-categories may exist even on $\V$-categories that are not small, existence of presheaf $\V$-categories is generally a strong size constraint. For instance, when $\V = \Set$, a $\V$-category admits a presheaf $\V$-category (\ie{} a locally small presheaf category) if and only it is essentially small~\cite{foltz1979legitimite,freyd1995size}. We do not know whether there exists an example of a non-thin monoidal category $\V$ together with a $\V$-category $A$ that admits a presheaf $\V$-category $\P A$ but that is not \emph{Morita-small}, \ie{} for which there does not exist a small $\V$-category $A'$ for which $\P A \equiv \P(A')$.
\end{remark}

Finally, we observe that the usual construction of pullbacks of $\V$-functors produces pullbacks in the \ve{} $\VCat$.

\begin{proposition}
    \label{pullbacks-of-enriched-categories}
    Let $\V$ be a monoidal category admitting pullbacks. Then the \ve{} $\VCat$ admits pullbacks.
\end{proposition}

\begin{proof}
  Given a cospan of $\V$-functors $f \colon A \to C \from B \cocolon g$, the apex of the pullback is the $\V$-category $f \times_C g$, whose objects are pairs $(a \in \ob{A}, b \in \ob{B})$ such that $\ob{f}a = \ob{g}b$, and whose hom-objects are given by pullbacks in $\V$:
\[\begin{tikzcd}
	{(f \times_C g)((a, b), (a', b'))} & {B(b, b')} \\
	{A(a, a')} & {C(\ob fa, \ob gb')}
	\arrow["{g_{b, b'}}", from=1-2, to=2-2]
	\arrow[""{name=0, anchor=center, inner sep=0}, "{f_{a, a'}}"', from=2-1, to=2-2]
	\arrow[from=1-1, to=2-1]
	\arrow[from=1-1, to=1-2]
	\arrow["\lrcorner"{anchor=center, pos=0.125}, draw=none, from=1-1, to=0]
\end{tikzcd}\]
  Composition and identities in $f \times_C g$ are uniquely determined by the fact that the projections form a span of $\V$-functors $A \xfrom{\pi_1} f \times_C g \xto{\pi_2} C$: this is the pullback in the 1-category of $\V$-categories and $\V$-functors.
  The actions of the $\V$-functors form a pullback of 2-cells in $\VCat$, since every $\V$-natural transformation $s_1, \dots, s_n \tto (A \times_C B)(1,1)$ is, in particular, a family of morphisms to which we may apply the universal property of pullbacks in $\V$.
\end{proof}

The following is now immediate.

\begin{theorem}
    \label{enriched-pullback-theorem}
    Let $\V$ be a complete left- and right-closed monoidal category and let $\jAE$ be a dense $\V$-functor with small domain. For a $j$-monad $T$, the $\V$-category of $T$-algebras may be constructed as the following pullback in $\VCat$.
    \[\begin{tikzcd}
    	{\Alg(T)} & {\P(\Kl(T))} \\
    	E & {\P A}
    	\arrow["{u_T}"', from=1-1, to=2-1]
    	\arrow["{{k_T}^*}", from=1-2, to=2-2]
    	\arrow[""{name=0, anchor=center, inner sep=0}, "{n_j}"', hook, from=2-1, to=2-2]
    	\arrow[hook, from=1-1, to=1-2]
    	\arrow["\lrcorner"{anchor=center, pos=0.125}, draw=none, from=1-1, to=0]
    \end{tikzcd}\]
    Above, we denote by $\P$ the $\V$-enriched presheaf construction; and by $n_j \colon E \to \P A$ the nerve of $j$, which sends an object $e \in \ob E$ to the restricted Yoneda embedding $E(j{-}, e)$.

    Furthermore, the unlabelled $\V$-functor is $\V$-naturally isomorphic to the nerve of the comparison $\V$-functor $i_T \colon \Kl(T) \to \Alg(T)$, exhibiting it as dense and \ff{}.
\end{theorem}

\begin{proof}
    Follows directly from \cref{pullback-theorem}, using that $\VCat$ is exact (\cref{VCat-is-exact}) and, under the given assumptions, admits presheaf objects for small $\V$-categories (\cite[\S8.1]{arkor2024formal} and \cref{VCat-admits-presheaf-objects}) and pullbacks (\cref{pullbacks-of-enriched-categories}).
\end{proof}

\begin{remark}
    When $A$ is large, presheaf $\V$-categories can no longer be assumed to exist, and so \cref{pullback-theorem} may not be applied. In this case, \cref{semanticiser-theorem} may be used directly.
\end{remark}

\begin{remark}
    If an algebra object for $T$ is already known to exist, the relative monadicity theorem~\cite{arkor2024relative} facilitates an alternative proof of \cref{pullback-theorem}, as was sketched in \cite[Example~5.11]{arkor2024relative}. However, we find this alternative proof strategy to be conceptually dissatisfying, as the assumption that the algebra object exists obscures the fact that the existence of the pullback is both necessary and sufficient for the existence of the algebra object.
\end{remark}

It follows from \cite[Corollary~8.20]{arkor2024formal} that, when $\V$ is complete and closed, $\V$-enriched relative monads with small domains admit algebra objects. \Cref{enriched-pullback-theorem} thus establishes the same result under almost the same conditions, except with the additional assumption that the root be dense (however, the conclusion is also stronger, establishing that $i_T$ is dense).

\section{Loose-monads and \texorpdfstring{$\yo$}{Yoneda}-relative monads}
\label{loose-monads-and-yo-relative-monads}

We conclude by observing that presheaf objects permit loose-monads to be captured by relative monads. In doing so, we give another application of the semanticiser theorem, in characterising the algebras for $\yo$-relative monads.

\begin{theorem}
    \label{yo-relative-monads}
    Let $\X$ be a strict virtual equipment, and let $A$ be an object admitting a presheaf object. There is an isomorphism
    \[\RMnd(\yo_A) \iso \LMnd(A)\]
    between the category of $\yo_A$-monads and the category of loose-monads on $A$.
\end{theorem}

\begin{proof}
    By \cite[Theorem~4.22]{arkor2024formal}, we have the following pullback of categories.
    \[\begin{tikzcd}[column sep=huge]
    	{\RMnd(\yo_A)} & {\LMnd(A)} \\
    	{\X[A, \P A]} & {\X\lh{A, A}}
    	\arrow[""{name=0, anchor=center, inner sep=0}, "{\P A(\yo_A, {-})}"', hook, from=2-1, to=2-2]
    	\arrow["{U_{\yo_A}}"', from=1-1, to=2-1]
    	\arrow["{U_A}", from=1-2, to=2-2]
    	\arrow["{\P A(\yo_A, {-})}", hook, from=1-1, to=1-2]
    	\arrow["\lrcorner"{anchor=center, pos=0.125}, draw=none, from=1-1, to=0]
    \end{tikzcd}\]
    However, by \cref{presheaf-via-adjunction}, the bottom functor is an isomorphism, hence so too is the top functor.
\end{proof}

\begin{remark}
    From \cref{yo-relative-monads}, we recover \cites[Theorem~9]{altenkirch2010monads}[Theorems~5.2 \& 5.4]{altenkirch2015monads} regarding the correspondence between arrows (\cref{arrows}) and monads relative to the Yoneda embedding.
\end{remark}

Given that $\yo$-relative monads correspond to loose-monads, we should hope that it is possible also to characterise their (op)algebra objects in terms of the loose-monads. In an exact equipment, this is indeed possible. First, we observe the following relationship between semanticisers and presheaf objects.

\begin{lemma}
    \label{semanticiser-iff-presheaf-object}
    The following square exhibits a semanticiser if and only if $\pi \colon P \lto B$ exhibits a presheaf object and $\breve r \colon P \to \P A$ exhibits the restriction of $f$.
    \[\begin{tikzcd}
    	P & B \\
    	{\P A} & A
    	\arrow["{\pi_A}"', "\shortmid"{marking}, from=2-1, to=2-2]
    	\arrow["f"', from=2-2, to=1-2]
    	\arrow["{\breve r}"', from=1-1, to=2-1]
    	\arrow["\pi", "\shortmid"{marking}, from=1-1, to=1-2]
    \end{tikzcd}\]
\end{lemma}

\begin{proof}
    Since $\pi_A$ is dense, it suffices by \cref{semanticiser-density} to show that the one-dimensional universal property of the semanticiser is equivalent to the one-dimensional universal property of the presheaf object.

    The universal property of the presheaf object asks that, for each loose-cell $p \colon X \lto B$, there is a unique tight-cell $\breve p \colon X \to P$ such that $p = \pi(1, \breve p)$.

    The universal property of the semanticiser asks that, for each tight-cell $e \colon X \lto \P A$ and loose-cell $p \colon X \lto B$ such that $p(f, 1) = \pi_A(1, e)$, there is a unique tight-cell $\tp{e, p} \colon X \to P$ such that $\tp{e, p} \d \breve r = e$ and $\pi(1, \tp{e, p}) = p$. Note that the cone condition uniquely determines $e = \widebreve{p(f, 1)}$.

    Thus, supposing that $\pi$ exhibits a presheaf object and that $\breve r = f^*$, it follows from \cref{presheaf-restriction} that $\breve p \d f^* = \widebreve{p(f, 1)}$, so that the square exhibits a semanticiser. Conversely, supposing the square exhibits a semanticiser, $\pi$ trivially exhibits a presheaf object. Furthermore, commutativity of the square expresses that $r = \pi_A(1, \breve r) = \pi(f, 1)$ and hence that $\breve r = \widebreve{\pi(f, 1)} = f^*$.
\end{proof}

\begin{corollary}
    \label{algebra-object-for-yo-monad}
    Let $\X$ be an exact \ve{}, let $A$ be an object admitting a presheaf object, and let $T$ be a $\yo_A$-monad. The opalgebra object for $T$ is given by the collapse of the corresponding loose-monad. Furthermore, $T$ admits an algebra object if and only if $\Opalg(T)$ admits a presheaf object, in which case they are isomorphic over $\P A$.
    \[\begin{tikzcd}
    	{\Alg(T)} && {\P(\Opalg(T))} \\
    	& {\P A}
    	\arrow["{{k_T}^*}", from=1-3, to=2-2]
    	\arrow["\iso", from=1-1, to=1-3]
    	\arrow["{u_T}"', from=1-1, to=2-2]
    \end{tikzcd}\]
\end{corollary}

\begin{proof}
    By \cref{yo-relative-monads}, the loose-monad associated to $T$ is given by $\P A(\yo_A, T)$; by \cref{opalgebra-object-from-collapse}, its collapse exhibits the opalgebra object for $T$. The second statement then follows from \cref{semanticiser-theorem}, using that $\yo_A$ is dense by \cref{Yoneda}, together with \cref{semanticiser-iff-presheaf-object} with respect to squares of the following form.
    \[\begin{tikzcd}
    	\cdot & {\Opalg(T)} \\
    	{\P A} & A
    	\arrow["{\pi_A}"', "\shortmid"{marking}, from=2-1, to=2-2]
    	\arrow["{k_T}"', from=2-2, to=1-2]
    	\arrow[from=1-1, to=2-1]
    	\arrow["\shortmid"{marking}, from=1-1, to=1-2]
    \end{tikzcd}\]
\end{proof}

It follows from \cref{yo-relative-monads,from-monad-to-associated-loose-monad} that every $j$-monad $T$ induces an associated $\yo_A$-monad $(T \d n_j)$ by postcomposing $n_j \colon E \to \P A$; alternatively, this can be seen to follow from \cite[Proposition~5.37]{arkor2024formal} using \cref{nerve-is-relative-adjoint}.
\[\begin{tikzcd}
	& E \\
	A && {\P A}
	\arrow["{\yo_A}"', from=2-1, to=2-3]
	\arrow[""{name=0, anchor=center, inner sep=0}, "{n_j}", from=1-2, to=2-3]
	\arrow[""{name=1, anchor=center, inner sep=0}, "j"{description}, from=2-1, to=1-2]
	\arrow[""{name=2, anchor=center, inner sep=0}, "t", curve={height=-18pt}, from=2-1, to=1-2]
	\arrow["\dashv"{anchor=center}, shift right, draw=none, from=1, to=0]
	\arrow["\eta"', shorten <=3pt, shorten >=3pt, Rightarrow, from=1, to=2]
\end{tikzcd}\]

We consequently recover the following generalisation of \cref{algebra-object-for-yo-monad} (\cf{}~\cite[Remark~5.5.7]{arkor2022monadic}). The relevance of this characterisation to the nerve theorem will be expounded elsewhere.

\begin{corollary}
    Let $\X$ be an exact \ve{} and let $\jAE$ be a tight-cell. Suppose that $A$ admits a presheaf object. The $\yo_A$-monad $(T \d n_j)$ admits an algebra object if and only if $\Opalg(T)$ admits a presheaf object, in which case they are isomorphic over $\P A$.
    \[\begin{tikzcd}
    	{\Alg(T \d n_j)} && {\P(\Opalg(T))} \\
    	& {\P A}
    	\arrow["{{k_T}^*}", from=1-3, to=2-2]
    	\arrow["\iso", from=1-1, to=1-3]
    	\arrow["{u_{T \d n_j}}"', from=1-1, to=2-2]
    \end{tikzcd}\]
\end{corollary}

\begin{proof}
    Immediate from \cref{algebra-object-for-yo-monad}, observing that $\Opalg(T \d n_j) \iso \Opalg(T)$ under $A$ by \cite[Remark~6.28]{arkor2024formal}.
\end{proof}

\appendix

\section{Strictification for \ve{}s}
\label{strictification}

Recall that \vdcs{}, their functors, and tight transformations form a 2-category $\VDbl$~\cite[Definition~3.1]{cruttwell2010unified}. An \emph{equivalence of \vdcs{}} is an equivalence in $\VDbl$.

\begin{theorem}
    \label{strictification-for-restriction}
    For every \vdc{} $\X$ admitting restrictions, there is an equivalent \vdc{} $\X'$, with the same underlying category of tight-cells, admitting a strictly functorial choice of restrictions.
\end{theorem}

\begin{proof}
    We define a \vdc{} $\X'$ as follows.
    \begin{enumerate}
        \item The category of objects and tight-cells is given by that of $\X$.
        \item A loose-cell from $A$ to $D$ is a diagram $A \xto g B \xlto p C \xfrom f D$ in $\X$.
        \item A 2-cell with frame
        \[\begin{tikzcd}[column sep=7em]
        	{A_0} & {A_1} & \cdots & {A_{n - 1}} & {A_n} \\
        	{B_0} &&&& {B_n}
        	\arrow["{(f_1, p_1, g_1)}"', "\shortmid"{marking}, from=1-2, to=1-1]
        	\arrow["{(f_2, p_2, g_2)}"', "\shortmid"{marking}, from=1-3, to=1-2]
        	\arrow["{(f_{n - 1}, p_{n - 1}, g_{n - 1})}"', "\shortmid"{marking}, from=1-4, to=1-3]
        	\arrow["{(f_n, p_n, g_n)}"', "\shortmid"{marking}, from=1-5, to=1-4]
        	\arrow["{(f', p', g')}", "\shortmid"{marking}, from=2-5, to=2-1]
        	\arrow["f"', from=1-1, to=2-1]
        	\arrow["g", from=1-5, to=2-5]
        \end{tikzcd}\]
        is a 2-cell with the following frame in $\X$.
        \[\begin{tikzcd}[column sep=7em]
        	{A_0} & {A_1} & \cdots & {A_{n - 1}} & {A_n} \\
        	{B_0} &&&& {B_n}
        	\arrow["{p_1(f_1, g_1)}"', "\shortmid"{marking}, from=1-2, to=1-1]
        	\arrow["{p_2(f_2, g_2)}"', "\shortmid"{marking}, from=1-3, to=1-2]
        	\arrow["{p_{n - 1}(f_{n - 1}, g_{n - 1})}"', "\shortmid"{marking}, from=1-4, to=1-3]
        	\arrow["{p_n(f_n, g_n)}"', "\shortmid"{marking}, from=1-5, to=1-4]
        	\arrow["{p'(f', g')}", "\shortmid"{marking}, from=2-5, to=2-1]
        	\arrow["f"', from=1-1, to=2-1]
        	\arrow["g", from=1-5, to=2-5]
        \end{tikzcd}\]
        \item Composition of 2-cells is given by that in $\X$.
        \item The identity of a loose-cell is given by that in $\X$.
    \end{enumerate}
    The associativity and unitality laws for $\X'$ follow from those for $\X$. There is a canonical functor $\X' \to \X$ sending each loose-cell $(f, p, g)$ to $p(f, g)$, which is the identity on the tight category and \ff{} on 2-cells. It is trivially surjective on loose-cells, since each loose-cell $p \colon A \lto D$ is the image of the diagram $A =\!= A \xlto p D =\!= D$, and thus the canonical functor is an equivalence of \vdcs{}.

    It remains to show that $\X'$ admits a strict choice of restrictions. To do so, observe that the cartesian 2-cell in $\X$ on the left below exhibits a cartesian 2-cell in $\X'$ on the right below.
    \[
    \begin{tikzcd}[column sep=huge]
    	{A_0} & {A_n} \\
    	{B_0} & {B_n}
    	\arrow["{p(f, g)}", "\shortmid"{marking}, from=2-2, to=2-1]
    	\arrow[""{name=0, anchor=center, inner sep=0}, "{f'}"', from=1-1, to=2-1]
    	\arrow[""{name=1, anchor=center, inner sep=0}, "{g'}", from=1-2, to=2-2]
    	\arrow["{p(f'f, g'g)}"', "\shortmid"{marking}, from=1-2, to=1-1]
    	\arrow["\cart"{description}, draw=none, from=0, to=1]
    \end{tikzcd}
    \hspace{4em}
    \begin{tikzcd}[column sep=huge]
    	{A_0} & {A_n} \\
    	{B_0} & {B_n}
    	\arrow["{(f, p, g)}", "\shortmid"{marking}, from=2-2, to=2-1]
    	\arrow[""{name=0, anchor=center, inner sep=0}, "{f'}"', from=1-1, to=2-1]
    	\arrow[""{name=1, anchor=center, inner sep=0}, "{g'}", from=1-2, to=2-2]
    	\arrow["{(f f', p, g g')}"', "\shortmid"{marking}, from=1-2, to=1-1]
    	\arrow["\cart"{description}, draw=none, from=0, to=1]
    \end{tikzcd}
    \]
    It is clear that this choice of restrictions is strictly functorial, being inherited from composition of tight-cells in $\X$.
\end{proof}

One might expect the definition of a strict \ve{} to include also a condition on the loose-identities, \ie{} that composition of loose-cells is strictly unital. However, as the following lemma shows, we may always choose composites in a given \vdc{} in such a way as to make this hold.

\begin{lemma}
    \label{strictification-for-loose-identities}
    Every opcartesian 2-cell,
    \[\begin{tikzcd}
    	{A_0} & \cdots & {A_n} \\
    	{A_0} && {A_n}
    	\arrow["{q_m}"', "\shortmid"{marking}, from=1-3, to=1-2]
    	\arrow["{q_1}"', "\shortmid"{marking}, from=1-2, to=1-1]
    	\arrow["q", "\shortmid"{marking}, from=2-3, to=2-1]
    	\arrow[""{name=0, anchor=center, inner sep=0}, Rightarrow, no head, from=1-3, to=2-3]
    	\arrow[""{name=1, anchor=center, inner sep=0}, Rightarrow, no head, from=1-1, to=2-1]
    	\arrow["\opcart"{description}, draw=none, from=1, to=0]
    \end{tikzcd}\]
    for which the object $A_i$ admits a loose-identity (for some fixed $0 \leq i \leq n$), factors through an opcartesian 2-cell as follows.
    \[\begin{tikzcd}[column sep=large]
    	{A_0} & \cdots & {A_i} & {A_i} & \cdots & {A_n} \\
    	{A_0} & \cdots & {A_i} & {A_i} & \cdots & {A_n} \\
    	{A_0} &&&&& {A_n}
    	\arrow["{q_m}"{description}, from=2-6, to=2-5]
    	\arrow["q", "\shortmid"{marking}, from=3-6, to=3-1]
    	\arrow[""{name=0, anchor=center, inner sep=0}, Rightarrow, no head, from=2-6, to=3-6]
    	\arrow[""{name=1, anchor=center, inner sep=0}, Rightarrow, no head, from=2-1, to=3-1]
    	\arrow["{q_1}"{description}, from=2-2, to=2-1]
    	\arrow["{A(1, 1)}"{description}, from=2-4, to=2-3]
    	\arrow["{q_i}"{description}, from=2-3, to=2-2]
    	\arrow["{q_{i + 1}}"{description}, from=2-5, to=2-4]
    	\arrow[Rightarrow, no head, from=1-4, to=1-3]
    	\arrow[""{name=2, anchor=center, inner sep=0}, Rightarrow, no head, from=1-3, to=2-3]
    	\arrow[""{name=3, anchor=center, inner sep=0}, Rightarrow, no head, from=1-4, to=2-4]
    	\arrow[""{name=4, anchor=center, inner sep=0}, Rightarrow, no head, from=1-1, to=2-1]
    	\arrow[""{name=5, anchor=center, inner sep=0}, Rightarrow, no head, from=1-6, to=2-6]
    	\arrow["{q_m}"', "\shortmid"{marking}, from=1-6, to=1-5]
    	\arrow["{q_{i + 1}}"', "\shortmid"{marking}, from=1-5, to=1-4]
    	\arrow["{q_i}"', "\shortmid"{marking}, from=1-3, to=1-2]
    	\arrow["{q_1}"', "\shortmid"{marking}, from=1-2, to=1-1]
    	\arrow["\opcart"{description}, draw=none, from=1, to=0]
    	\arrow["{=}"{description}, draw=none, from=4, to=2]
    	\arrow["{=}"{description}, draw=none, from=3, to=5]
    	\arrow["\opcart"{description}, draw=none, from=2, to=3]
    \end{tikzcd}\]
\end{lemma}

\begin{proof}
    That the given opcartesian 2-cell factors through \emph{some} 2-cell follows immediately from the universal property of the opcartesian 2-cell defining $A(1, 1)$. That the 2-cell is opcartesian follows from opcartesianness of the given 2-cell.
\end{proof}

We conclude by observing that \cref{strictification-for-restriction} gives a full strictification result for representable \ve{}s.

\begin{corollary}
    \label{strictification-of-pseudo-equipments}
    For every \pdc{} admitting companions and conjoints, there is an equivalent strict double category admitting a strictly functorial choice of companions and conjoints.
\end{corollary}

\begin{proof}
    To every \pdc{} $\X$, there is an equivalent strict double category $\X'$~\cite[Theorem~7.5]{grandis1999limits}. Applying \cref{strictification-for-restriction} then produces an equivalent \pdc{} $\X''$ admitting a strictly functorial choice of restrictions. By \cite[Theorem~4.1]{shulman2008framed}, a (strictly functorial) choice of restrictions in a \pdc{} is equivalent to a (strictly functorial) choice of companions and conjoints. It remains to show that $\X''$ is actually a strict double category. The loose-composite of $(f, p, g)$ and $(f', p', g')$ in $\X''$ is represented by the loose-composite of $p(f, g)$ and $p'(f', g')$ in $\X'$, and so if loose-composition is strict in $\X'$, it is also strict in $\X''$.
\end{proof}

\begin{remark}
    Michael Shulman has suggested an alternative approach to the strictification of pseudo double categories with companions and conjoints~\cite{shulman2022strictness}, in which one modifies the tight category -- equipping each tight-cell with a specified companion and conjoint -- rather than modifying the loose bicategory. However, such a strictification $\X'$ is equivalent to the original pseudo double category $\X$ only in a weak sense, in which the tight 2-categories $\tX$ and $\u{\X'}$ are merely biequivalent rather than being 2-equivalent, so that $\X$ and $\X'$ are not equivalent as objects of $\VDbl$.
\end{remark}

\printbibliography

\end{document}

%% file: core.tex
\newcommand{\ifarticle}[2]{
    \csname@ifclassloaded\endcsname{beamer}{#2}{#1}
}

    \usepackage{xparse} 
    \usepackage{xspace} 
    \usepackage{etoolbox} 
    \usepackage{xpatch} 
    \usepackage{pgffor} 

    \usepackage{amsmath,amssymb,amsthm} 
    \allowdisplaybreaks[1]

    \usepackage{mathtools} 
    \usepackage{mathrsfs} 
    \usepackage{stmaryrd} 
    \usepackage{bbm} 
    \usepackage{quiver} 
    \usepackage{xcolor} 
    \usepackage{scalerel} 
    \usepackage{booktabs} 
    \usepackage{nameref} 
    \usepackage[british]{babel} 
    \usepackage{csquotes} 
    \usepackage[UKenglish]{isodate} 
    \usepackage{microtype} 
    \usepackage[splitrule]{footmisc} 

    \ifarticle{
        \PassOptionsToPackage{hyphens}{url}
        \usepackage[pdfusetitle,linktoc=all,colorlinks,citecolor=cyan,linkcolor=purple,urlcolor=blue]{hyperref} 
        \urlstyle{rm}

        \usepackage[nameinlink,noabbrev,capitalize]{cleveref} 
        \usepackage{footnotebackref} 

        \usepackage[shortlabels]{enumitem} 
        \setlist{topsep=2pt,itemsep=2pt,partopsep=2pt,parsep=2pt} 
    }{}
    \usepackage[style=alphabetic,backref=true,backrefstyle=three,maxnames=9,minalphanames=3]{biblatex} 
    \DeclareUnicodeCharacter{0301}{\TODO[Invalid symbol.]} 

    \DeclareFieldFormat*{citetitle}{\emph{#1}} 
    \DeclareCiteCommand{\citetitle}
        {\boolfalse{citetracker}%
        \boolfalse{pagetracker}%
        \usebibmacro{prenote}}
        {\ifciteindex
            {\indexfield{indextitle}}
            {}%
        \printtext[bibhyperref]{\printfield[citetitle]{labeltitle}}}
        {\multicitedelim}
        {\usebibmacro{postnote}}
    \DeclareCiteCommand*{\citeauthor}
        {\defcounter{maxnames}{99}%
        \defcounter{minnames}{99}%
        \defcounter{uniquename}{2}%
        \boolfalse{citetracker}%
        \boolfalse{pagetracker}%
        \usebibmacro{prenote}}
        {\ifciteindex{\indexnames{labelname}}{}%
        \printnames{labelname}}
        {\multicitedelim}
        {\usebibmacro{postnote}}

    \usepackage{calc} 
    \usepackage{tabularray} 
    \usepackage{ebproof} 
    \usepackage{forest} 
    \usepackage{adjustbox} 
    \usepackage{float} 
    \usepackage{stfloats}

    \makeatletter
    \ifarticle{
        \xpretocmd{\@adminfootnotes}{\let\@makefntext\BHFN@OldMakefntext}{}{}
        \renewcommand\@makefntext[1]{%
        \@ifundefined{@makefnmark}
            {}
            {%
            \renewcommand\@makefnmark{%
            \mbox{%
                \textsuperscript{%
                \normalfont
                \hyperref[\BackrefFootnoteTag]{\@thefnmark}%
                }%
            }\,%
            }%
            \BHFN@OldMakefntext{#1}%
        }%
        }

        \LetLtxMacro{\BHFN@Old@footnotemark}{\@footnotemark}
        \renewcommand*{\@footnotemark}{%
            \refstepcounter{BackrefHyperFootnoteCounter}%
            \xdef\BackrefFootnoteTag{bhfn:\theBackrefHyperFootnoteCounter}%
            \label{\BackrefFootnoteTag}%
            \BHFN@Old@footnotemark
        }

        \makeatletter
        \def\paragraph{\@startsection{paragraph}{4}%
          \z@\z@{-\fontdimen2\font}%
          {\normalfont\bfseries}}
        \makeatother
    }{}
    \makeatother

    \ifarticle{
        \theoremstyle{plain}
        \newtheorem{theorem}{Theorem}[section]
        \newtheorem{proposition}[theorem]{Proposition}
        \newtheorem{lemma}[theorem]{Lemma}
        \newtheorem{corollary}[theorem]{Corollary}

        \newtheorem*{theorem*}{Theorem}
        \newtheorem*{corollary*}{Corollary}

        \theoremstyle{definition}
        \newtheorem{definition}[theorem]{Definition}
        
        \newtheorem{example}[theorem]{Example}

        \newtheorem{remark}[theorem]{Remark}

        \newenvironment{sketch}{\proof}{\endproof}

        \Crefname{theoremenumi}{Theorem}{Theorems}
        \AtBeginEnvironment{theorem}{%
            \crefalias{enumi}{theoremenumi}%
            \setlist[enumerate,1]{
                ref={\csname thetheorem\endcsname.(\arabic*)}
            }%
            \crefalias{enumii}{theoremenumi}%
            \setlist[enumerate,2]{
                ref={\thetheorem.(\arabic*).(\alph*)}
            }%
        }

        \Crefname{propositionenumi}{Proposition}{Propositions}
        \AtBeginEnvironment{proposition}{%
            \crefalias{enumi}{propositionenumi}%
            \setlist[enumerate,1]{
                ref={\csname theproposition\endcsname.(\arabic*)}
            }%
            \crefalias{enumii}{propositionenumi}%
            \setlist[enumerate,2]{
                ref={\theproposition.(\arabic*).(\alph*)}
            }%
        }

        \Crefname{lemmaenumi}{Lemma}{Lemmas}
        \AtBeginEnvironment{lemma}{%
            \crefalias{enumi}{lemmaenumi}%
            \setlist[enumerate,1]{
                ref={\csname thelemma\endcsname.(\arabic*)}
            }%
            \crefalias{enumii}{lemmaenumi}%
            \setlist[enumerate,2]{
                ref={\thelemma.(\arabic*).(\alph*)}
            }%
        }

        \Crefname{corollaryenumi}{Corollary}{Corollaries}
        \AtBeginEnvironment{corollary}{%
            \crefalias{enumi}{corollaryenumi}%
            \setlist[enumerate,1]{
                ref={\csname thecorollary\endcsname.(\arabic*)}
            }%
            \crefalias{enumii}{corollaryenumi}%
            \setlist[enumerate,2]{
                ref={\thecorollary.(\arabic*).(\alph*)}
            }%
        }

        \Crefname{definitionenumi}{Definition}{Definitions}
        \AtBeginEnvironment{definition}{%
            \crefalias{enumi}{definitionenumi}%
            \setlist[enumerate,1]{
                ref={\csname thedefinition\endcsname.(\arabic*)}
            }%
            \crefalias{enumii}{definitionenumi}%
            \setlist[enumerate,2]{
                ref={\thedefinition.(\arabic*).(\alph*)}
            }%
        }

        \Crefname{exampleenumi}{Example}{Examples}
        \AtBeginEnvironment{example}{%
            \crefalias{enumi}{exampleenumi}%
            \setlist[enumerate,1]{
                ref={\csname theexample\endcsname.(\arabic*)}
            }%
            \crefalias{enumii}{exampleenumi}%
            \setlist[enumerate,2]{
                ref={\theexample.(\arabic*).(\alph*)}
            }%
        }

        \foreach \env in {definition,remark,example,notation}{
        \AtBeginEnvironment{\env}{%
          \pushQED{\qed}%
        }
        \AtEndEnvironment{\env}{\popQED\endexample}
        }
    }{}

    \DeclareDocumentCommand{\mathcommand}{mO{0}m}{%
    \expandafter\let\csname old\string#1\endcsname=#1
    \expandafter\newcommand\csname new\string#1\endcsname[#2]{#3}
    \DeclareRobustCommand#1{%
        \ifmmode
        \expandafter\let\expandafter\next\csname new\string#1\endcsname
        \else
        \expandafter\let\expandafter\next\csname old\string#1\endcsname
        \fi
        \next
    }%
    }

    \mathcommand{\h}{\textup{-}}

    \newcommand{\tx}{\mathrm}
    \mathcommand{\b}{\mathbf}
    \newcommand{\s}{\mathsf}
    \newcommand{\cl}{\mathcal}
    \mathcommand{\bb}{\mathbb}
    \DeclareMathAlphabet{\bbn}{U}{bbold}{m}{n}
    \newcommand{\dc}[1]{{\b{\bb#1}}}
    
    \mathcommand{\sf}{\mathsf}
    \mathcommand{\u}{\underline}

    \newcommand{\TODO}[1][TODO]{\textcolor{orange}{\textup{#1}}\xspace}

    \newcommand{\flip}[1]{\text{\rotatebox[origin=c]{-180}{$#1$}}}

    \newcommand{\datetoday}{\date{\cleanlookdateon\today}}

    \newcommand{\defeq}{\mathrel{:=}}
    \mathcommand{\d}{\mathbin{;}}
    \mathcommand{\c}{\circ}
    \newcommand{\ph}[1][]{{({-}_{#1})}}
    
    \newcommand{\iso}{\cong}
    \renewcommand{\equiv}{\simeq}

    \newcommand{\from}{\leftarrow}
    \newcommand{\xto}{\xrightarrow}
    \newcommand{\xfrom}{\xleftarrow}
    \newcommand{\tto}{\Rightarrow}
    \newcommand{\tfrom}{\Leftarrow}
    \newcommand{\xtto}{\xRightarrow}
    \newcommand{\xtfrom}{\xLeftarrow}
    
    \newcommand{\ffto}{\hookrightarrow}
    
    \newcommand*\cocolon{%
            \nobreak
            \mskip6mu plus1mu
            \mathpunct{}%
            \nonscript
            \mkern-\thinmuskip
            {:}%
            \mskip2mu
            \relax
    }

    \makeatletter
    \def\slashedarrowfill@#1#2#3#4#5{%
    $\m@th\thickmuskip0mu\medmuskip\thickmuskip\thinmuskip\thickmuskip
    \relax#5#1\mkern-7mu%
    \cleaders\hbox{$#5\mkern-2mu#2\mkern-2mu$}\hfill
    \mathclap{#3}\mathclap{#2}%
    \cleaders\hbox{$#5\mkern-2mu#2\mkern-2mu$}\hfill
    \mkern-7mu#4$%
    }
    \def\rightslashedarrowfill@{%
    \slashedarrowfill@\relbar\relbar\mapstochar\rightarrow}
    \newcommand\xslashedrightarrow[2][]{%
    \ext@arrow 0055{\rightslashedarrowfill@}{#1}{#2}}
    \def\leftslashedarrowfill@{%
    \slashedarrowfill@\leftarrow\relbar\mapsfromchar\relbar}
    \newcommand\xslashedleftarrow[2][]{%
    \ext@arrow 0055{\leftslashedarrowfill@}{#1}{#2}}
    \makeatother

    \newcommand{\xlto}{\xslashedrightarrow}
    \newcommand{\lto}{\xlto{}}
    \newcommand{\xlfrom}{\xslashedleftarrow}

    \newcommand{\op}{{}^\tx{op}}
    \newcommand{\co}{{}^\tx{co}}

    \newcommand{\tp}[1]{\langle#1\rangle}
    \newcommand{\unit}{{\tp{}}}
    
    \newcommand{\rf}{\mathbin{\blacktriangleleft}}
    
    \newcommand{\adj}{\dashv}
    \newcommand{\radj}[1]{\mathrel{\adj_{#1}}}

    \newcommand{\ob}[1]{|#1|}

    \DeclareFontFamily{U}{min}{}
    \DeclareFontShape{U}{min}{m}{n}{<-> udmj30}{}
    \newcommand{\yo}{\!\text{\usefont{U}{min}{m}{n}\symbol{'210}}\!}

    \mathcommand{\comma}{\downarrow}

    \newcommand{\copi}{\flip\pi}
    
    \newsavebox{\whitecircstar}\sbox{\whitecircstar}{\kern.075em\tikz{\node[draw, circle,line width=.36pt, inner sep=0]{$*$};}\kern.075em}
    
    \newsavebox{\blackcircstar}\sbox{\blackcircstar}{\kern.075em\tikz{\node[fill, circle, line width=.36pt, inner sep=0, text=white]{$*$};}\kern.075em}

    \makeatletter
    \def\widebreve{\mathpalette\wide@breve}
    \def\wide@breve#1#2{\sbox\z@{$#1#2$}%
         \mathop{\vbox{\m@th\ialign{##\crcr
    \kern0.08em\brevefill#1{0.8\wd\z@}\crcr\noalign{\nointerlineskip}%
                        $\hss#1#2\hss$\crcr}}}\limits}
    \def\brevefill#1#2{$\m@th\sbox\tw@{$#1($}%
      \hss\resizebox{#2}{\wd\tw@}{\rotatebox[origin=c]{90}{\upshape(}}\hss$}
    \makeatother

    \NewDocumentCommand{\jrule}{om}{%
        \IfNoValueTF{#1}
            {\textsc{#2}}
            {$#1$-\textsc{#2}}%
    }

    \newcommand{\Set}{{\b{Set}}}
    \newcommand{\FinSet}{{\b{FinSet}}}
    \newcommand{\V}{{\bb V}} 
    \newcommand{\Cat}{\b{Cat}}
    
    \newcommand{\VCat}{{\V\h\Cat}}
    
    \newcommand{\Mnd}{\b{Mnd}}
    \newcommand{\RMnd}{\b{RMnd}}
    
    \newcommand{\Alg}{\b{Alg}}
    \newcommand{\Coalg}{\b{Coalg}}
    \newcommand{\Opalg}{\b{Opalg}}

    \newcommand{\ff}{fully faithful}
    
    \newcommand{\ffness}{full faithfulness}

    \newcommand{\ioo}{identity-on-objects}

    \newcommand{\EM}{Eilenberg--Moore}

    \newcommand{\eg}{e.g.\@\xspace}
    \newcommand{\ie}{i.e.\@\xspace}
    
    \newcommand{\cf}{cf.\@\xspace}
    
    \newcommand{\aka}{a.k.a.\@\xspace}
    \NewDocumentCommand{\etc}{t.}{etc.\@\xspace}
    \NewDocumentCommand{\ibid}{t.}{ibid.\@\xspace}
    \NewDocumentCommand{\loccit}{t.}{loc.\ cit.\@\xspace}

    \newcommand{\pdc}{pseudo double category}

    \newcommand{\vd}{virtual double}
    \newcommand{\vdc}{\vd{} category}
    \newcommand{\vdcs}{\vd{} categories}
    
    \newcommand{\ve}{virtual equipment}

\makeatletter

\define@key{beamerframe}{c}[true]{
    \beamer@frametopskip=0pt plus 1fill\relax%
    \beamer@framebottomskip=0pt plus 1fill\relax%
}

\patchcmd{\beamer@sectionintoc}{\vfill}{\vskip\itemsep}{}{}

\makeatother

\ifarticle{}{

  \colorlet{colour-bg}{black!85} 
  \definecolor{colour-primary}{HTML}{cc80ff} 
  \colorlet{colour-text}{black!10} 
  \colorlet{colour-subtle}{black!40} 
  \colorlet{colour-block-bg}{black!80} 
  \definecolor{colour-warning-bg}{HTML}{ffea80} 
  \definecolor{colour-warning-primary}{HTML}{e08152} 

  \hypersetup{linktoc=all,colorlinks, citecolor={colour-primary}, linkcolor={colour-primary}, urlcolor={colour-primary}} 
  \urlstyle{sf}

  \usepackage{changepage} 

  \usepackage{ragged2e} 

  \setbeamertemplate{section in toc}[sections numbered]

  \apptocmd{\frame}{}{\justifying}{}

  \usefonttheme[onlymath]{serif}

  \beamertemplatenavigationsymbolsempty
  \setbeamertemplate{footline}{
    \hfill%
    \usebeamercolor[fg]{page number in head/foot}%
    \usebeamerfont{page number in head/foot}%
    \setbeamertemplate{page number in head/foot}[pagenumber]%
    \usebeamertemplate*{page number in head/foot}\kern.2cm\vskip.2cm%
  }

  \setbeamertemplate{frametitle}[default][center]
  \addtobeamertemplate{frametitle}{\vspace*{1ex}}{\vspace*{1ex}}

  \setbeamertemplate{itemize item}{$\bullet$}

  \setbeamertemplate{theorems}[numbered]
  \newtheorem{proposition}[theorem]{\translate{Proposition}}
  \newtheorem{sketch}[theorem]{\translate{Proof (sketch)}}

  \addtobeamertemplate{block begin}
  {}
  {\vspace{0ex} 
  \begin{adjustwidth}{1em}{1em} 
  \tikzcdset{background color=colour-block-bg}
  }
  \addtobeamertemplate{block end}{\end{adjustwidth}%
  \vspace{1ex}} 
  {}
  \renewenvironment<>{block}[1]{%
      \begin{actionenv}#2%
        \def\insertblocktitle{\begin{adjustwidth}{.8em}{.8em}#1\end{adjustwidth}}%
        \par%
        \usebeamertemplate{block begin}}
      {\par%
        \usebeamertemplate{block end}%
      \end{actionenv}}

  \addtobeamertemplate{block example begin}
  {}
  {\vspace{0ex} 
  \begin{adjustwidth}{1em}{1em} 
  \tikzcdset{background color=colour-block-bg}
  }
  \addtobeamertemplate{block example end}{\end{adjustwidth}%
  \vspace{1ex}} 
  {}
  \renewenvironment<>{exampleblock}[1]{%
      \begin{actionenv}#2%
          \def\insertblocktitle{\begin{adjustwidth}{.8em}{.8em}#1\end{adjustwidth}}%
          \par%
          \only<presentation>{
            \setbeamercolor{local structure}{parent=example text}}%
          \usebeamertemplate{block example begin}}
        {\par%
          \usebeamertemplate{block example end}%
        \end{actionenv}}

    \setbeamercolor{background canvas}{bg=colour-bg}
    \tikzcdset{background color=colour-bg}

    \setbeamercolor{block title}{fg=colour-primary,bg=colour-block-bg}
    \setbeamercolor{block title example}{fg=colour-primary,bg=colour-block-bg}
    \setbeamercolor{block body}{bg=colour-block-bg}
    \setbeamercolor{block body example}{bg=colour-block-bg}

    \setbeamercolor{title}{fg=colour-primary}
    \setbeamercolor{frametitle}{fg=colour-primary}
    \setbeamercolor{alerted text}{fg=colour-primary}

    \setbeamercolor{normal text}{fg=colour-text}
    \setbeamercolor{item}{fg=colour-text}
    \setbeamercolor{section in toc}{fg=colour-text}

    \setbeamercolor{page number in head/foot}{fg=colour-subtle}

    \setbeamercolor*{bibliography entry author}{fg=colour-text}
    \setbeamercolor*{bibliography entry note}{fg=colour-text}

}

%% file: core-ftrm.tex
\IfFileExists{tangle.sty}{\usepackage{tangle}}{}
\usepackage{bbm}
\usepackage{wasysym}
\usepackage{suffix}

\newcommand{\X}{\bb X}
\newcommand{\tX}{\u\X}
\renewcommand{\Cat}{\dc{Cat}}

\newcommand{\I}{\s I}

\newcommand{\lh}[1]{\llbracket #1 \rrbracket}
\mathcommand{\L}{\textup{\sagittarius}}
\newcommand{\LMnd}{\L\Mnd}

\newcommand{\jAE}{j \colon A \to E}

\newcommand{\jadj}{\radj j}

\newcommand{\cp}[1]{\mathbin{\smallsmile_{#1}}}
\newcommand{\pc}[1]{\mathbin{\smallfrown_{#1}}}

\newcommand{\Kl}{\b{Kl}}

\newcommand{\aop}{{\rtimes}}

\newcommand{\oop}{{\ltimes}}

\newcommand{\colim}{\operatorname{colim}}

\renewcommand{\lim}{\operatorname{lim}}

\renewcommand{\defeq}{\mathrel{:=}}
\newcommand{\opcart}{\tx{opcart}}

\newcommand{\cart}{\tx{cart}}

\renewcommand{\EM}{\relax\ifmmode\b{EM}\else{}Eilenberg\nobreakdash--Moore\fi}

\mathcommand{\P}{\cl P}

\makeatletter
\@ifpackageloaded{tangle}{
\tgColours{F5A3A3,F5CCA3,F5F5A3,CCF5A3,A3F5A3,A3F5CC,A3F5F5,A3CCF5,A3A3F5,CCA3F5,F5A3F5,F5A3CC}

}{}
\makeatother

\newcounter{eqstep}
\newcounter{eqsubstep}[eqstep]
\newcommand{\nexteqstep}{\stepcounter{eqstep}}

\makeatletter
\newcommand{\raisedtarget}[1]{\Hy@raisedlink{\hypertarget{#1}{}}}
\makeatother

\newcommand{\tangleeq}{\refstepcounter{eqsubstep}\raisedtarget{eqstep:\theeqstep.\theeqsubstep}{\mathclap{\overset{(\theeqstep.\theeqsubstep)}{=}}}}

\WithSuffix\newcommand\tangleeq*{\mathclap{=}}

\newcommand{\tangleeql}{\refstepcounter{eqsubstep}\raisedtarget{eqstep:\theeqstep.\theeqsubstep}{\mathllap{\overset{(\theeqstep.\theeqsubstep)}{=}}}}

\WithSuffix\newcommand\tangleeql*{\mathllap{=}}

\newenvironment{tangleeqs*}{%
    \mathcommand{\=}{\tangleeq*}
    \global\let\externaldblbackslash\\
    \csname gather*\endcsname
    \ifundef{\internaldblbackslash}{%
        \global\let\internaldblbackslash\\%
        \gdef\\{\internaldblbackslash\tangleeql*}%
    }{}
}{%
    \csname endgather*\endcsname
    \global\undef\internaldblbackslash
    \global\let\\\externaldblbackslash
}